\documentclass[11pt,a4paper,reqno]{amsart}
\usepackage{amsfonts,amssymb,amsmath,amsthm,graphicx,color,amscd,xspace,verbatim,dsfont}
\usepackage{scalerel}
\makeatletter

\makeatletter
\def\@tocline#1#2#3#4#5#6#7{\relax
  \ifnum #1>\c@tocdepth 
  \else
    \par \addpenalty\@secpenalty\addvspace{#2}%
    \begingroup \hyphenpenalty\@M
    \@ifempty{#4}{%
      \@tempdima\csname r@tocindent\number#1\endcsname\relax
    }{%
      \@tempdima#4\relax
    }%
    \parindent\z@ \leftskip#3\relax \advance\leftskip\@tempdima\relax
    \rightskip\@pnumwidth plus4em \parfillskip-\@pnumwidth
    #5\leavevmode\hskip-\@tempdima
      \ifcase #1
       \or\or \hskip 1em \or \hskip 2em \else \hskip 3em \fi%
      #6\nobreak\relax
      \dotfill
      \hbox to\@pnumwidth{\@tocpagenum{#7}}
    \par
    \nobreak
    \endgroup
  \fi}
\makeatother

\newtheorem{theorem}{Theorem}[section]
\newtheorem{lemma}[theorem]{Lemma}
\newtheorem{proposition}[theorem]{Proposition}

\theoremstyle{definition}
\newtheorem{definition}[theorem]{Definition}

\newtheorem{remark}[theorem]{Remark}

\newcommand{\N}{{\mathbb N}}
\newcommand{\R}{{\mathbb R}}

\newcommand{\BB}{\mathbb{B}}

\newcommand{\tr}{\mathrm{tr}^*}

\newcommand{\beqn}{\begin{eqnarray}}
\newcommand{\eeqn}{\end{eqnarray}}   
\newcommand{\beq}{\begin{eqnarray*}}
\newcommand{\eeq}{\end{eqnarray*}}

\newcommand{\be}{\small\begin{equation}}
\newcommand{\bel}[1]{\small\begin{equation}\label{#1}}
\newcommand{\ee}{\end{equation}\normalsize}

\newcommand{\BA}{\begin{array}}
\newcommand{\EA}{\end{array}}
\newcommand{\BAN}{\renewcommand{\arraystretch}{1.2}
\setlength{\arraycolsep}{2pt}\begin{array}}
\newcommand{\BAV}[2]{\renewcommand{\arraystretch}{#1}
\setlength{\arraycolsep}{#2}\begin{array}}

\newcommand{\BSA}{\begin{subarray}}
\newcommand{\ESA}{\end{subarray}}
\newcommand{\BAL}{\begin{aligned}}
\newcommand{\EAL}{\end{aligned}}

\newcommand{\forevery}{\quad \forall}

\newcommand{\norm}[1]{\left \|#1\right \|}

\newcommand{\supp}{\mathrm{supp}\,}
\newcommand{\dist}{\mathrm{dist}\,}
\newcommand{\sign}{\mathrm{sign}}
\newcommand{\diam}{\mathrm{diam}\,}

\newcommand{\q}{\quad}

\newcommand{\prt}{\partial}

\newcommand{\tl}{\tilde}

\newcommand{\sbs}{\subset}

\newcommand\1{{\ensuremath {\mathds 1} }}

\def\dist{\mathrm{dist}}

\def\ga{\alpha}            
             
\def\gth{\theta}                         
\def\gf{\phi}       \def\vgf{\varphi}    
            \def\gl{\lambda}
\def\gm{\mu}        \def\gn{\nu}         
            
\def\gs{\sigma}       \def\gt{\tau}
      \def\gw{\omega}

\def\Gl{\Lambda}          
\def\Gw{\Omega}              

\def\CS{{\mathcal S}}      \def\CN{{\mathcal N}}
   \def\CO{{\mathcal O}}   
\def\CA{{\mathcal A}}   \def\CB{{\mathcal B}}   \def\CC{{\mathcal C}}
\def\CD{{\mathcal D}}      \def\CF{{\mathcal F}}
   \def\CH{{\mathcal H}}   
      \def\CL{{\mathcal L}}

   \def\BBB {\mathbb B}    
       
\def\BBG {\mathbb G}   \def\BBH {\mathbb H}    
\def\BBJ {\mathbb J}   \def\BBK {\mathbb K}    
   \def\BBN {\mathbb N}    
   \def\BBR {\mathbb R}    \def\BBS {\mathbb S}

   \def\GTB {\mathfrak B}    \def\GTC {\mathfrak C}

\def\GTM {\mathfrak M}

\def\tr{\mathrm{tr}}

\def \dd {\,\mathrm{d}}
\def \dx {\mathrm{d}x}
\def \dy {\mathrm{d}y}
\def \dz {\mathrm{d}z}

\def \dtau {\mathrm{d}\tau}
\def \p {\mathbf{p}}
\def \q {\mathbf{q}}

\newcommand{\ei}{{\phi_{\xm }}}
\newcommand{\xa}{\alpha}
\newcommand{\xb}{\beta}
\newcommand{\xg}{\gamma}
\newcommand{\xG}{\Gamma}
\newcommand{\xd}{\delta}

\newcommand{\xe}{\varepsilon}

\newcommand{\xk}{\kappa}
\newcommand{\xl}{\lambda}

\newcommand{\xm}{\mu}
\newcommand{\xn}{\nu}

\newcommand{\xr}{\rho}
\newcommand{\xs}{\sigma}
\newcommand{\xS}{\Sigma}
\newcommand{\xf}{\phi}

\newcommand{\xo}{\omega}
\newcommand{\xO}{\Omega}

\newcommand{\myint}[2]{{\displaystyle \int_{#1}^{#2}}}

\def\Nthb{\BBN_{\xa}}

\newcommand{\ap}{{\xa_{\scaleto{+}{3pt}}}}
\newcommand{\am}{{\xa_{\scaleto{-}{3pt}}}}

\def\bal#1\eal{\small\begin{align*}#1\end{align*}\normalsize}
\def\ba#1\ea{\small\begin{align}#1\end{align}\normalsize}
\numberwithin{equation}{section}

\begin{document}

\title[Semilinear elliptic Schr\"odinger equations]{Semilinear elliptic Schr\"odinger equations involving singular potentials and source terms}
\author{Konstantinos T. Gkikas}
\address{Konstantinos T. Gkikas, Department of Mathematics, National and Kapodistrian University of Athens, 15784 Athens, Greece}
\email{kugkikas@math.uoa.gr}

\author[P.T. Nguyen]{Phuoc-Tai Nguyen}
\address{Phuoc-Tai Nguyen, Department of Mathematics and Statistics, Masaryk University, Brno, Czech Republic}
\email{ptnguyen@math.muni.cz}

\date{\today}

\begin{abstract}
Let $\Omega \subset \mathbb{R}^N$ ($N>2$) be a $C^2$ bounded domain and  $\Sigma \subset \Omega$ be a compact, $C^2$ submanifold without boundary, of dimension $k$ with $0\leq k < N-2$. Put $L_\mu = \Delta + \mu d_\Sigma^{-2}$ in $\Omega \setminus \Sigma$, where $d_\Sigma(x) = \mathrm{dist}(x,\Sigma)$ and $\mu$ is a parameter. We study the boundary value problem (P) $-L_\mu u = g(u) + \tau$ in $\Omega \setminus \Sigma$ with condition $u=\nu$ on $\partial \Omega \cup \Sigma$, where $g: \mathbb{R} \to \mathbb{R}$ is a nondecreasing, continuous function and $\tau$ and $\nu$ are positive measures. The interplay between the inverse-square potential $d_\Sigma^{-2}$, the nature of the source term $g(u)$ and the measure data $\tau,\nu$ yields substantial difficulties in the research of the problem. We perform a deep analysis based on delicate estimate on the Green kernel and Martin kernel and fine topologies induced by appropriate capacities to establish various necessary and sufficient conditions for the existence of a solution in different cases.
	
\medskip
	
\noindent\textit{Key words: Hardy potentials, critical exponents, source terms, capacities, measure data}
	
\medskip
	
\noindent\textit{Mathematics Subject Classification: 35J10, 35J25, 35J61, 35J75}
	
\end{abstract}

\maketitle
\tableofcontents
\section{Introduction}
\subsection{Motivation and aim}
The research of Schr\"odinger equations is a hot topic in the area of partial differential equations because of its applications in encoding physical properties of quantum systems. In the literature, a large number of publications have been devoted to the investigation of stationary Schr\"odinger equations involving the Laplacian with a singular potential. The presence of the singular potential yields distinctive features of the research and leads to disclose new phenomena.

The borderline case where the potential is the inverse-square of the distance to a submanifold of the domain under consideration is of interest since in this case the potential admits the same scaling (of degree $-2$) as the Laplacian and hence cannot be treated simply by standard perturbation methods. Several works have been carried out to investigate the effect of such a potential in various aspects, including a recent study on linear equations.

The present paper originated in attempts to set a step forward in the study of elliptic nonlinear Sch\"rodinger equations involving an inverse-square potential and a source term in measure frameworks.

\subsection{Background and main results}
Let $\Omega \subset \R^N$ be a $C^2$ bounded domain and $\Sigma\subset\xO$ be a compact, $C^2$ submanifold in $\R^N$ without boundary, of dimension $k$ with $0 \leq k < N-2$. Put
\bel{distance} d(x):=\dist(x,\partial\xO) \quad \text{and} \quad  d_\Sigma(x): = \dist(x,\Sigma).
\ee
For $\mu \in \R$, denote by $L_\mu$ the Schr\"odinger operator with the inverse-square potential $d_\Sigma^{-2}$ as
\bal L_\mu = L_\mu^{\Omega,\Sigma}:=\Delta + \frac{\mu}{d_\Sigma^2}
\eal
in $\Omega \setminus \Sigma$. The study of $L_\mu$ was carried out in \cite{GkiNg_linear} in which the optimal Hardy constant
\bal{\mathcal C}_{\Omega,\Sigma}:=\inf_{\varphi \in H^1_0(\Omega)}\frac{\int_\Omega |\nabla \varphi|^2\dx}{\int_\Omega d_\Sigma^{-2}\varphi^2 \dx}
\eal
is deeply involved. It is well known that ${\mathcal C}_{\xO,\Sigma}\in (0,H^2]$ (see D\'avila and Dupaigne \cite{DD1, DD2} and Barbatis, Filippas and Tertikas \cite{BFT}), where
\be \label{valueH}
H:=\frac{N-k-2}{2}.
\ee
It is classical that ${\mathcal C}_{\Omega,\{0\}}=\left(\frac{N-2}{2} \right)^2$. We also know that ${\mathcal C}_{\Omega,\Sigma}=H^2$ if $-\Delta d_\Sigma^{2+k-N} \geq 0$ in the sense of distributions in $\Omega \setminus \Sigma$ or if $\Omega=\Sigma_\beta$ with $\beta$  small enough (see \cite{BFT}), where
\bal
\Sigma_\beta :=\{ x \in \R^N \setminus \Sigma: d_\Sigma(x) < \beta \}.
\eal
For $\mu \leq H^2$, let $\am$ and $\ap$ be the roots of the algebraic equation $\ga^2 - 2H\ga + \mu=0$, i.e.
\bel{apm}
\am:=H-\sqrt{H^2-\mu}, \quad \ap:=H+\sqrt{H^2-\mu}.
\ee
Notice that $\am\leq H\leq\ap<2H$, and $\am \geq 0$ if and only if $\mu \geq 0$. Moreover, by  \cite[Lemma 2.4 and Theorem 2.6]{DD1} and \cite[page 337, Lemma 7, Theorem 5]{DD2},
\bal \lambda_\mu:=\inf\left\{\int_{\Gw}\left(|\nabla u|^2-\frac{\xm }{d_\Sigma^2}u^2\right)dx: u \in C_c^1(\Omega), \int_{\Gw} u^2 dx=1\right\}>-\infty.
\eal
We note that $\lambda_\mu$ is the first eigenvalue associated to $-L_\mu$ and its corresponding eigenfunction $\phi_\mu$, with normalization $\| \phi_\mu \|_{L^2(\Omega)}=1$, satisfies two-sided estimate $\phi_\mu \approx d\,d_\Sigma^{-\am}$ in $\Omega \setminus \Sigma$ (see subsection \ref{subsect:eigen} for more detail).

The sign of $\lambda_\mu$ plays an important role in the study of $L_\mu$. If $\mu<\CC_{\Omega,\Sigma}$ then $\lambda_\mu>0$; however, in general, this does not hold.  Under the assumption $\lambda_\mu>0$, the authors of the present paper obtained the existence and sharp two-sided estimates of the Green function $G_\mu$ and Martin kernel $K_\mu$ associated to $-L_\mu$ (see \cite{GkiNg_linear}). These are crucial tools in the study of the boundary value problem with measure data for linear equations of the form
\ba \label{eq:linear} \left\{ \begin{aligned}
-L_\mu u &= \tau \quad &&\text{in } \Omega \setminus \Sigma, \\
 \tr(u) &= \nu, &&
\end{aligned} \right. \ea
where $\tau \in \GTM(\Omega;\phi_\mu)$ (i.e. $\|  \tau\|_{\GTM(\Omega \setminus \Sigma;\ei)}:=\int_{\Omega \setminus \Sigma}\phi_\mu \dd |\tau|<\infty$) and $\nu \in \GTM(\partial \Omega \cup \Sigma)$ (i.e. $\| \nu \|_{\GTM(\partial \Omega \cup \Sigma)}:= \int_{\partial \Omega \cup \Sigma}\dd |\nu| < \infty$).

In \eqref{eq:linear}, $\tr(u)$ denotes the \textit{boundary trace} which was introduced in \cite{GkiNg_linear} in terms of harmonic measures of $-L_\mu$  (see Subsection \ref{subsec:boundarytrace}). An important feature of this notion is $\tr(\BBG_\mu[\tau]) = 0$ for any $\tau \in \GTM(\Omega \setminus \Sigma;\phi_\mu)$ and  $\tr(\BBK_\mu[\tau]) = \nu$ for any $\nu \in \GTM(\partial \Omega \cup \Sigma)$, where
\bal
\BBG_\mu[\tau](x): &= \int_{\Omega \setminus \Sigma}G_\mu(x,y)\dd\tau(y), \quad \tau \in \GTM(\Omega \setminus \Sigma;\phi_\mu), \\
\BBK_\mu[\nu](x): &= \int_{\partial \Omega \cup \Sigma}K_\mu(x,y)\dd\nu(y), \quad \nu \in \GTM(\partial \Omega \cup \Sigma).
\eal
Note that for a positive measure $\tau$, $\BBG_\mu[\tau]$ is finite a.e. in $\Omega \setminus \Sigma$ if and only if $\tau \in \GTM(\Omega \setminus \Sigma; \ei)$.

Moreover, it was shown in \cite{GkiNg_linear} that $\BBG_\mu[\tau]$ is the unique solution of \eqref{eq:linear} with $\nu=0$, and $\BBK_\mu[\nu]$ is the unique solution of \eqref{eq:linear} with $\tau=0$. By the linearity, the unique solution to \eqref{eq:linear} is of the form
\bal
u = \BBG_\mu[\tau] + \BBK_\mu[\nu] \quad \text{a.e. in } \Omega \setminus \Sigma.
\eal
Further results for linear problem \eqref{eq:linear} are discussed in Subsection \ref{subsec:linear}.

As a continuation and development of the work \cite{GkiNg_linear} in this research topic, this paper studies the boundary value problem for semilinear equations with a source term of the form
\ba \label{NLP} \left\{ \begin{aligned}
	-L_\mu u &= g(u) + \rho \tau \quad &&\text{in } \Omega \setminus \Sigma, \\
	\tr(u) &= \sigma\nu, &&
\end{aligned} \right. \ea
where $\rho, \sigma$ are nonnegative parameters, $\tau$ and $\nu$ are Radon measures on  $\Omega \setminus \Sigma$ and $\partial \Omega \cup \Sigma$ respectively, and $g: \R \to \R$ is a nondecreasing continuous function such that $g(0)=0$.

Various works on problem \eqref{NLP} and related problems have been published in the literature, including excellent papers of D\'avila and Dupaigne \cite{DN,DD1,DD2} where important tools in function settings are established and combined with a monotonicity argument in derivation of existence, nonexistence, uniqueness for solutions with zero boundary datum. Afterwards, deep nonexistence results for nonnegative distributional supersolutions have been obtained by Fall \cite{Fall-0} via a linearization argument. Recently, a description on isolated singularities in case $\Sigma=\{0\}\subset \Omega$ has been provided by Chen and Zhou \cite{CheZho}.

In the present paper, the interplay between dimention of the set $\Sigma$, the value of $\mu$, the growth of the source term and the concentration of measure data causes the invalidity or quite restrictive applicability of the techniques used in the mentioned papers and leads to the involvement of several \textit{critical exponents} for the solvability of problem \eqref{NLP}. Therefore, our aim is to perform further analysis and to establish effective tools, which allow us to obtain existence and nonexistence results for \eqref{NLP} in various cases.

\textit{Let us assume throughout the paper that}
\be \label{assump1}
\mu \leq H^2 \quad \text{and} \quad \lambda_\mu > 0.
\ee
Assumption \eqref{assump1} ensures the validity of sharp two-sided estimates for the Green kernel and Martin kernel as well as other results regarding linear equations as mentioned above.

In order to state our main results, we introduce some notations. For $\alpha , \gamma \in \R$, put
\ba \label{varphi}
\varphi_{\alpha,\gamma}(x):=d_\Sigma(x)^{-\alpha}d(x)^\gamma, \quad x \in \Omega \setminus \Sigma.
\ea
It can be seen from \eqref{eigenfunctionestimates} that $\varphi_{\am,1} \approx \phi_\mu$ (we notice that $\am$ is defined in \eqref{apm}). Let $\GTM(\Omega \setminus \Sigma;\varphi_{\am,\gamma})$ be the space of measures $\tau$ such that
\bal \| \tau \|_{\GTM(\Omega \setminus \Sigma;\varphi_{\am,\gamma})}:= \int_{\Omega \setminus \Sigma}\varphi_{\am,\gamma}\,\dd |\tau|<\infty.
\eal

The notion of the weak solutions of \eqref{NLP} is given below.
\begin{definition} \label{weaksol-LP}
	Let $\gamma \in [0,1]$, $\rho \geq 0$, $\sigma \geq 0$,  $\tau\in\mathfrak{M}(\xO\setminus \Sigma;\varphi_{\am,\gamma})$ and $\nu \in \mathfrak{M}(\partial\xO\cup \Sigma)$. We say that $u$  is a \textit{weak solution} of \eqref{NLP} if $u\in L^1(\Omega;\ei)$, $g(u) \in L^1(\Omega;\phi_\mu)$ and
	\be \label{lweakform}
	- \int_{\Omega} u L_{\xm }\zeta \, \dx=\int_{\Omega} g(u)\zeta \, \dx + \rho \int_{\Omega \setminus \Sigma}\zeta \,\dtau- \sigma\int_{\Omega} \mathbb{K}_{\xm}[\xn]L_{\xm }\zeta \, \dx
	\qquad\forall \zeta \in\mathbf{X}_\xm(\xO\setminus \Sigma),
	\ee
	where the \textit{space of test function} ${\bf X}_\mu(\Gw\setminus \Sigma)$ is defined by
	\ba \label{Xmu} {\bf X}_\mu(\Gw\setminus \Sigma):=\{ \zeta \in H_{loc}^1(\Omega \setminus \Sigma): \phi_\mu^{-1} \zeta \in H^1(\Gw;\phi_\mu^{2}), \, \phi_\mu^{-1}L_\mu \zeta \in L^\infty(\Omega)  \}.
	\ea
\end{definition}

The space ${\bf X}_\mu(\Omega \setminus \Sigma)$ was introduced in \cite{GkiNg_linear} to study linear problem \eqref{eq:linear}. From \eqref{Xmu}, it is easy to see that the term on the left-hand side of \eqref{lweakform} is finite. By \cite[Lemma 7.3]{GkiNg_linear} and \eqref{eigenfunctionestimates}, for any $\zeta \in {\bf X}_\mu(\Omega \setminus \Sigma)$, $|\zeta| \lesssim \phi_\mu \approx d\, d_\Sigma^{-\am}$, hence the first term on the right-hand side of \eqref{lweakform} is finite. Moreover, for any $\zeta \in {\bf X}_\mu(\Omega \setminus \Sigma)$ and $\gamma \in [0,1]$, we have $|\zeta| \lesssim d^\gamma d_\Sigma^{-\am} = \varphi_{\am,\gamma}$. This implies that the second term on the right-hand side of \eqref{lweakform} is finite. Finally, since $\BBK_{\mu}[\nu] \in L^1(\Omega;\ei)$, the third term on the right-hand side of \eqref{lweakform} is also finite. 	

By Theorem \ref{linear-problem}, $u$ is a weak solution of \eqref{NLP} if and only if
\bal
u = \BBG_\mu[g(u)] + \BBG_\mu[\tau] + \BBK_\mu[\nu] \quad \text{in } \Omega \setminus \Sigma.
\eal


Our main results disclose different scenarios, depending on the interplay between the concentration and the total variation of the measure data, and the size of the set $\Sigma$, in which the existence of a solution to \eqref{NLP} can be derived.
In the following theorem, we show the existence, as well as weak Lebesgue estimates, of a solution to \eqref{NLP} provided that the nonlinearity $g$ has mild growth and the measure data have small norm.

\begin{theorem} \label{th1} Let $0<\xm\leq H^2$, $0\leq\xg\leq1$, $\tau \in \GTM(\Omega \setminus \Sigma; \varphi_{\am,\gamma})$ with
$\norm{\tau}_{\mathfrak{M}(\Omega\setminus \Sigma;\varphi_{\am,\gamma})}=1$
and $\nu \in \GTM(\prt \Gw\cup \xS)$ with
$\norm{\nu}_{\GTM(\prt \Gw\cup \xS)}=1$.
Assume $g$ satisfies
\ba \label{subcd-0} \Gl_g:=\int_1^\infty  s^{-q-1}  (g(s)-g(-s))\,\dd s < \infty
\ea	
 for some $q \in (1,\infty)$ and
 \ba \label{gcomparepower}
 |g(s)|\leq a|s|^{\tilde q} \quad \text{for some } a>0,\; \tilde q>1 \text{ and for any } |s|\leq 1.
 \ea

Assume one of the following conditions holds.

$(i)$ $\1_{\partial \Omega} \, \nu \equiv 0$ and \eqref{subcd-0} holds for $q = \frac{N+\gamma}{N+\gamma-2}$.

$(ii)$ $\1_{\partial \Omega} \, \nu\not\equiv 0$  and \eqref{subcd-0} holds for $q = \frac{N+1}{N-1}$.

 Then there exist positive numbers $\xr_0,\xs_0,t_0$ depending on $N,\mu,\Gw,\Sigma,\Gl_g,\gamma, \tilde q$ such that, for every $\xr \in (0,\xr_0)$ and $\xs\in (0,\xs_0)$, problem \eqref{NLP} admits a weak solution $u$ satisfying
	\bel{est:t0} \|u\|_{L_w^{q}(\Gw\setminus\xS;\ei)} \leq t_0,
	\ee
	where $q=\frac{N+\gamma}{N+\gamma-2}$ if case $(i)$ happens or $q=\frac{N+1}{N-1}$ if case $(ii)$ happens.
\end{theorem}

The proof of Theorem \ref{th1} contains several steps, relying on various ingredients such as sharp weak Lebesgue estimates on Green kernel and Martin kernel (see Theorems \ref{lpweakgreen}--\ref{lpweakmartin1}) and Schauder fixed point theorem.

When $\tau$ or $\nu$ is zero measure and $\mu$ is not restricted to be positive, the value of $q$ in \eqref{subcd-0} can be enlarged or adjusted, as shown in the following theorem.

\begin{theorem} \label{th2} Let $\xm\leq H^2$, $0\leq\xg\leq1$ and $g$ satisfy \eqref{gcomparepower}.
	
$(i)$ Assume $0< \mu \leq \left(\frac{N-2}{2}\right)^2$, $\nu=0$, $\tau \in \GTM(\Gw \setminus \xS;\varphi_{\am,\gamma})$ with $\norm{\tau}_{\mathfrak{M}(\Omega\setminus \Sigma;\varphi_{\am,\gamma})}=1$, and \eqref{subcd-0} holds with  $q=\frac{N+\gamma}{N+\gamma-2}$. Then the conclusion of Theorem \ref{th1} holds true with  $q=\frac{N+\gamma}{N+\gamma-2}$.

$(ii)$ Assume $\mu \leq 0$, $0 \leq \kappa \leq -\am$,  $\nu=0$, $\tau \in \GTM(\Gw \setminus \xS;\varphi_{-\kappa,\gamma})$ with $\norm{\tau}_{\mathfrak{M}(\Omega\setminus \Sigma;\varphi_{-\kappa,\gamma})}=1$, and $g$ satisfy \eqref{subcd-0} with
\be \label{pkd}
q=\min\left\{\frac{N+\xg}{N+\xg-2},\frac{N+\kappa}{N+\kappa-2}\right\}.
\ee
Then the conclusion of Theorem \ref{th1} holds true with $q$ as in \eqref{pkd}.

$(iii)$ Assume $\mu \leq \left(\frac{N-2}{2}\right)^2$, $\tau=0$, $\nu \in \GTM(\prt \Gw\cup \xS)$ has compact support in $\xS$ with $\norm{\nu}_{\GTM(\prt \Gw\cup \xS)}=1$, and \eqref{subcd-0} holds with \be \label{pkd-a}
q=\min\left\{\frac{N+1}{N-1},\frac{N-\am}{N-\am-2}\right\}.
\ee
Then the conclusion of Theorem \ref{th1} holds true with $q$ as in \eqref{pkd-a}.

$(iv)$ Assume $\mu \leq \left(\frac{N-2}{2}\right)^2$, $\tau=0$, $\nu \in \GTM(\prt \Gw\cup \xS)$ has compact support in $\partial \Omega$ with $\norm{\nu}_{\GTM(\prt \Gw\cup \xS)}=1$, and \eqref{subcd-0} holds with $q=\frac{N+1}{N-1}$. Then the conclusion of Theorem \ref{th1} holds true with $q=\frac{N+1}{N-1}$.

\end{theorem}

We remark that condition \eqref{pkd-a} is not sharp. When $g$ is a pure power function, condition \eqref{pkd-a} can be improved to be sharp, as pointed out in the remark following Theorem \ref{subm}.

When $g$ is a power function, namely $g(u)=|u|^{p-1}u$ for $p>1$, problem \eqref{NLP} becomes
\be \label{power-source} \left\{ \BAL
- L_\gm u&=|u|^{p-1}u + \rho\tau\qquad \text{in }\;\Gw\setminus \Sigma,\\
\tr(u)&= \sigma\xn.
\EAL \right.
\ee
We will point out below that the exponents $\frac{N+\gamma}{N+\gamma-2}$, $\frac{N-\am}{N-\am-2}$ and $\frac{N+1}{N-1}$ are \textit{critical exponents} for the existence of a solution to \eqref{power-source}. Moreover, by performing further analysis, we are able to provide necessary and sufficient conditions in terms of estimates of the Green kernel and Martin kernel, as well as in terms of appropriate capacities.

We first consider \eqref{power-source} with $\sigma \nu=0$. Let us introduce suitable capacities.
For $\xa\leq N-2$, set
\bal
\CN_{\ga}(x,y):=\frac{\max\{|x-y|,d_\xS(x),d_\xS(y)\}^\xa}{|x-y|^{N-2}\max\{|x-y|,d(x),d(y)\}^2},\quad (x,y)\in\overline{\xO}\times\overline{\xO}, x \neq y,
\eal
and
\bal \Nthb[\omega](x):=\int_{\overline{\Gw}} \CN_{\xa}(x,y) \dd\omega(y), \quad \omega \in \GTM^+(\overline \Gw).\eal

For $\alpha \leq N-2$, $b>0$, $\theta>-N+k$ and $s>1$, define capacity $\text{Cap}_{\Nthb,s}^{b,\theta}$ by
\bal \text{Cap}_{\Nthb,s}^{b,\theta}(E) :=\inf\left\{\int_{\overline{\xO}}d^b d^\theta_\xS\gf^s\,\dx:\;\; \gf \geq 0, \;\;\Nthb[ d^b d_{\Sigma}^\theta\gf ]\geq\1_E\right\} \quad \text{for Borel set } E\subset\overline{\xO}.
\eal
Here $\1_E$ denotes the indicator function of $E$. By \cite[Theorem 2.5.1]{Ad}, we have
\bal
(\text{Cap}_{\Nthb,s}^{b,\theta}(E))^\frac{1}{s}=\sup\{\omega(E):\omega \in\GTM^+(E), \|\Nthb[\omega]\|_{L^{s'}(\xO;d^b d^\theta_\xS)} \leq 1 \}.
\eal

\begin{theorem}\label{theoremint}
	We assume that $\mu< \left(\frac{N-2}{2}\right)^2$ and
	\ba \label{p-cond}
	1<p<\frac{2+\am}{\am} \text{ if } \mu>0 \quad \text{ or } \quad p>1 \text{ if } \mu \leq0.
	\ea
	Let $\tau \in \GTM^+(\Omega \setminus \Sigma; \ei)$.  Then the following statements are equivalent.
	
	1. The equation
	\ba \label{u-rhotau} u=\BBG_\xm[u^p]+\rho\BBG_\xm[\gt]
	\ea
	has a positive solution for $\rho>0$ small.
	
	2. For any Borel set $E \subset \xO\setminus \Sigma$, there holds
	\bal\int_E \BBG_\xm[\1_E \tau]^p \ei \, \dx \leq C\int_E\ei \dd \tau.\eal
	
	3. The following inequality holds
	\bal
\BBG_\xm[\BBG_\xm[\gt]^p]\leq C\, \BBG_\xm[\gt]<\infty\quad \text{a.e. in } \Omega \setminus \Sigma.
\eal

4. For any Borel set $E \subset \xO\setminus \Sigma$, there holds
\bal\int_E\ei \dd \tau \leq C\, \mathrm{Cap}_{\BBN_{2\am},p'}^{p+1,-\am(p+1)}(E).\eal
\end{theorem}

 It is worth mentioning that when $\mu< \left(\frac{N-2}{2}\right)^2$, if $1< p < \frac{N+1}{N-1}$ then  all the statements 1--4 of Theorem \ref{theoremint} hold true (see Remark \ref{capp-1}), while if $p \geq \frac{N+1}{N-1}$ then, for any $\rho>0$, \textit{there exists} $\tau \in \GTM^+(\Omega \setminus \Sigma;\ei)$ such that equation \eqref{u-rhotau} does not admit any positive solution (see Proposition \ref{nonexist-1}). Furthermore, when $0<\mu<\left(\frac{N-2}{2}\right)^2$ and $p\geq \frac{2+\am}{\am}$, \textit{for any} $\rho>0$ and \textit{any} $\tau \in \GTM^+(\Omega \setminus \Sigma; \ei)$, equation \eqref{u-rhotau} has no solution  (see Proposition \ref{remove1}).

We note that when $\Sigma=\{0\}$ and $\mu=\left(\frac{N-2}{2}\right)^2$, Theorem \ref {theoremint} remains valid under the assumption that $\tau \in \GTM^+(\Omega \setminus \{0\}; \phi_{\mu})$ with compact support in $\Omega \setminus \{0\}$. This is shown in Theorem \ref{theeqint}.	

Next we investigate \eqref{power-source} with $\tau=0$. To this end, we make use of  a different type of capacities whose definition is introduced in \eqref{Capsub}. These capacities are denoted by $\mathrm{Cap}_{\gth,s}^{\Gamma}$, where $\Gamma=\partial \Omega$ or $\Gamma=\Sigma$, which allow us to measure Borel subsets of $\partial \Omega \cup \Sigma$ in a subtle way.

\begin{theorem}\label{subm}
	Assume that $\mu< \left(\frac{N-2}{2}\right)^2$ and condition \eqref{p-cond} holds.
	 Let $\xn \in \GTM^+(\partial\xO\cup \xS)$ with compact support in $\xS$.
	Then the following statements are equivalent.
	
	1. The equation
	\ba \label{u-sigmanu} u=\BBG_\xm[u^p]+\gs\BBK_\xm[\xn]
	\ea
	has a positive solution for $\gs>0$ small.
	
	2. For any Borel set $E \subset \partial \Omega \cup \Sigma$, there holds
	\ba \label{Kp<nu} \int_E \BBK_\xm[\1_E\xn]^p \ei \dx \leq C\,\xn(E).
	\ea
	
	3. The following inequality holds
	\ba \label{GKp<K} \BBG_\xm[\BBK_\xm[\xn]^p]\leq C\,\BBK_\xm[\xn]<\infty\quad \text{a.e. in } \Omega.
	\ea
	
	Assume, in addition, that
	\be \label{p-cond-2}
	k \geq 1 \quad \text{and} \quad \max\left\{1,\frac{N-k-\am}{N-2-\am} \right\}< p<\frac{2+\ap}{\ap}.
	\ee
	Put
	\ba \label{gamma} \vartheta: = \frac{2-(p-1)\ap}{p}.
	\ea
	Then any of the above statements is equivalent to the following statement.
	
	4. For any Borel set $E \subset  \Sigma$, there holds
	\bal\xn(E)\leq C\, \mathrm{Cap}_{\vartheta,p'}^{\xS}(E).\eal
\end{theorem}

 We remark that  when $1< p < \frac{N-\am}{N-2-\am}$, all statements of Theorem \ref{subm} hold (see Remark \ref{existSigma}), while when $p \geq \frac{N-\am}{N-2-\am}$, for any $z \in \Sigma$ and any $\sigma>0$, problem \eqref{u-sigmanu} with $\nu=\delta_z$ does not admit any positive weak solution (see Remark \ref{nonexistSigma}). Assumption \eqref{p-cond-2} is imposed to ensure the validity of delicate estimates related to the Martin kernel (see \cite[Lemma 8.1]{GkiNg_absorption}), which enables us to deal with capacity $\mathrm{Cap}_{\vartheta,p'}^\Sigma$.

We also note that in case $\xS=\{0\}$ and $\mu=\left(\frac{N-2}{2}\right)^2$, if $p< \frac{2+\ap}{\ap}$ then for $\sigma>0$ small, there is a solution of \eqref{u-sigmanu} with $\tau=0$ and $\xn=\xd_0$ (see Remark \ref{rem2}). On the contrary, when $p \geq \frac{2+\ap}{\ap}$, then for any $\sigma>0$ and any $\xn \in \GTM(\partial \Omega \cup \Sigma)$ with compact support in $\xS$, there is no solution of problem \eqref{u-sigmanu}  (see Remark \ref{rem3} for more details).

Existence results in case boundary data are concentrated on $\partial \Omega$ are stated in the next theorem.
\begin{theorem}\label{th:existnu-prtO}
	Assume that $\mu\leq \left(\frac{N-2}{2}\right)^2$, $p$ satisfies \eqref{p-cond} and  $\xn \in \GTM^+(\partial\xO\cup \xS)$ with compact support in $\partial \xO$.
	Then the following statements are equivalent.
	
	1. Equation \eqref{u-sigmanu} has a positive solution for $\gs>0$ small.
	
	2. For any Borel set $E \subset \partial \Omega $, \eqref{Kp<nu} holds.
	
	3. Estimate \eqref{GKp<K} holds.
	
	4. For any Borel set $E \subset \partial \Omega $, there holds $\xn(E)\leq C\, \mathrm{Cap}_{\frac{2}{p},p'}^{\partial \Omega}(E)$.	
\end{theorem}

Note that when $1<p<\frac{N+1}{N-1}$, statements 1--4 of  Theorem \ref{th:existnu-prtO} are valid, while when $p \geq \frac{N+1}{N-1}$, for any $\sigma>0$ and any $z \in \partial \Omega$, equation \eqref{u-sigmanu} with $\nu=\delta_z$ does not admit any positive solution (see Remark \ref{partO-N}). It will be also pointed out that when $\mu>0$ and $p\geq \frac{2+\am}{\am}$, for any $\sigma>0$ and any $\nu \in \GTM^+(\partial \Omega \cup \Sigma)$ with compact support in $\partial\xO$, problem \eqref{u-sigmanu} does not admit any positive weak solution. This is discussed in Lemma \ref{rem4}. \medskip

\noindent \textbf{Organization of the paper.} In Section \ref{pre}, we present main properties of the submanifold $\Sigma$ and recall important facts about the first eigenfunction, Green kernel and Martin kernel of $-L_\mu$. In Section \ref{sec:weakLp}, we establish sharp estimates on the Green operator and Martin operator, which play an important role in proving the existence of a solution to \eqref{NLP}. We then discuss the notion of boundary trace and several results regarding linear equations involving $-L_\mu$ in Section \ref{sec:linear}. Section \ref{sec:gennon} is devoted to the proof of Theorems \ref{th1} and \ref{th2}. In section \ref{sec:powercase}, we focus on the power case and provide the proof of Theorems \ref{theoremint}--\ref{th:existnu-prtO}. In Appendix \ref{app:A}, we give an estimate which is useful in the proof of several results in Section \ref{sec:weakLp}. \medskip

\subsection{Notations} \label{subsec:notations} We list below notations that are frequently used in the paper.

$\bullet$ Let $\phi$ be a positive continuous function in $\Omega \setminus \Sigma$ and $\kappa \geq 1$. Let $L^\kappa(\Omega;\phi)$ be the space of functions $f$ such that
\bal \| f \|_{L^\kappa(\Omega;\phi)} := \left( \int_{\Omega} |f|^\kappa \phi \, \dd x \right)^{\frac{1}{\kappa}}.
\eal

The weighted Sobolev space $H^1(\Omega;\phi)$ is the space of functions $f \in L^2(\Omega;\phi)$ such that $\nabla f \in L^2(\Omega;\phi)$. This space is endowed with the norm
\bal \| f \|_{H_0^1(\Omega;\phi)}^2= \int_{\Omega} |f|^2 \phi \,\dd x +  \int_{\Omega} |\nabla f|^2 \phi \,\dd x.
\eal
The closure of $C_c^\infty(\Omega)$ in $H^1(\Omega;\phi)$ is denoted by $H_0^1(\Omega;\phi)$.

Denote by $\mathfrak{M}(\Omega;\phi)$ the space of Radon measures $\tau$ in $\Omega$ such that \bal \| \tau\|_{\mathfrak{M}(\Omega;\phi)}:=\int_{\Omega}\phi \, \dd|\tau|<\infty,
\eal
and by $\mathfrak{M}^+(\Omega;\phi)$ its positive cone. Denote by $\GTM(\partial \Omega \cup \Sigma)$ the space of finite measure $\nu$ on $\partial \Omega \cup \Sigma$, namely
\bal \| \nu \|_{\GTM(\partial \Omega \cup \Sigma)}:=|\nu|(\partial \Omega \cup \Sigma) < \infty,
\eal
and by $\GTM^+(\partial \Omega \cup \Sigma)$ its positive cone.

$\bullet$ For a measure $\omega$, denote by $\omega^+$ and $\omega^-$ the positive part and negative part of $\omega$.

$\bullet$ For $\beta>0$, $ \Omega_{\beta}=\{ x \in \Omega: d(x) < \beta\}$, $\Sigma_{\beta}=\{ x \in \R^N \setminus \Sigma:  d_\Sigma(x)<\beta \}$.

$\bullet$ We denote by $c,c_1,C...$ the constant which depend on initial parameters and may change from one appearance to another.

$\bullet$ The notation $A \gtrsim B$ (resp. $A \lesssim B$) means $A \geq c\,B$ (resp. $A \leq c\,B$) where the implicit $c$ is a positive constant depending on some initial parameters. If $A \gtrsim B$ and $A \lesssim B$, we write $A \approx B$. \textit{Throughout the paper, most of the implicit constants depend on some (or all) of the initial parameters such as $N,\Omega,\Sigma,k,\mu$ and we will omit these dependencies in the notations (except when it is necessary).}

$\bullet$ For $a,b \in \BBR$, denote $a \wedge b = \min\{a,b\}$, $a \lor b =\max\{a,b \}$.

$\bullet$ For a set $D \subset \R^N$, $\1_D$ denotes the indicator function of $D$.

\medskip
\noindent \textbf{Acknowledgement.} K. T. Gkikas acknowledges support by the Hellenic Foundation for Research and Innovation
(H.F.R.I.) under the “2nd Call for H.F.R.I. Research Projects to support Post-Doctoral Researchers” (Project
Number: 59). P.-T. Nguyen was supported by Czech Science Foundation, Project GA22-17403S.

\section{Preliminaries} \label{pre}

\subsection{Submanifold $\Sigma$.} \label{assumptionK} Throughout this paper, we assume that $\Sigma \subset \Omega$ is a $C^2$ compact submanifold in $\mathbb{R}^N$ without boundary, of dimension $k$, $0\leq k < N-2$. When $k = 0$ we assume that $\Sigma = \{0\}$.

For $x=(x_1,...,x_k,x_{k+1},...,x_N) \in \R^N$, we write $x=(x',x'')$ where $x'=(x_1,..,x_k) \in \R^k$ and $x''=(x_{k+1},...,x_N) \in \R^{N-k}$. For $\beta>0$, we denote by $B_{\beta}^k(x')$ the ball  in $\R^k$ with center at $x'$ and radius $\beta.$ For any $\xi\in \Sigma$, we set
\bal  \Sigma_\beta &:=\{ x \in \R^N \setminus \Sigma: d_\Sigma(x) < \beta \}, \\
V(\xi,\xb)&:=\{x=(x',x''): |x'-\xi'|<\beta,\; |x_i-\Gamma_i^\xi(x')|<\xb,\;\forall i=k+1,...,N\},
\eal
for some functions $\Gamma_i^\xi: \R^k \to \R$, $i=k+1,...,N$.

Since $\Sigma$ is a $C^2$ compact submanifold in $\mathbb{R}^N$ without boundary, there is $\xb_0$ such that the followings hold.

\begin{itemize}
\item $\Sigma_{6\beta_0}\Subset \Omega$ and for any $x\in \Sigma_{6\beta_0}$, there is a unique $\xi \in \Sigma$  satisfies $|x-\xi|=d_\Sigma(x)$.

\item $d_\Sigma \in C^2(\Sigma_{4\beta_0})$, $|\nabla d_\Sigma|=1$ in $\Sigma_{4\beta_0}$ and there exists $\eta\in L^\infty(\Sigma_{4\beta_0})$ such that (see \cite[Lemma 2.2]{Vbook} and \cite[Lemma 6.2]{DN})
\bal
\Delta d_\Sigma(x)=\frac{N-k-1}{d_\Sigma(x)}+ \eta(x) \quad \text{in } \Sigma_{4\beta_0} .
\eal

\item For any $\xi \in \Sigma$, there exist $C^2$ functions $\Gamma_i^\xi \in C^2(\R^k;\R)$, $i=k+1,...,N$, such that (upon relabeling and reorienting the coordinate axes if necessary), for any $\beta \in (0,6\beta_0)$, $V(\xi,\beta) \subset \Omega$ and
\bal
V(\xi,\beta) \cap \Sigma=\{x=(x',x''): |x'-\xi'|<\beta,\;  x_i=\Gamma_i^\xi (x'), \; \forall i=k+1,...,N\}.
\eal

\item There exist $ m_0 \in \N$ and points $\xi^{j} \in \Sigma$, $j=1,...,m_0$, and $\beta_1 \in (0, \beta_0)$ such that
\be \label{cover}
\Sigma_{2\xb_1}\subset \cup_{j=1}^{m_0} V(\xi^j,\beta_0)\Subset \Omega.
\ee
\end{itemize}

Now set
\bal \xd_\Sigma^\xi(x):=\left(\sum_{i=k+1}^N|x_i-\Gamma_i^\xi(x')|^2\right)^{\frac{1}{2}}, \qquad x=(x',x'')\in V(\xi,4\beta_0).\eal

Then we see that there exists a constant $C=C(N,\Sigma)$ such that
\be\label{propdist}
d_\Sigma(x)\leq	\xd_\Sigma^{\xi^j}(x)\leq C \| \Sigma \|_{C^2} d_\Sigma(x),\quad \forall x\in V(\xi^j,2\beta_0),
\ee
where $\xi^j=((\xi^j)', (\xi^j)'') \in \Sigma$, $j=1,...,m_0$, are the points in \eqref{cover} and
\ba \label{supGamma}
\| \Sigma \|_{C^2}:=\sup\{  || \Gamma_i^{\xi^j} ||_{C^2(B_{5\beta_0}^k((\xi^j)'))}: \; i=k+1,...,N, \;j=1,...,m_0 \} < \infty.
\ea
Moreover, $\beta_1$ can be chosen small enough such that for any $x \in \Sigma_{\beta_1}$,
\bal B(x,\beta_1) \subset V(\xi,\beta_0),
\eal
where $\xi \in \Sigma$ satisfies $|x-\xi|=d_\Sigma(x)$.

\subsection{Eigenvalue of $-L_\mu$} \label{subsect:eigen} Let
\bal
H=\frac{N-k-2}{2}.
\eal
and for $\mu \leq H^2$, let
\bal
\am=H-\sqrt{H^2-\mu}, \quad \ap=H+\sqrt{H^2-\mu}.
\eal
Note that $\am\leq H\leq\ap<2H$ and $\am \geq 0$ if and only if $\mu \geq 0$.

We summarize below main properties of the first eigenfunction of the operator $-L_\mu$ in $\Omega \setminus \Sigma$ from \cite[Lemma 2.4 and Theorem 2.6]{DD1} and \cite[page 337, Lemma 7, Theorem 5]{DD2}.

(i) For any $\mu \leq H^2$, it is known that
\be\label{Lin01} \lambda_\mu:=\inf\left\{\int_{\Gw}\left(|\nabla u|^2-\frac{\xm }{d_\Sigma^2}u^2\right)\dx: u \in C_c^1(\Omega), \int_{\Gw} u^2 \dx=1\right\}>-\infty.
\ee

\smallskip

(ii) If $\mu < H^2$, there exists a minimizer $\gf_{\xm }$ of \eqref{Lin01} belonging to $H^1_0(\Gw)$. Moreover, it satisfies $-L_\mu \phi_\mu= \lambda_\mu \phi_\mu$  in $\Omega \setminus \Sigma$ and
$\phi_{\mu }\approx d_\Sigma^{-\am}$ in $\Sigma_{\beta_0}$.

\smallskip

(iii) If $\xm =H^2$, there is no minimizer of \eqref{Lin01} in $H_0^1(\Gw)$, but there exists a nonnegative function $\phi_{H^2}\in H_{loc}^1(\xO)$  such that $-L_{H^2}\phi_{H^2}=\lambda_{H^2}\phi_{H^2}$ in the sense of distributions in $\Omega \setminus \Sigma$ and $\phi_{H^2}\approx d_\Sigma^{-H}  \quad  \text{in } \Sigma_{\beta_0}$.
In addition, the function $d_\Sigma^{-H}\xf_{H^2}\in H^1_0(\Gw; d_\Sigma^{-2H})$.

From (ii) and (iii) we deduce that, for $\mu \leq H^2$, there holds
\be \label{eigenfunctionestimates}
\xf_\xm \approx d\,d^{-\am}_\Sigma \quad \text{in } \Omega \setminus \Sigma.
\ee

\subsection{Green function and Martin kernel} \label{sec:GreenMartin} Throughout the paper, we always assume that \eqref{assump1} holds. Let $G_\mu$ and $K_{\mu}$ be the Green kernel and Martin kernel of $-L_\mu$ in $\Omega \setminus \Sigma$ respectively.  Let us recall sharp two-sided estimates on Green kernel and Martin kernel.

\begin{proposition}[ {\cite[Proposition 4.1]{GkiNg_linear}} ]  \label{Greenkernel} ~~
	
	(i) If $\mu< \left( \frac{N-2}{2}\right)^2$ then, for any $x,y \in \Omega \setminus \Sigma$, $x \neq y$,
	\bel{Greenesta} \BAL
	G_{\mu}(x,y)&\approx |x-y|^{2-N} \left(1 \wedge \frac{d(x)d(y)}{|x-y|^2}\right)  \left(\frac{|x-y|}{d_\Sigma(x)}+1\right)^\am
	\left(\frac{|x-y|}{d_\Sigma(y)}+1\right)^\am \\
	&\approx |x-y|^{2-N} \left(1 \wedge \frac{d(x)d(y)}{|x-y|^2}\right) \left(1 \wedge \frac{d_\Sigma(x)d_\Sigma(y)}{|x-y|^2} \right)^{-\am}.
	\EAL \ee
	
	(ii) If $k=0$, $\Sigma=\{0\}$ and $\mu = \left( \frac{N-2}{2}\right)^2$ then, for any $x,y \in \Omega \setminus \Sigma$, $x \neq y$,
	\ba\label{Greenestb} \BAL
	&G_{\mu}(x,y) \approx |x-y|^{2-N} \left(1 \wedge \frac{d(x)d(y)}{|x-y|^2}\right) \left(\frac{|x-y|}{|x|}+1\right)^{\frac{N-2}{2}}
	\left(\frac{|x-y|}{|y|}+1\right)^{\frac{N-2}{2}}\\
	& \qquad \qquad +(|x||y|)^{-\frac{N-2}{2}}\left|\ln\left(1 \wedge \frac{|x-y|^2}{d(x)d(y)}\right)\right| \\
	&\approx |x-y|^{2-N} \left(1 \wedge \frac{d(x)d(y)}{|x-y|^2}\right) \left(1 \wedge \frac{|x||y|}{|x-y|^2} \right)^{-\frac{N-2}{2}} +(|x||y|)^{-\frac{N-2}{2}}\left|\ln\left(1 \wedge \frac{|x-y|^2}{d(x)d(y)}\right)\right|.
	\EAL \ea
	
	The implicit constants in \eqref{Greenesta} and \eqref{Greenestb} depend on $N,\Omega,\Sigma,\mu$.
\end{proposition}

%
\begin{proposition}[{\cite[Theorem 1.2]{GkiNg_linear}}] \label{Martin} ~~
	
	(i) If $\mu< \left( \frac{N-2}{2}\right)^2$ then
	\be \label{Martinest1}
	K_{\mu}(x,\xi) \approx\left\{
	\BAL
	&\frac{d(x)d_\Sigma(x)^{-\am}}{|x-\xi|^N}\quad &&\text{if } x \in \Omega \setminus \Sigma,\;  \xi \in \partial\xO, \\
	&\frac{d(x)d_\Sigma(x)^{-\am}}{|x-\xi|^{N-2-2\am}} &&\text{if } x \in \Omega \setminus \Sigma,\; \xi \in \Sigma.
	\EAL \right.
	\ee
	
	(ii) If  $k=0$, $\Sigma=\{0\}$ and $\mu= \left( \frac{N-2}{2}\right)^2$ then
	\be\label{Martinest2}
	K_{\mu}(x,\xi) \approx\left\{
	\BAL
	&\frac{d(x)|x|^{-\frac{N-2}{2}}}{|x-\xi|^N}\quad &&\text{if } x \in \Omega \setminus \{0\},\; \xi \in \partial\xO, \\
	&d(x)|x|^{-\frac{N-2}{2}}\left|\ln\frac{|x|}{\CD_\Omega}\right| &&\text{if } x \in \Omega \setminus \{0\},\; \xi=0,
	\EAL \right.
	\ee
	where $\CD_\Omega:=2\sup_{x \in \Omega}|x|$.
	
	The implicit constants depend on $N,\Omega,\Sigma,\mu$.
\end{proposition}

\section{Weak Lebesgue estimates} \label{sec:weakLp}
\subsection{Auxiliary estimates} We first recall the definition of weak Lebesgue spaces (or Marcinkiewicz spaces). Let $D \subset \R^N$ be a domain. Denote by $L^\kappa_w(D;\tau)$, $1 \leq \kappa < \infty$, $\tau \in \GTM^+(D)$, the
weak $L^\kappa$ space defined as follows: a measurable function $f$ in $D$
belongs to this space if there exists a constant $c$ such that
\bal \gl_f(a;\tau):=\tau(\{x \in D: |f(x)|>a\}) \leq ca^{-\kappa},
\forevery a>0. \eal
The function $\gl_f$ is called the distribution function of $f$ (relative to
$\tau$). For $p \geq 1$, denote
\bel{weakLp}
 L^\kappa_w(D;\tau):=\{ f \text{ Borel measurable}:
\sup_{a>0}a^\kappa\gl_f(a;\tau)<\infty\}
\ee
and
\bel{semi}
\norm{f}^*_{L^\kappa_w(D;\tau)}:=(\sup_{a>0}a^\kappa\gl_f(a;\tau))^{\frac{1}{\kappa}}. \ee
Note that $\norm{.}_{L^\kappa_w(D;\tau)}^*$ is not a norm, but for $\kappa>1$, it is
equivalent to the norm
\bal \norm{f}_{L^\kappa_w(D;\tau)}:=\sup\left\{
\frac{\int_{A}|f|\dd\tau}{\tau(A)^{1-\frac{1}{\kappa}}}: A \sbs D, \, A \text{
	measurable},\, 0<\tau(A)<\infty \right\}. \eal
More precisely,
\bel{equinorm} \norm{f}^*_{L^\kappa_w(D;\tau)} \leq \norm{f}_{L^\kappa_w(D;\tau)}
\leq \frac{\kappa}{\kappa-1}\norm{f}^*_{L^\kappa_w(D;\tau)}. \ee
We also denote by $\tilde L_w^\kappa$ the weak type $L^\kappa$ space with norm
\bel{normLwww} \norm{f}_{\tilde L^\kappa_w(D;\tau)}:=\sup\left\{
\frac{\int_{A}|f|\dd\tau}{\tau(A)^{1-\frac{1}{\kappa}} \ln(e+\tau(A)^{-1})}: A \sbs D, \, A \text{
	measurable},\, 0<\tau(A)<\infty \right\}. \ee

When $\dd\tau=\varphi \, \dx$ for some positive continuous function $\varphi$, for simplicity, we use the notation $L_w^\kappa(D;\varphi)$.  Notice that
$ L_w^\kappa(D;\varphi) \sbs L^{r}(D;\varphi)$ for any $r \in [1,\kappa)$. From \eqref{semi} and \eqref{equinorm}, one can derive the following estimate which is useful in the sequel. For any $f \in L_w^\kappa(D;\varphi)$, there holds
\bel{ue} \int_{\{x \in D: |f(x)| \geq s\} }\varphi \dd x \leq s^{-\kappa}\norm{f}^\kappa_{L_w^\kappa(D;\varphi)}.
\ee

Let us recall a result from \cite{BVi} which will be used in the proof of weak Lebesgue estimates for the Green kernel and Martin kernel.

\begin{proposition}[{\cite[Lemma 2.4]{BVi}}] \label{bvivier}
	Assume $D$ is a bounded domain in $\R^N$ and denote by $\tilde D$ either the set $D$ or the boundary $\partial D$. Let $\gw$ be a nonnegative bounded Radon measure in $\tilde D$ and $\eta\in C(D)$ be a positive weight function. Let $\CH$ be a continuous nonnegative function
	on $\{(x,y)\in D\times \tilde D:\;x\neq y\}.$ For any $\xl > 0$ we set
	\bal
	A_\xl(y):=\{x\in D\setminus\{y\}:\;\; \CH(x,y)>\xl\}\quad \text{and} \quad
	m_{\xl}(y):=\int_{A_\xl(y)}\eta(x) \, \dd x.
	\eal
	Suppose that there exist $C>0$ and $\kappa>1$ such that $m_{\xl}(y)\leq C\xl^{-\kappa}$ for every $\gl>0$.  Then the operator
	\bal
\BBH[\gw](x):=\int_{\tilde D}\CH(x,y)\dd\gw(y)
\eal
	belongs to $L^\kappa_w(D;\eta )$ and
	\bal
\left|\left|\BBH[\gw]\right|\right|_{L^\kappa_w(D;\eta)}\leq (1+\frac{C\xk}{\kappa-1})\gw(\tilde D).
\eal
\end{proposition}

In the sequel, we will use the following notations.
For $\alpha,\xg \in \R$, let
\bel{varphia} \varphi_{\alpha,\xg}(x):= d_\Sigma(x)^{-\alpha}d(x)^{\xg}, \quad x \in \Omega \setminus \Sigma.
\ee
For $\kappa, \theta,\gamma\in \R$, we define
\bel{F1}
F_{\kappa,\theta,\gamma}(x,y):=d_\Sigma(x)^{\kappa} |x-y|^{-N+2+\theta}d(y)^{-\xg} \left(  1 \land \frac{d(x)d(y)}{|x-y|^2}\right),
\ee
for $x \neq y, x,y \in \Omega \setminus \Sigma$, and for any positive function $\varphi$ on $\Omega \setminus \Sigma$, set
\bal
\mathbb{F}_{\kappa,\theta,\xg}[\varphi \tau](x):=\int_{\Omega \setminus \Sigma}F_{\kappa,\theta,\gamma}(x,y)\varphi(y)\dd\tau(y), \quad \tau \in \GTM(\Omega  \setminus \Sigma;\varphi).
\eal
Put
\bel{p1}  \p_{\alpha,\theta,\xg}:=\min\left\{\frac{N-\xa}{N-2-\xa},\frac{N+\xg}{N-2+\xg-\theta}\right\}.
\ee

\begin{lemma} \label{anisotitaweakF1a}
Let $0<\alpha\leq H$, where $H$ is defined in \eqref{valueH}, and $0 \leq \xg\leq1$. Then
\bel{estF1}
	\| \mathbb{F}_{-\xa,2\xa,\gamma}[\varphi_{\alpha,\gamma}\tau]\|_{L_w^{\p_{\alpha,2\xa,\xg}}(\Gw\setminus \Sigma;\varphi_{\alpha,1} )} \lesssim
	\|\gt\|_{\mathfrak{M}(\xO\setminus \Sigma;\varphi_{\alpha,\xg})}, \quad \forall \tau\in \mathfrak{M}(\xO\setminus \Sigma; \varphi_{\alpha,\xg}).
	\ee
	The implicit constant in \eqref{estF1} depends on $N,\Omega,\Sigma,\alpha,\xg$.
\end{lemma}

\begin{proof}
Without loss of generality, we may assume $\tau \in \GTM^+(\Omega \setminus \Sigma;\varphi_{\xa,\xg})$. Set
\bal
A_\xl(y)&:=\Big\{x\in(\xO\setminus \Sigma)\setminus\{y\}:\;\; F_{-\xa,2\alpha,\gamma}(x,y)>\xl \Big \},\\ \nonumber
A_{\xl,1}(y)&:=\Big\{x\in(\xO\setminus \Sigma)\setminus\{y\}:\;\; F_{-\xa,2\alpha,\gamma}(x,y)>\xl\;\text{ and}\;d_\Sigma(x)\leq |x-y| \Big \},\\ \nonumber
A_{\xl,2}(y)&:=\Big\{x\in(\xO\setminus \Sigma)\setminus\{y\}:\;\; F_{-\xa,2\alpha,\gamma}(x,y)>\xl\;\text{ and}\;d_\Sigma(x)> |x-y| \Big \},\\ \nonumber
m_{\xl}(y)&:=\int_{A_\xl(y)}\varphi_{\alpha,1} \dx,\quad m_{\xl,i}(y):=\int_{A_{\xl,i}(y)}\varphi_{\alpha,1}\dx, \quad i=1,2.
\eal
Then
$A_\xl(y)=A_{\xl,1}(y)\cup A_{\xl,2}(y)$
and
\bal
m_{\xl}(y)= m_{\xl,1}(y)+ m_{\xl,2}(y).
\eal

Let $\beta_1$ be as in \eqref{cover}. We write
\ba\label{ml1F2-0}
m_{\xl}(y)=\int_{A_\xl(y)\cap \Sigma_{ \frac{\beta_1}{4} }}d(x)d_\Sigma(x)^{-\alpha} \dx+\int_{A_\xl(y)\setminus \Sigma_{ \frac{\beta_1}{4}}}d(x)d_\Sigma(x)^{-\alpha} \dx.
\ea
We will estimate successively the terms on the right hand side of \eqref{ml1F2-0}. We consider only the case  $H<\frac{N-2}{2}$ since the case $H=\frac{N-2}{2}$ (i.e. in case $k=0$, $\Sigma=\{0\}$) can be treated in a similar way.

We split the first term on the right hand side of \eqref{ml1F2-0} as
\ba \label{mcompose-1} \int_{A_{\xl}(y)\cap \Sigma_{ \frac{\beta_1}{4} }}d(x)d_\Sigma(x)^{-\alpha} \dx = \int_{A_{\xl,1}(y)\cap \Sigma_{ \frac{\beta_1}{4} }}d(x)d_\Sigma(x)^{-\alpha} \dx + \int_{A_{\xl,2}(y)\cap \Sigma_{ \frac{\beta_1}{4} }}d(x)d_\Sigma(x)^{-\alpha} \dx.
\ea
We note that
\bel{app:6}
1 \land \frac{d(x)d(y)}{|x-y|^2}  \leq 2\left( 1 \land \frac{d(y)}{|x-y|} \right)  \leq 4  \frac{d(y)}{d(x)}, \quad \forall x,y \in \Omega, \; x \neq y,
\ee
therefore
\be\label{52}
F_{-\xa,2\alpha,\gamma}(x,y)\leq 4^\xg d_\Sigma(x)^{-\xa} d(x)^{-\xg}|x-y|^{-N+2+2\xa},\quad \forall x,y \in \Omega, \; x \neq y.
\ee
Since $0<\alpha <\frac{N-2}{2}$, from \eqref{52} we see that
\bal
A_{\xl,1}(y)\cap \Sigma_{ \frac{\beta_1}{4} }
\subset \Big\{x\in(\xO\setminus \Sigma)\setminus\{y\}:\; d_\Sigma(x) < c\lambda^{-\frac{1}{N-2-\alpha}},\;  |x-y| <c\lambda^{-\frac{1}{N-2-2\alpha}}d_\Sigma(x)^{-\frac{\xa}{N-2-2\alpha}} \Big \}.
\eal
By applying Lemma \ref{lemapp:1} with $\alpha_1=-\alpha$, $\alpha_2=-\frac{\xa}{N-2-2\alpha}$, $\ell_1=\lambda^{-\frac{1}{N-2-\alpha}}$, $\ell_2=\lambda^{-\frac{1}{N-2-2\alpha}}$ and taking into account that $N-k - \alpha -\frac{k\xa}{N-2-2\alpha} \geq 2$ since $\alpha\leq H$, we deduce, for $\lambda \geq 1$,
\ba \label{ca1-1.1}
\int_{A_{\xl,1}(y)\cap \Sigma_{ \frac{\beta_1}{4} }}d(x)d_\Sigma(x)^{-\alpha} \dx \lesssim \lambda^{-\frac{N-\alpha}{N-2-\alpha}} \leq \lambda^{-\p_{\alpha,2\alpha,\xg}}.
\ea
Next, by \eqref{52}, we see that
\bal
A_{\lambda,2}(y)\cap \Sigma_{\frac{\beta_1}{4}} \subset \left\{ x \in \Omega \setminus \Sigma:  |x-y| < c\lambda^{-\frac{1}{N-2-\alpha}} \text{ and } d_\Sigma(x) > |x-y| \right\}.
\eal
Therefore, for every $\lambda \geq 1$,
\ba \label{ca1-1.3} \begin{aligned}
\int_{A_{\lambda,2}(y)\cap \Sigma_{\frac{\beta_1}{4}}}d(x)d_\Sigma(x)^{-\alpha}\dx
\lesssim \int_{\{|x-y|\leq c\lambda^{-\frac{1}{N-2-\alpha}}\}}|x-y|^{-\alpha}\dx
\lesssim \lambda^{-\frac{N-\alpha}{N-2-\alpha}} \leq \lambda^{-\p_{\alpha,2\alpha,\xg}}.
\end{aligned} \ea
Combining \eqref{mcompose-1}, \eqref{ca1-1.1} and \eqref{ca1-1.3} yields, for any $\lambda \geq 1$,
\be \label{AA1}
\int_{A_{\xl}(y)\cap \Sigma_{ \frac{\beta_1}{4} }}d(x)d_\Sigma(x)^{-\alpha} \dx \lesssim \lambda^{-\p_{\alpha,2\alpha,\xg}}.
\ee

Next we estimate the second term on the right hand side of \eqref{ml1F2-0}. By \eqref{app:6}, we have
\bal
A_\xl(y)\cap (\xO\setminus\Sigma_{ \frac{\beta_1}{4}})\subset\left\{ x \in \Omega \setminus \Sigma:  |x-y| < c\lambda^{-\frac{1}{N+\xg-2-2\alpha}} \text{ and } d(x)^{\xg} \leq \xl^{-1}|x-y|^{-N+2+2\xa} \right\}.
\eal
This yields, for $\lambda \geq 1$,
\ba \label{AA-2} \BAL
\int_{A_{\lambda}(y) \setminus \Sigma_{\frac{\beta_1}{4}}} d(x)d_{\Sigma}(x)^{-\alpha}\dx &\lesssim \int_{A_{\lambda}(y) \setminus \Sigma_{\frac{\beta_1}{4}}} d(x)^\xg\dx \\ &\lesssim\int_{\{|x-y| < c\lambda^{-\frac{1}{N+\xg-2-2\alpha}}\}} \xl^{-1}|x-y|^{-N+2+2\xa}\dx\\
&\lesssim \lambda^{-\frac{N+\gamma}{N-2+\gamma-2\alpha}} \lesssim \lambda^{-\p_{\alpha,2\alpha,\gamma}}.
\EAL
\ea

Combining \eqref{ml1F2-0}, \eqref{AA1} and \eqref{AA-2} yields
\bel{ca2-1.5}
m_{\lambda}(y)\leq  C\lambda^{-\p_{\alpha, 2\alpha,\xg}},  \quad \forall \lambda>0,
\ee
where $C=C(N,\Omega,\Sigma,\alpha,\xg)$.
By applying Proposition \ref{bvivier} with $\CH(x,y)=F_{\xa,2\alpha,\gamma}(x,y),$ $\tilde D=D=\Omega \setminus \Sigma$, $\eta=d\, d_\Sigma^{-\alpha}$ and $\omega=d^\xg\, d_\Sigma^{-\alpha} \tau$ and using \eqref{ca2-1.5}, we finally derive \eqref{estF1}.
\end{proof}

By using a similar argument as in the proof of Lemma \ref{anisotitaweakF1a}, one can obtain the following lemma.
\begin{lemma} \label{anisotitaweakF1b}
Let $0<\alpha\leq H$ and $0 \leq \xg\leq1$. Then
\bel{estF1b}
	\| \mathbb{F}_{\xa,0,\gamma}[\varphi_{\alpha,\gamma}\tau]\|_{L_w^{\p_{\alpha,0,\xg}}(\Gw\setminus \Sigma;\varphi_{\alpha,1} )} \lesssim
	\|\gt\|_{\mathfrak{M}(\xO\setminus \Sigma;\varphi_{\alpha,\xg})}, \quad \forall \tau\in \mathfrak{M}(\xO\setminus \Sigma; \varphi_{\alpha,\xg}).
	\ee
	The implicit constant in \eqref{estF1b} depends on $N,\Omega,\Sigma,\alpha,\xg$.
\end{lemma}

Set
\bel{F4}
\tilde F_{\xg}(x,y):=|x|^{-\frac{N-2}{2}} d(y)^{-\xg} \left|\ln\left(1 \land \frac{|x-y|^2}{d(x)d(y)}\right)\right|, \quad  x \neq y,\; x,y \in \Omega \setminus \{0\},
\ee
\bal
\mathbb{\tilde F}_{\xg}[\varphi_{\frac{N-2}{2},\xg}\tau](x):=\int_{\Omega \setminus \{0\}}\tilde F_{\xg}(x,y)\varphi_{\frac{N-2}{2},\xg}(y)d\tau(y), \quad \tau \in \mathfrak{M}(\xO\setminus \{0\};\varphi_{\frac{N-2}{2},\xg}),
\eal
where $\varphi_{\frac{N-2}{2},\gamma}$ is defined in \eqref{varphia}.

For $\theta, \kappa \in \R$, put
\bal
\tilde \p_{\theta,\kappa}:=\min\left\{\frac{N+\theta}{N-2},N+\kappa \right\}.
\eal
\begin{lemma} \label{anisotitaweakF4}
Let $k=0$, $\Sigma=\{0\}$, $-N+1<\kappa<1$, $-2<\theta<2$. Then
\bel{estF4}
	\norm{\mathbb{\tilde F}_\gamma[\varphi_{\frac{N-2}{2},\xg}\gt]}_{L_w^{\tilde \p_{\theta,\kappa}}(\Gw\setminus \{0\};\varphi_{\frac{N-2}{2},1})}
	\lesssim \norm{\gt}_{\mathfrak{M}(\xO\setminus \{0\};\varphi_{\frac{N-2}{2},\xg})}, \quad \forall \tau\in \mathfrak{M}(\Omega\setminus \{0\};\varphi_{\frac{N-2}{2},\xg}).
	\ee
	The implicit constant depends on $N,\Omega,\gamma,\theta,\kappa$.
\end{lemma}
\begin{proof} We may assume $\tau\in \mathfrak{M}^+(\Omega\setminus \{0\};\varphi_{\frac{N-2}{2},\xg})$.
For $\lambda>0$ and $y \in \Omega \setminus \{0\}$, set
\bal
&A_\xl(y):=\Big\{x\in \xO\setminus \{0,y\}:\;\; \tilde F_{\xg}(x,y)>\xl \Big \} \quad \text{and} \quad m_{\xl}(y):=\int_{A_\xl(y)}d(x)|x|^{-\frac{N-2}{2}} \dx, \\
&A_{\xl,1}(y):=\Big\{x\in \xO\setminus \{0,y\}:\;\; \tilde F_{\xg}(x,y)>\xl \text{ and } |x-y| \leq |x|  \Big \}, \\
&A_{\xl,2}(y):=\Big\{x\in \xO\setminus \{0,y\}:\;\; \tilde F_{\xg}(x,y)>\xl \text{ and } |x-y| \geq |x|  \Big \}.
\eal

It can be shown from \eqref{F4} that
\ba \label{f5a}
\tilde F_{\xg}(x,y)\leq 2d(y)^{-\xg}|x|^{-\frac{N-2}{2}}\left(-\ln\frac{|x-y|}{\CD_\Omega}\right) \left(  1 \land \frac{d(x)d(y)}{|x-y|^2} \right), \; \forall x\neq y,\;x,y\in \xO\setminus\{0\},
\ea
where $\CD_\Omega = 2\sup_{x \in \Omega}|x|$.

We write
\ba
m_{\xl}(y)=\int_{A_\xl(y)\cap B(0,\frac{\beta_1}{4})}d(x)|x|^{-\frac{N-2}{2}} \dx+\int_{A_\xl(y)\setminus B(0,\frac{\beta_1}{4})}d(x)|x|^{-\frac{N-2}{2}} \dx.\label{ml1F4}
\ea

The first term on the right hand side of \eqref{ml1F4} is estimated by using \eqref{f5a} and \eqref{app:6} as
\ba \label{ml6f5} \begin{aligned}
&\int_{A_{\xl}(y)\cap B(0,\frac{\beta_1}{4})}d(x)|x|^{-\frac{N-2}{2}}\dx \lesssim \int_{A_{\xl,1}(y)\cap B(0,\frac{\beta_1}{4}) }|x|^{-\frac{N-2}{2}}\dx+\int_{A_{\xl,2}(y)\cap B(0,\frac{\beta_1}{4})}|x|^{-\frac{N-2}{2}}\dx\\
&\lesssim \int_{\{|x-y|\leq c(\xl^{-1}\ln\xl)^{\frac{2}{N-2}}\}}|x-y|^{-\frac{N-2}{2}}\dx+\int_{\{|x|\leq c(\xl^{-1}\ln\xl)^{\frac{2}{N-2}}\}}|x|^{-\frac{N-2}{2}}\dx \\
&\lesssim (\xl^{-1}\ln\xl)^{\frac{N+2}{N-2}}, \quad \forall \lambda>e.
\end{aligned}
\ea

The second term on the right hand side of \eqref{ml1F4} is estimated using \eqref{app:6} and \eqref{f5a} as
\ba \label{ml4f5} \begin{aligned}
\int_{A_\xl(y)\setminus B(0,\frac{\beta_1}{4})}d(x)|x|^{-\frac{N-2}{2}} \dx&\lesssim \int_{\{|x-y|^{-1}(-\ln\frac{|x-y|}{\CD_\Omega})\geq c\xl\}} \lambda^{-1}\left(-\ln\frac{|x-y|}{\CD_\Omega} \right)\dx\\
&\lesssim \int_{\{|x-y|\leq c'\xl^{-1}\ln\xl\}} \xl^{-1}\left(-\ln\frac{|x-y|}{\CD_\Omega}\right) \dd x \\
&\lesssim (\xl^{-1}\ln\xl)^{N+1},\quad \forall \xl>e.
\end{aligned}
\ea

Combining \eqref{ml1F4},  \eqref{ml6f5} and \eqref{ml4f5}, together with $-2<\theta<2$, we deduce
\bel{ml6f4}
m_{\xl}(y) \lesssim (\xl^{-1}\ln\xl)^{\frac{N+2}{N-2}}+(\xl^{-1}\ln\xl)^{N+1} \lesssim \lambda^{-\frac{N+\theta}{N-2}} + \lambda^{-(N+\kappa)} \lesssim \lambda^{-\tilde \p_{\theta,\kappa}},\quad \forall\xl>e.
\ee

Thus by applying Proposition \ref{bvivier} with $\mathcal{H}(x,y)=\tilde F_{\xg}(x,y)$, $\tilde D=D=\Omega \setminus \{0\}$, $\eta(x)=d(x)|x|^{-\frac{N-2}{2}}$ and $\dd\nu=d(x)^\xg|x|^{-\frac{N-2}{2}}\dd\tau$, we obtain \eqref{estF4}.
\end{proof}

For $\alpha,\theta \in \R$, put
\ba \label{Halthe} H_{\alpha,\theta}(x,y):=d(x)d_{\Sigma}(x)^{-\alpha} |x-y|^{-N+\theta}, \quad x \in \Omega \setminus \Sigma, y \in \partial \Omega \cup \Sigma,
\ea
\bal
\mathbb{H}_{\alpha,\theta}[\nu](x): = \int_{\partial \Omega \cup \Sigma} H_{\alpha,\theta}(x,y)\dd\nu(y), \quad \nu \in \GTM(\partial \Omega \cup \Sigma),
\eal
and
\bal
\q_{\alpha,\theta}:= \min \left \{  \frac{N-k-\alpha}{\alpha}, \frac{N+1}{N-1-\theta} \right\}.
\eal

\begin{theorem} \label{H}

(i) Assume $k \geq 0$, $0<\alpha \leq H$, $\theta<N-1$,  and $\gn\in \mathfrak{M}(\partial\xO\cup \Sigma)$ with compact support in $\partial\xO$. Then
\ba \label{est:H1}
\norm{\mathbb{H}_{\alpha,\theta}[\nu]}_{L_w^{ \q_{\alpha,\theta}}(\Gw\setminus \Sigma;\varphi_{\alpha,1})} \lesssim \|\nu\|_{\mathfrak{M}(\partial\Omega \cup \Sigma)}.
\ea

(ii) Assume $k>0$, $\alpha \leq H$, $\theta\leq N-k$, $\theta<N+\alpha$, and $\gn\in \mathfrak{M}(\partial\xO\cup \Sigma)$ with compact support in $\Sigma$. Then
\ba \label{est:H2}
\norm{\mathbb{H}_{\alpha,\theta}[\nu]}_{L_w^{\frac{N-\alpha}{N+\alpha-\theta}}(\Gw\setminus \Sigma;\ei)} \lesssim \norm{\nu}_{\mathfrak{M}(\partial \Omega \cup \Sigma)}.
\ea
The implicit constants in \eqref{est:H1} and \eqref{est:H2} depend only on $N,\Omega,\Sigma,\alpha,\theta$.
\end{theorem}
\begin{proof} For $y\in \partial \Omega \cup \Sigma$, set
\bal
A_\xl(y):=\Big\{x\in(\xO\setminus \Sigma):\; H_{\alpha,\theta}(x,y)>\xl \Big \}, \quad m_{\xl}(y)&:=\int_{A_\xl(y)}d(x)d_\Sigma(x)^{-\alpha} \dx.
\eal	
We write
\ba \label{split-K}
m_{\xl}(y)=\int_{A_\xl(y)\cap \Sigma_{\xb_1}}d(x)d_\Sigma(x)^{-\alpha} \dx+\int_{A_\xl(y)\setminus \Sigma_{\xb_1}}d(x)d_\Sigma(x)^{-\alpha} \dx.
\ea

(i) Assume $\gn\in \mathfrak{M}(\partial\xO\cup \Sigma)$ with compact support in $\partial\xO$ and without loss of generality, we may assume that $\gn \geq 0$. Let $y \in \partial \Omega$.

First we treat the first term on the right hand side of  \eqref{split-K}. If $0<\alpha \leq H$ then by applying Lemma \ref{lemapp:1}, we obtain, for $\lambda \geq 1$,
\bal
\int_{A_\xl(y)\cap \Sigma_{\xb_1}}d_\Sigma(x)^{-\alpha} \dx \lesssim \int_{\{d_\Sigma(x)\leq c\lambda^{-\frac{1}{\alpha}}\} \cap \Sigma_{\beta_1}}d_\Sigma(x)^{-\alpha}\dx \lesssim \lambda^{-\frac{N-k-\alpha}{\alpha}} \leq \lambda^{ - \q_{\alpha,\theta}}.
\eal
If $\alpha \leq 0$ then there exists $\bar C=\bar C(N,\Omega,\Sigma,\alpha,\theta)>1$ such that for any $\lambda>\bar C$, $A_\lambda(y) \cap \Sigma_{\beta_1}=\emptyset$. Consequently, for all $\lambda>\bar C$,
\bel{ml3mb}
\int_{A_\xl(y)\cap \Sigma_{\xb_1}}d_\Sigma(x)^{-\alpha} \dx=0.
\ee

Next we treat the second term on the right hand side of \eqref{split-K}. By using the estimate $d(x) \leq |x-y|$, we see that, for $\lambda \geq 1$,
\ba \label{ml2ma}
\int_{A_\xl(y)\setminus \Sigma_{\xb_1}}d(x)\dx \lesssim \int_{\{|x-y|\leq c\lambda^{-\frac{1}{N-1-\theta}}\}}|x-y| \dx \lesssim \xl^{-\frac{N+1}{N-1-\theta}} \leq \lambda^{ - \q_{\alpha,\theta}}.
\ea
Combining \eqref{ml3mb} and \eqref{ml2ma}, we obtain
\ba \label{ml4ma}
m_{\xl}(y)\leq C\lambda^{-\q_{\alpha,\theta}},
\ea
for all $\lambda>\bar C$, where $C=C(N,\Omega,\Sigma,\alpha,\theta)$. Then we can show that \eqref{ml4ma} holds true for all $\lambda>0$. By applying Proposition \ref{bvivier} with $\mathcal{H}(x,y)=H_{\alpha,\theta}(x,y)$, $\tilde D=D=\Omega \setminus \Sigma$, $\eta=\varphi_{\alpha,1}$ and $\omega=\xn$, we obtain \eqref{estmartin2}.

(ii) Assume $\gn\in \mathfrak{M}(\partial\xO\cup \Sigma)$ with compact support in $\Sigma$ and without loss of generality, we may assume that $\gn \geq 0$. Let $y \in \Sigma$. \medskip

\noindent \textbf{Case 1:} $0<\alpha \leq H$.
First we treat the first term in \eqref{split-K}. We notice that since $y \in \Sigma$, $d_\Sigma(x) \leq |x-y|$ for every $x \in \Omega \setminus \Sigma$, hence
\bal
A_\lambda(y) \subset \{ x \in \Omega \setminus \Sigma: d_\Sigma(x) \leq c\lambda^{-\frac{1}{N+\alpha-\theta}} \quad \text{and} \quad |x-y|<c\lambda^{-\frac{1}{N-\theta}}d_\Sigma(x)^{-\frac{\alpha}{N-\theta}}   \}.
\eal
Therefore, by applying Lemma \ref{lemapp:1} with $\alpha_1=-\alpha$, $\alpha_2=-\frac{\alpha}{N-\theta}$, $\ell_1=c\lambda^{-\frac{1}{N+\alpha-\theta}}$, $\ell_2=c\lambda^{-\frac{1}{N-\theta}}$ and noting that $N-k-\alpha- \frac{k\alpha}{N-\theta}\geq 2$ due to the fact that $\alpha \leq H$ and $\theta \leq N-k$, we obtain
\ba \label{ml3maK-1}
\int_{A_\xl(y)\cap \Sigma_{\beta_1}}d_\Sigma(x)^{-\alpha} \dx \lesssim \lambda^{-\frac{N-\alpha}{N+\alpha-\theta}}.
\ea

Next we treat the second term in \eqref{split-K}. We see that there exists a constant $\bar C=\bar C(N,\Omega,\Sigma,\alpha,\theta)>1$ such that for any $\lambda>\bar C$, there holds
\ba \label{ml3maK-2}
\int_{A_\xl(y)\setminus \Sigma_{\beta_1}}d(x)\dx=0.
\ea

Combining \eqref{split-K}, \eqref{ml3maK-1} and \eqref{ml3maK-2}, we deduce
\ba \label{mlam-1}
m_{\xl}(y)\leq C\,\lambda^{-\frac{N-\alpha}{N+\alpha-\theta}}.
\ea
for all $\lambda > \hat C$, where $C=C(N,\Omega,\Sigma,\alpha,\theta)$. \medskip

\noindent \textbf{Case 2:} $\alpha \leq 0$. By noting that $d_\Sigma(x)^{-\alpha} \leq |x-y|^{-\alpha}$ and $|x-y| \leq c \lambda^{-\frac{1}{N-2-\am}}$ for every $x \in A_\lambda(y)$, we can easily obtain \eqref{mlam-1}.

From case 1 and case 2, by applying Proposition \ref{bvivier} with $\mathcal{H}(x,y)=H_{\alpha,\theta}(x,y)$, $D = \Omega \setminus \Sigma$, $\tilde D=\partial\xO\cup \xS,$ $\eta=\varphi_{\alpha,1}$ and $\omega=\xn$, we obtain \eqref{estmartin2}. The proof is complete.
\end{proof}
\smallskip

We put
\bal
\tilde H_{\alpha}(x,y):=d(x)|x-y|^{-\alpha} \left| \ln\frac{|x-y|}{\CD_\Omega} \right|, \quad x \in \Omega \setminus \{y\},
\eal
\bal
\mathbb{\tilde H}_{\alpha}[\nu](x): = \int_{\partial \Omega \cup \Sigma} \tilde H_{\alpha}(x,y)\,\dd\nu(y),
\eal
where $\CD_\Omega=2\sup_{x \in \Omega}|x|$.
\begin{theorem} \label{LwSigma0}
Assume  $0<\alpha<\frac{N}{2}$, $\xr\in[-1,1] \setminus \{0\}$, $0 \in \Omega$ and let $\xd_0$ be the Dirac measure concentrated on $\{0\}$. For $\lambda>0$, set
\ba
\tilde{A}_\xl(0):=\Big\{x\in \xO\setminus \{0\}:\;  |\mathbb{\tilde H}_{\alpha}[\rho \delta_0](x)|>\xl \Big \}, \quad \tilde{m}_{\xl}&:=\int_{\tilde{A}_\xl(0)}d(x)|x|^{-\alpha} \dx.\label{69}
\ea
Then
\ba\label{54}
\tilde{m}_{\xl}\lesssim (\xl^{-1}\ln(e+\xl))^{\frac{N-\alpha}{\alpha}}(|\xr|\ln(e+|\xr|^{-1}))^{\frac{N-\alpha}{\alpha}}, \quad \forall \lambda>0,
\ea
and
\ba \label{est-LwSigma0}
\| \mathbb{\tilde H}_{\alpha}[\rho \delta_0] \|_{\tilde L_w^{\frac{N-\xa}{\xa}}(\Omega \setminus \{0\};\varphi_{\alpha,1})} \lesssim |\xr|.
\ea 	
The implicit constants in the above estimates depend only on $N,\Omega, \alpha$. Here weak Lebesgue spaces $\tilde L_w^p$ are defined in \eqref{normLwww}.
\end{theorem}
\begin{proof}

Consider $\xl> \max\{\CD_\Omega,\CD_\Omega^{-1},e \}$ and
\bal
A_\xl(0):=\Big\{x\in \xO\setminus \{0\}:\;  \tilde H_{\alpha}(x,0)>\xl \Big \}, \quad m_{\xl}&:=\int_{A_\xl(0)}d(x)|x|^{-\alpha} \dx.
\eal
We note that
$
A_\lambda(0) \subset \left\{ x \in \Omega \setminus \{0\}:  |x| \leq c\left(\xl^{-1} \ln\xl \right)^{\frac{1}{\alpha}} \right\} $.
As a consequence,
\bal
m_{\xl} &\lesssim \int_{A_\xl(0)}|x|^{-\alpha} \dx \lesssim \int_{ \left\{|x|\leq c(\xl^{-1} \ln \xl)^{\frac{1}{\alpha}} \right\}}|x|^{-\alpha} \dx  \lesssim (\xl^{-1} \ln\xl)^{\frac{N-\alpha}{\alpha}}.
\eal
Therefore,
\bal
m_{\xl} &\lesssim \int_{A_\xl(0)}|x|^{-\alpha} \dx  \lesssim (\xl^{-1}\ln (e+\xl))^{\frac{N-\alpha}{\alpha}},\quad\forall \xl>0.
\eal
This implies \eqref{54}.

Let $A \subset \xO\setminus \{0\}$ be a measurable set such that $|A|>0$ and let $\dd\tau=d(x)|x|^{-\alpha} \dd x.$ Then for any $\xl>0,$ we have
\bal
\int_A \tilde H_{\alpha}(x,0)d(x)|x|^{-\alpha}\, \dd x &\leq \xl \tau(A) + \int_{A_\xl(0)} \tilde H_{\alpha}(x,0)d(x)|x|^{-\alpha}\, \dd x \\
& = \xl \tau(A) + \xl m_\xl+\int_\xl^\infty m_s \dd s \\
& \lesssim \xl \tau(A) + \xl m_\xl+\int _\xl^\infty (s^{-1}\ln (e+s))^{\frac{N-\alpha}{\alpha}} \dd s \\
& \lesssim \xl \tau(A) + \xl^{1-\frac{N-\alpha}{\alpha}}(\ln (e+\xl))^{\frac{N-\alpha}{\xa}}.
\eal
Taking $\xl=\tau(A)^{-\frac{\xa}{N-\xa}}\ln(e+\tau(A)^{-1})$,  we obtain
\bal
\int_A \tilde H_{\alpha}(x,0)\varphi_{\alpha,1}\, \dd x \lesssim \tau(A)^{1-\frac{\xa}{N-\xa}}\ln \left(e+\tau(A)^{-\frac{\xa}{N-\xa}}\right).
\eal
Thus estimate \eqref{est-LwSigma0} follows by using \eqref{normLwww}.
\end{proof}

\begin{remark}\label{critical}  Conversely, if we assume that
\ba\label{55}
\| \mathbb{\tilde H}_{\alpha}[\rho \delta_0] \|_{\tilde L_w^{\frac{N-\xa}{\xa}}(\Omega \setminus \{0\};\varphi_{\alpha,1})} \lesssim |\xr|
\ea
for some $\xr\in[-1,1] \setminus \{0\}$ then \eqref{54} holds. Indeed, we assume that \eqref{55} is valid. Then by \eqref{55}, we have
\ba\label{56} \nonumber
\xl \tilde m_\xl^{\frac{\xa}{N-\xa}}\ln \left(e+\tilde m_\xl^{-1}\right)^{-1} &\leq \frac{\int_{\tilde A_\xl(0)}|\mathbb{\tilde H}_{\alpha}[\rho \delta_0](x)|d(x) |x|^{-\alpha}\, \dd x}{\tilde m_\xl^{1-\frac{\xa}{N-\xa}}\ln \left(e+\tilde m_\xl^{-1}\right)} \\
&\leq
\| \mathbb{\tilde H}_{\alpha}[\rho \delta_0] \|_{\tilde L_w^{\frac{N-\xa}{\xa}}(\Omega \setminus \{0\};d(x)|x|^{-\alpha})}\leq C|\xr|,
\ea
where $\tilde A_\lambda(0)$ and $\tilde m_\xl$ have been defined in \eqref{69}. Therefore,
\ba\label{70}
\tilde m_\xl^{-\frac{\xa}{N-\xa}}\ln \left(e+\tilde m_\xl^{-1}\right)\geq C^{-1}\frac{\xl}{|\xr|}.
\ea
Hence, if $\tilde m_\xl<\frac{1}{e}$ we have that
\ba\label{71}
\tilde m_\xl^{-\frac{\xa}{N-\xa}}\ln(\tilde m_\xl^{-1})\geq C_0\frac{\xl}{|\xr|}.
\ea

Now we observe that if $r\in(0,1)$ and $s>e$ then
\ba
r^{-1}\ln (r^{-1})>s\Longrightarrow r\leq s^{-1}\ln s.\label{57}
\ea
Taking $r=\tilde m_\xl^{\frac{\xa}{N-\xa}}$ and $s=C_1(\xa,N)\frac{\xl}{|\xr|}$ in \eqref{57} yields
$
\tilde m_\lambda \lesssim \lambda^{-\frac{N-\alpha}{\alpha}} \bigg(|\rho|\ln\frac{\xl}{|\xr|}\bigg)^{\frac{N-\alpha}{\alpha}},
$
which implies \eqref{54}.
\end{remark}

\subsection{Weak Lebesgue estimate on Green kernel}

In this subsection, we will use the results of the previous subsection to establish estimates of the Green kernel. Let $\varphi_{\alpha,\gamma}$ be as in \eqref{varphia}. For a measure $\tau$ on $\Omega \setminus \Sigma$, the Green operator acting on $\tau$ is
\bal
\BBG_\mu[\tau](x)=\int_{\xO \setminus \Sigma}G_{\mu}(x,y)\dd\tau(y).
\eal

\begin{theorem} \label{lpweakgreen}
Assume $k \geq 0$, $0<\mu \leq H^2$ and $0\leq\xg\leq1$.
Then
\ba \label{estgreen}
	\norm{\BBG_\mu[\gt]}_{L_w^{\frac{N+\xg}{N+\xg-2}}(\Gw\setminus \Sigma;\ei)} \lesssim \norm{\gt}_{\mathfrak{M}(\xO\setminus \Sigma;\varphi_{\am,\xg})}, \quad \forall \tau\in \mathfrak{M}(\xO\setminus \Sigma;\varphi_{\am,\xg}).
\ea
	The implicit constant depends on $N,\Omega,\Sigma,\mu,\xg$.
\end{theorem}

\begin{proof} Without loss of generality we may assume that $\tau$ is nonnegative. We consider the following cases.
	
\noindent \textbf{Case 1: $0<\mu<\left( \frac{N-2}{2} \right)^2$.} Then $0<\am< \frac{N-2}{2}$.  From \eqref{eigenfunctionestimates}, \eqref{Greenesta}, \eqref{F1} and the fact that $d_\Sigma(y) \leq |x-y|+d_\Sigma(x)$, we  obtain, for all $x,y\in \xO\setminus \Sigma, x\neq y$,
\bal
G_\mu(x,y)\varphi_{\am,\xg}(y)^{-1} &\lesssim |x-y|^{2-N} \min \left\{ 1, \frac{d(x)d(y)}{|x-y|^2} \right\} (|x-y|+d_\Sigma(x))^{2\am} d_\Sigma(x)^{-\am}d(y)^{-\xg} \\
&\lesssim F_{-\am,2\am,\xg}(x,y)+F_{\am,0,\xg}(x,y).
\eal
This, together with Lemmas \ref{anisotitaweakF1a}--\ref{anisotitaweakF1b} , estimate  $\varphi_{\am,1} \approx \phi_\mu$ and the fact that (see \eqref{p1})
\bal \frac{N+\xg}{N+\xg-2} = \p_{\am,0,\gamma} \leq \p_{-\am,2\am,\gamma},
\eal
implies \eqref{estgreen}. \medskip

\noindent \textbf{Case 2:} $k=0$, $\Sigma=\{0\}$ and $\mu=\left( \frac{N-2}{2} \right)^2$. Then $\am=\frac{N-2}{2}$.
From \eqref{eigenfunctionestimates}, \eqref{Greenestb} and the fact that $|y| \leq |x-y|+|x|$, we obtain, for all  $x,y\in \xO\setminus \{0\}, x\neq y$,
\bal
G_{(\frac{N-2}{2})^2}(x,y)\varphi_{\frac{N-2}{2},\xg}(y)^{-1}\lesssim F_{-\frac{N-2}{2},N-2,\xg}(x,y)+F_{\frac{N-2}{2},0,\xg}(x,y)+\tilde F_{\xg}(x,y),
\eal
This, together with
Lemmas \ref{anisotitaweakF1a}--\ref{anisotitaweakF4} and the fact that (see \eqref{p1})
\bal \frac{N+\xg}{N+\xg-2} = \p_{-\frac{N-2}{2},N-2,\gamma} \leq \p_{\frac{N-2}{2},0,\gamma},
\eal
implies \eqref{estgreen}. The proof is complete.
\end{proof}

\begin{remark} Assume $0<\mu\leq H^2$. By combining \eqref{estgreen} with $\gamma=1$, \eqref{eigenfunctionestimates} and the embedding after \eqref{normLwww}, we derive that for any $1<p<\frac{N+1}{N-1}$,
	\be \label{Gphi_mu}
	\sup_{z \in \Omega \setminus \Sigma} \int_{\Omega \setminus \Sigma} \left( \frac{G_\mu(x,z)}{d(z)d_\Sigma(z)^{-\am}} \right)^p d(x)d_\Sigma(x)^{-\am}dx < C.
	\ee	
\end{remark}

Next we treat the case $\mu \leq 0$.
\begin{theorem} \label{lpweakgreen2}
Assume  $0\leq\xg\leq1,$ $\xm\leq 0$ and $0\leq\kappa\leq -\am$. Let
\bal
p_{\kappa,\gamma}:=\min\left\{\frac{N+\kappa}{N+\kappa-2}, \frac{N+\xg}{N+\xg-2}\right\}.
\eal
Then
\ba \label{estgreen2}
	\norm{\BBG_\mu[\gt]}_{L_w^{p_{\kappa,\gamma}}(\Gw\setminus \Sigma;\ei)} \lesssim \norm{\gt}_{\mathfrak{M}(\xO\setminus \Sigma;\varphi_{-\kappa,\xg})}, \quad \forall \tau\in \mathfrak{M}^+(\xO\setminus \Sigma;\varphi_{-\kappa,\xg}).
\ea
	The implicit constant depends on $N,\Omega,\Sigma,\mu,\kappa,\gamma$.
\end{theorem}
\begin{proof}
\noindent For $y \in \Omega \setminus \Sigma$ and $\lambda>0$, set
\bal
A_\xl(y):=\Big\{x\in(\xO\setminus \Sigma)\setminus\{y\}:\;\; G_{\mu}(x,y)\varphi_{-\kappa,\xg}(y)^{-1}>\xl \Big \} \quad \text{and} \quad
m_{\xl}(y):=\int_{A_\xl(y)}d(x)d_\Sigma(x)^{-\am} \dd x.
\eal
Put
\bal
 F(x,y):=d_\Sigma(y)^{-\kappa}|x-y|^{-N+2} d(y)^{-\xg} \left( 1 \land \frac{d(x)d(y)}{|x-y|^2} \right) \left(1 \land \frac{d_\Sigma(x)d_\Sigma(y)}{|x-y|^2} \right)^{-\am}, \;\; x,y \in \Omega \setminus \Sigma, x \neq y.
\eal
By \eqref{Greenesta} and \eqref{eigenfunctionestimates}, $F(x,y) \geq c\,G_{\mu}(x,y)\varphi_{-\kappa,\gamma}(y)^{-1}$ for some positive constant $c$ depending only on $N,\xO,\Sigma,\mu.$
Consequently,
\bal
&A_\lambda(y)\subset \Big\{x\in(\xO\setminus \Sigma)\setminus\{y\}:\; F(x,y)>c \lambda  \Big \}=:\tilde A_\lambda(y).
\eal
Let $\beta_0$ be as in Subsection \ref{assumptionK}. We write
\ba\nonumber
m_{\xl}(y)&=\int_{A_\xl(y)\cap \Sigma_{\xb_0}}d(x)d_\Sigma(x)^{-\am} \dd x+\int_{A_\xl(y)\setminus \Sigma_{\xb_0}}d(x)d_\Sigma(x)^{-\am} \dd x.\\
&\lesssim \int_{\tilde A_\xl(y)\cap \Sigma_{\xb_0}}d_\Sigma(x)^{-\am} \dd x+\int_{\tilde A_\xl(y)\setminus \Sigma_{\xb_0}}d(x)\dd x.
\label{ml1gr}
\ea
Note that, for $\Gamma=\partial \xO$ or $\xS$, we have
\be \label{min2}
  1 \land \frac{d_\Gamma(x)d_\Gamma(y)}{|x-y|^2} \leq 2 \left( 1 \land \frac{d_\Gamma(y)}{|x-y|} \right) \leq 4\frac{d_\Gamma(y)}{d_\Gamma(x)}.
 \ee
By \eqref{min2} we have
\ba
\int_{\tilde A_\xl(y)\setminus \Sigma_{\xb_0}}d(x) \dd x \lesssim \int_{\{|x-y|\leq c\lambda^{-\frac{1}{N+\xg-2}}\}}\xl^{-1}|x-y|^{-N+2} \dd x \lesssim \lambda^{-\frac{N+\xg}{N+\xg-2}} \label{ml2gr}
\ea
and
\ba
\int_{\tilde A_\xl(y)\cap \Sigma_{\xb_0}}d_\Sigma(x)^{-\am} \dd x \lesssim \int_{\{|x-y|\leq c \lambda^{-\frac{1}{N+\kappa-2}}\}}\xl^{-1}|x-y|^{-N+2} \dd x \lesssim \lambda^{-\frac{N+\kappa}{N+\kappa-2}}.\label{ml3gr}
\ea
Combining \eqref{ml1gr}, \eqref{ml2gr} and \eqref{ml3gr}, we obtain
\bel{ml4gr}
 m_{\xl}(y)\leq C \lambda^{-p_{\kappa,\xg}}
\ee
for all $\lambda \geq 1$, where $C=C(N,\Omega,\Sigma,\mu)$. Then we can show that \eqref{ml4gr} holds for every $\lambda>0$.

By applying Proposition \ref{bvivier}  with $\mathcal{H}(x,y)=G_{\mu}(x,y)\varphi_{-\kappa,\gamma}(y)^{-1}, $  $\tilde D=D=\xO\setminus \Sigma$, $\eta=\ei$ and $\omega=\varphi_{-\kappa,\xg}\tau$, we obtain \eqref{estgreen2}. The proof is complete.
\end{proof}
\subsection{Weak $L^p$ estimates on Martin kernel}
Recall that
\bal
\mathbb{K}_\mu[\gn](x)=\int_{\partial\xO \cup \Sigma}K_{\mu}(x,y) \dd\xn(y), \quad \gn\in \mathfrak{M}(\partial\xO\cup \Sigma).
\eal
\begin{theorem}\label{lpweakmartin1} ~~
	
{\sc I.} Assume $\mu \leq \left( \frac{N-2}{2}\right)^2$ and $\gn\in \mathfrak{M}(\partial\xO\cup \Sigma)$ with compact support in $\partial\xO.$ Then
\ba \label{estmartin1}
	\norm{\mathbb{K}_\mu[\nu]}_{L_w^{\frac{N+1}{N-1}}(\Gw\setminus \Sigma;\ei)} \lesssim \|\nu\|_{\mathfrak{M}(\partial\Omega \cup \Sigma)}.
\ea
	
{\sc II.} Assume $\gn\in \mathfrak{M}(\partial\xO\cup \Sigma)$ with compact support in $\Sigma$.

(i) If $\mu < \left( \frac{N-2}{2} \right)^2$ then
\ba \label{estmartin2}
\norm{\mathbb{K}_{\mu}[\nu]}_{L_w^{\frac{N-\am}{N-\am-2}}(\Gw\setminus \Sigma;\ei)} \lesssim \norm{\nu}_{\mathfrak{M}(\partial \Omega \cup \Sigma)}.
\ea

(ii) If $k=0$, $\Sigma=\{0\}$ and $\mu = \left( \frac{N-2}{2} \right)^2$ then
\ba \label{estmartin2cr}
\norm{\mathbb{K}_{\mu}[\nu]}_{\tilde L_w^{\frac{N+2}{N-2}}(\Gw\setminus \{0\};\ei )} \lesssim \norm{\nu}_{\mathfrak{M}(\partial \Omega \cup \Sigma)}.
\ea
The implicit constants in the above estimates depends on $N,\Omega,\Sigma,\mu$.
\end{theorem}
\begin{proof}
I. By applying Theorem \ref{H} (i) with $\alpha=-\am$, $\theta=0$ and noting that $K_\mu(x,y) \approx H_{\am,0}$ (due to \eqref{Martinest1} and \eqref{Halthe}),  $\varphi_{\am,1} \approx \phi_\mu$ (due to\eqref{eigenfunctionestimates}) and $\q_{\am,0}=\frac{N+1}{N-1}$, we obtain \eqref{estmartin1}.

II (i). By applying Theorem \ref{H} (ii) with $\alpha=\am$, $\theta=2+2\am$ and noting that $K_\mu(x,y) \approx H_{\am,2+2\am}$ (due to \eqref{Martinest1} and \eqref{Halthe}),  $\varphi_{\alpha,1} \approx \phi_\mu$ (due to\eqref{eigenfunctionestimates}), we obtain \eqref{estmartin2}.

II (ii). By applying Theorem \ref{LwSigma0}  with $\alpha=\frac{N-2}{2}$, we obtain \eqref{estmartin2cr}.
\end{proof}

\section{Boundary value problem for linear equations} \label{sec:linear}
In this section, we first recall the notion of boundary trace which is defined with respect to harmonic measures related to $L_\mu$. Then we provide the existence, uniqueness and a priori estimates of the solution to the boundary value problem for linear equations. We refer the reader to \cite{GkiNg_linear} for the proofs.
\subsection{Boundary trace} \label{subsec:boundarytrace}
Let $\beta_0$ be the constant in Subsection \ref{assumptionK}. Let $\eta_{\beta_0}$ be a smooth function such that $0 \leq \eta_{\beta_0} \leq 1$, $\eta_{\beta_0}=1$ in $\overline{\Sigma}_{\frac{\xb_0}{4}}$ and $\supp \eta_{\beta_0} \subset \Sigma_{\frac{\beta_0}{2}}$. We define
\bal W(x):=\left\{ \BAL &d_\Sigma(x)^{-\ap} \qquad&&\text{if}\;\mu <H^2, \\
	&d_\Sigma(x)^{-H}|\ln d_\Sigma(x)| \qquad&&\text{if}\;\mu =H^2,
	\EAL \right. \quad x \in \Omega \setminus \Sigma,
\eal
	and
\bal
	\tilde W(x):=1-\eta_{\beta_0}(x)+\eta_{\beta_0}(x)W(x), \quad x \in \Omega \setminus \Sigma.
\eal

Let $z \in \Omega \setminus \Sigma$ and $h\in C(\partial\Omega \cup \Sigma)$ and denote $\CL_{\mu ,z}(h):=v_h(z)$ where $v_h$ is the unique solution of the Dirichlet problem
\be \label{linear} \left\{ \BAL
L_{\mu}v&=0\qquad \text{in}\;\;\xO\setminus \Sigma\\
v&=h\qquad \text{on}\;\;\partial\xO\cup \Sigma.
\EAL \right. \ee
Here the boundary value condition in \eqref{linear} is understood in the sense that
\bal
	\lim_{\dist(x,F)\to 0}\frac{v(x)}{\tilde W(x)}=h \quad \text{for every compact set } \; F\subset \partial \Omega \cup \Sigma.
\eal
The mapping $h\mapsto \CL_{\mu,z}(h)$ is a linear positive functional on $C(\partial\Omega \cup \Sigma)$. Thus there exists a unique Borel measure on $\partial\Omega \cup \Sigma$, called {\it $L_{\mu}$-harmonic measure in $\partial \Omega \cup \Sigma$ relative to $z$} and  denoted by $\omega_{\Omega \setminus \Sigma}^{z}$, such that
\bal
v_{h}(z)=\int_{\partial\Omega\cup \Sigma}h(y) \dd\omega_{\Omega \setminus \Sigma}^{z}(y).
\eal
Let $x_0 \in \Omega \setminus \Sigma$ be a fixed reference point. Let $\{\xO_n\}$ be an increasing sequence of bounded $C^2$ domains  such that
\bal  \overline{\xO}_n\subset \xO_{n+1}, \quad \cup_n\xO_n=\xO, \quad \mathcal{H}^{N-1}(\partial \Omega_n)\to \mathcal{H}^{N-1}(\partial \Omega),
\eal
where $\mathcal{H}^{N-1}$ denotes the $(N-1)$-dimensional Hausdorff measure in $\R^N$.
  Let $\{\Sigma_n\}$ be a decreasing sequence of bounded $C^2$ domains  such that
\bal \Sigma \subset \Sigma_{n+1}\subset\overline{\Sigma}_{n+1}\subset \Sigma_{n}\subset\overline{\Sigma}_{n} \subset\Omega_n, \quad \cap_n \Sigma_n=\Sigma.
\eal
For each $n$, set $O_n=\xO_n\setminus \Sigma_n$  and assume that $x_0 \in O_1$. Such a sequence $\{O_n\}$ will be called a {\it $C^2$ exhaustion} of $\Gw\setminus \Sigma$.

Then $-L_\mu$ is uniformly elliptic and coercive in $H^1_0(O_n)$ and its first eigenvalue $\lambda_\mu^{O_n}$ in $O_n$ is larger than its first eigenvalue $\lambda_\mu$ in $\Omega \setminus \Sigma$.

For $h\in C(\prt O_n)$, the following problem
\bal\left\{ \BAL
-L_{\xm } v&=0\qquad&&\text{in } O_n\\
v&=h\qquad&&\text{on } \prt O_n,
\EAL \right.
\eal
admits a unique solution which allows to define the $L_{\xm }$-harmonic measure $\omega_{O_n}^{x_0}$ on $\prt O_n$
by
\bal
v(x_0)=\myint{\prt O_n}{}h(y) \dd\gw^{x_0}_{O_n}(y).
\eal

Let $G^{O_n}_\xm(x,y)$ be the Green kernel of $-L_\mu$ on $O_n$.  Then $G^{O_n}_\xm(x,y)\uparrow G_\mu(x,y)$ for $x,y\in\xO\setminus \xS, x \neq y$.

We recall below the definition of the boundary trace which is defined in a \textit{dynamic way}.

\begin{definition}[Boundary trace] \label{nomtrace}
	A function $u\in W^{1,\kappa}_{loc}(\xO\setminus\xS),$ for some $\kappa>1,$ possesses a \emph{boundary trace}  if there exists a measure $\nu \in\GTM(\partial \Omega \cup \Sigma)$ such that for any $C^2$ exhaustion  $\{ O_n \}$ of $\Omega \setminus \Sigma$, there  holds
	\ba \label{trab}
	\lim_{n\rightarrow\infty}\int_{ \partial O_n}\phi u\dd\omega_{O_n}^{x_0}=\int_{\partial \Omega \cup \Sigma} \phi \dd\nu \quad\forall \phi \in C(\overline{\Omega}).
	\ea
	The boundary trace of $u$ is denoted by $\tr(u)$.
\end{definition}


\begin{proposition}[{\cite[Proposition 1.5]{GkiNg_linear}}]   \label{traceKG} ~~
	
	(i) For any $\nu \in \GTM(\partial \Omega \cup \Sigma)$, $\tr(\BBK_{\mu}[\nu])=\nu$.
	
	(ii) For any $\tau \in \GTM(\Omega \setminus \Sigma;\ei)$, $\tr(\BBG_\mu[\tau])=0$.
\end{proposition}

The next result is the Representation Theorem.
\begin{theorem}[{\cite[Theorem 1.3]{GkiNg_linear}}]\label{th:Rep} For any $\nu \in \GTM^+(\partial \Omega \cup \Sigma)$, the function $\BBK_{\mu}[\nu]$ is a positive $L_\mu$-harmonic function (i.e. $L_\mu \BBK_{\mu}[\nu]=0$ in the sense of distributions in $\Omega \setminus \Sigma$). Conversely, for any positive $L_\mu$-harmonic function $u$ (i.e. $L_\mu u = 0$ in the sense of distribution in $\Omega \setminus \Sigma$), there exists a unique measure $\nu \in \GTM^+(\partial \Omega \cup \Sigma)$ such that $u=\BBK_{\mu}[\nu]$.
\end{theorem}

Nonnegative $L_\mu$-superharmonic functions can be decomposed in terms of Green kernel and Martin kernel.
\begin{proposition}[{\cite[Theorem 1.6]{GkiNg_linear}}]\label{super} Let $u$ be a nonnegative $L_\xm$-superharmonic function. Then $u\in L^1(\xO;\ei)$ and there exist positive measures $\tau\in\mathfrak{M}^+(\xO\setminus \Sigma;\ei)$ and $\nu \in \mathfrak{M}^+(\partial\xO\cup \Sigma)$ such that
	\bal
	u=\mathbb{G}_{\mu}[\tau]+\mathbb{K}_{\xm}[\xn].
	\eal
\end{proposition}

\begin{proposition}\label{prop}
	Let $\vgf \in L^{1}(\xO;\ei)$, $\vgf \geq 0$ and $\tau\in \mathfrak{M}^+(\xO\setminus \Sigma;\ei).$ Set
	\bal
	w:=\BBG_\xm[\vgf+\gt]\quad\text{and}\quad \psi=\BBG_\xm[\gt].
	\eal
	Let $\xf$ be a concave nondecreasing $C^2$ function on $[0,\infty),$
	such that $\xf(1) \geq 0.$ Then the function $\xf'(w/\psi)\vgf$ belongs to $L^1(\xO;\ei)$ and the following holds in the weak sense in $\Gw\setminus \xS$
	\bal
	-L_\xm(\psi \xf(w/\psi))\geq\xf'(w/\psi)\vgf.
	\eal
\end{proposition}
\begin{proof}
	The proof is the same as that of \cite[Propositions 3.1]{GkNg} and we omit it.
\end{proof}

Similarly we may prove that
\begin{proposition}\label{prop2}
	Let $\vgf \in L^{1}(\xO;\ei)$, $\vgf \geq 0$ and $\nu \in \mathfrak{M}^+(\partial\xO\cup \Sigma).$ Set
	\bal
	w:=\BBG_\xm[\vgf]+\BBK_\xm[\xn]\quad\text{and}\quad \psi=\BBK_\xm[\xn].
	\eal
	Let $\xf$ be a concave nondecreasing $C^2$ function on $[0,\infty),$
	such that $\xf(1) \geq 0.$ Then the function $\xf'(w/\psi)\vgf$ belongs to $L^1(\xO;\ei)$ and the following holds in the weak sense in $\Gw\setminus \xS$
	\bal
	-L_\xm(\psi \xf(w/\psi))\geq\xf'(w/\psi)\vgf.
	\eal
\end{proposition}

\subsection{Boundary value problem for linear equations} \label{subsec:linear}
We recall the definition and properties of weak solutions to the boundary value problem for linear equations.
\begin{definition}
 Let $\tau\in\mathfrak{M}(\xO\setminus \Sigma;\ei)$ and $\nu \in \mathfrak{M}(\partial\xO\cup \Sigma)$. We say that $u$ is a weak solution of
\ba\label{NHL} \left\{ \BAL
- L_\gm u&=\tau\qquad \text{in }\;\Gw\setminus \Sigma,\\
\tr(u)&=\xn,
\EAL \right. \ea
if $u\in L^1(\xO\setminus \Sigma;\ei)$ and it satisfies
\bal
	- \int_{\Gw}u L_{\xm }\zeta \dd x=\int_{\Gw \setminus \Sigma} \zeta \dd\tau - \int_{\Gw} \mathbb{K}_{\xm}[\xn]L_{\xm }\zeta \dd x
	\qquad\forall \zeta \in\mathbf{X}_\xm(\xO\setminus \Sigma),
\eal
	where the space of test function ${\bf X}_\mu(\Gw\setminus \Sigma)$ has been defined in \eqref{Xmu}
\end{definition}

\begin{theorem}[ {\cite[Theorem 1.8]{GkiNg_linear}}] \label{linear-problem}
Let $\tau \in\mathfrak{M}(\xO\setminus \Sigma;\ei)$ and $\xn \in \mathfrak{M}(\partial\xO\cup \Sigma)$.  Then there exists a unique weak solution $u\in L^1(\xO;\ei)$ of \eqref{NHL}. Furthermore
\bal
u=\mathbb{G}_{\mu}[\tau]+\mathbb{K}_{\xm}[\xn]
\eal
and there exists a positive constant $C=C(N,\Omega,\Sigma,\mu)$ such that
\bal
\|u\|_{L^1(\Omega;\ei)} \leq \frac{1}{\lambda_\mu}\| \tau \|_{\GTM(\Omega \setminus \Sigma;\ei)} + C \| \nu \|_{\GTM(\partial\Omega \cup \Sigma)}.
\eal
In addition, if $\dd \tau=f\dd x+\dd\rho$ then, for any $0 \leq \zeta \in \mathbf{X}_\xm(\xO\setminus \Sigma)$, the following estimates are valid
\be\label{poi4}
-\int_{\Gw}|u|L_{\xm }\zeta  \dd x\leq \int_{\Gw}\sign(u)f\zeta \dd x +\int_{\Gw \setminus \Sigma}\zeta  \dd|\rho|-
\int_{\Gw}\mathbb{K}_{\xm}[|\xn|] L_{\xm }\zeta  \dd x,
\ee
\be\label{poi5}
-\int_{\Gw}u^+L_{\xm }\zeta  \dd x\leq \int_{\Gw} \sign^+(u)f\zeta \dd x +\int_{\Gw \setminus \Sigma}\zeta \dd\rho^+-
\int_{\Gw}\mathbb{K}_{\xm}[\nu^+]L_{\xm }\zeta \dd x.
\ee

\end{theorem}



\section{General nonlinearities} \label{sec:gennon}
In this section, we provide various sufficient conditions for the existence of a solution to \eqref{NLP}. Throughout this sections we assume that $g: \R \to \R$ is continuous and nondecreasing and satisfies $g(0)=0$. We start with the following result.

\begin{lemma} \label{subcrcon} Assume
\ba \label{subcd0} \int_1^\infty  s^{-q-1}(\ln s)^{m} (g(s)-g(-s)) \dd s<\infty
\ea
for $q,m \in \R$, $q >0$ and $m \geq 0$. Let $v$ be a function defined in $\Omega \setminus \Sigma$. For $s>0$, set
\bal E_s(v):=\{x\in \xO\setminus \Sigma:| v(x)|>s\} \quad \text{and} \quad e(s):=\int_{E_s(v)} \ei \dd x.
\eal
Assume that there exists a positive constant $C_0$ such that
\ba \label{e}
e(s) \leq C_0s^{-q}(\ln s)^m, \quad \forall s>e^\frac{2 m}{q}.
\ea
Then for any $s_0>e^\frac{2 m}{q},$ there holds
\ba\label{53}
\norm{g(|v|)}_{L^1(\Omega;\ei)}&\leq \int_{(\Omega \setminus \Sigma) \setminus E_{s_0}(v)} g(|v|)\ei \dd x+ C_0 q\int_{s_0}^\infty  s^{-q-1}(\ln s)^{m} g(s) \dd s, \\ \label{53-a}
\norm{g(-|v|)}_{L^1(\Omega;\ei)}&\leq  -\int_{(\Omega \setminus \Sigma) \setminus E_{s_0}(v)} g(-|v|)\ei \dd x-C_0 q\int_{s_0}^\infty  s^{-q-1}(\ln s)^{m} g(-s) \dd s.
\ea
\end{lemma}
\begin{proof}
We note that $g(|v|) \geq g(0)=0$. Let $s_0>1$ to be determined later on. Using the fact that $g$ is nondecreasing, we obtain
\bal \begin{aligned}
\int_{\Omega \setminus \Sigma} g(|v|)\ei dx &\leq \int_{(\Omega \setminus \Sigma) \setminus E_{s_0}(v)} g(|v|)\ei \dd x + \int_{E_{s_0}(v)} g(|v|)\ei \dd x \\
&\leq g(s_0)e(s_0) -\int_{s_0}^{\infty}g(s) \dd e(s).
\end{aligned} \eal
From \eqref{subcd0}, we deduce that there exists an increasing sequence $\{T_n\}$ such that
\ba \label{Tn}
\lim_{T_n \to \infty}T_n^{-q}(\ln T_n)^m g(T_n) = 0.
\ea
For $T_n > s_0$, we have
\ba \label{gTn} \begin{aligned}
-\int_{s_0}^{T_n}g(s)\dd e(s) &= -g(T_n)e(T_n) + g(s_0)e(s_0) + \int_{s_0}^{T_n} e(s)\dd g(s) \\
&\leq -g(T_n)e(T_n) + g(s_0)e(s_0)+C_0\int_{s_0}^{T_n}s^{-q}(\ln s)^m \dd g(s)\\
&\leq (CT_n^{-q}(\ln T_n)^m - e(T_n))g(T_n)   - C_0\int_{s_0}^{T_n} (s^{-q}(\ln s)^m)' g(s)\dd s.
\end{aligned} \ea
Here in the last estimate, we have used \eqref{e}.
Note that if we choose $s_0>e^{\frac{2 m}{q}}$ then
\ba \label{deriv} -q s^{-q-1}(\ln s)^m  < (s^{-q}(\ln s)^m)' < -\frac{q}{2}s^{-q-1}(\ln s)^m \quad \forall s \geq s_0.
\ea
Combining \eqref{Tn}--\eqref{deriv} and then letting $n \to \infty$, we obtain
\bal
-\int_{s_0}^{\infty}g(s) \dd e(s) < C_0q \int_{s_0}^{\infty} s^{-q-1}(\ln s)^m g(s) \dd s.
\eal
Thus we have proved estimate \eqref{53}. By  applying estimate \eqref{53} with $g$ replaced by $h(t)=-g(-t)$, we obtain \eqref{53-a}.
\end{proof}

\begin{lemma} \label{transf}
Let $0<\xm\leq H^2$, $0\leq\xg\leq1$, $\tau \in \GTM^+(\Omega \setminus \Sigma; \varphi_{\am,\gamma})$ with
$\norm{\tau}_{\mathfrak{M}(\Omega\setminus \Sigma;\varphi_{\am,\gamma})}=1,
$
and $\nu \in \GTM^+(\prt \Gw\cup \xS)$ with
$\norm{\nu}_{\GTM(\prt \Gw\cup \xS)}=1$.
Assume $g \in L^\infty(\R) \cap C(\R)$ satisfies
\ba \label{subcd-00} \Gl_g=\int_1^\infty  s^{-q-1}  (g(s)-g(-s))\,\dd s < \infty
\ea	
for some $q \in (1,\infty)$ and
\bal
|g(s)|\leq a|s|^{\tilde q} \quad \text{for some } a>0,\; \tilde q>1 \text{ and for any } |s|\leq 1.
\eal

Assume one of the following conditions holds.

 (i) $ \1_{\partial \Omega}\nu \equiv0$ and \eqref{subcd-00} holds for $q= \frac{N+\gamma}{N+\gamma-2}$.

 (ii) $ \1_{\partial \Omega}\nu \not\equiv 0$  and \eqref{subcd-00} holds for $q=\frac{N+1}{N-1}$.

Then there exist positive numbers $\xr_0,\xs_0,t_0$ depending on $N,\Omega,\mu,\Gl_g,\xg, \tilde q$ such that for every $\xr \in (0,\xr_0)$ and $\xs\in (0,\xs_0)$ the following problem
	\ba\label{trans} \left\{ \BAL -L_\mu v &= g (v + \xr\BBG_\mu[\tau] + \xs\BBK_\mu[\nu]) \quad \text{in } \Gw\setminus\xS, \\ \tr(v) &= 0
	\EAL \right. \ea
	admits a positive weak solution $v$ satisfying
	\bal  \|v\|_{L_w^{q}(\Gw\setminus\xS;\ei)} \leq t_0
	\eal
	where $q=\frac{N+\gamma}{N+\gamma-2}$ if case (i) happens or $q=\frac{N+1}{N-1}$ if case (ii) happens.
\end{lemma}
\begin{proof}
We shall use Schauder fixed point theorem to show the existence of a positive weak solution of \eqref{trans}. \medskip

(i) Assume that $\1_{\partial \Omega} \nu \equiv 0$, namely $\nu$ has compact support in $\Sigma$, and \eqref{subcd-00} holds for $q= \frac{N+\gamma}{N+\gamma-2}$ (in the proof of statement (i), we always assume that $q= \frac{N+\gamma}{N+\gamma-2}$). \medskip

\noindent \textbf{Step 1:} Derivation of $t_0$, $\rho_0$ and $\sigma_0$.

Define the operator $\BBS$ by
	\bal \BBS(v):=\BBG_\mu[g(v + \xs\BBK_\mu[\nu]+\xr\BBG_\mu[\tau])] \quad \text{for }  v \in L_w^{q}(\Gw\setminus \Sigma;\ei).
	\eal
Fix $1<\kappa<\min\{q, \tilde q\},$ and put
	\bal \BAL Q_1(v) &: = \| v\|_{L_w^{q}(\Gw\setminus \Sigma;\ei)}, \quad && v \in L_w^{q}(\Gw\setminus \Sigma;\ei), \\
	Q_2(v) &:=\| v\|_{L^{\xk}(\Gw;\ei)}, \quad && v \in L^{\xk}(\Gw;\ei), \\
	Q(v)&:=Q_1(v) + Q_2(v), \quad && v \in L_w^{q}(\Gw\setminus \Sigma;\ei).
	\EAL \eal

Let $v \in L_w^{q}(\Gw\setminus \Sigma;\ei)$. For $s>0$, set
\bal E_s:=\{x\in\xO: |v(x) + \xr\BBG_\mu[\tau](x) + \xs\BBK_\mu[\nu](x)|>s\} \quad \text{and} \quad e(s):=\int_{E_s}\phi_\mu \, \dx.
\eal
By  Theorem \ref{lpweakgreen}, $\BBG_\mu[\tau] \in L_w^{q}(\Omega \setminus \Sigma;\ei)$ and
\ba \label{Q1G}  \| \BBG_\mu[\tau] \|_{L_w^{q}(\Omega \setminus \Sigma;\ei)} \lesssim \| \tau \|_{\GTM(\Omega \setminus \Sigma;\varphi_{\am,\gamma})} = 1.
\ea
By Theorem \ref{lpweakmartin1} II. (i), $\BBK_\mu[\nu] \in L_w^{\frac{N-\am}{N-\am-2}}(\Omega \setminus \Sigma;\ei)$ and
\ba \label{Q1K} \| \BBK_\mu[\nu] \|_{L_w^{\frac{N-\am}{N-\am-2}}(\Omega \setminus \Sigma;\ei)} \lesssim \| \nu \|_{\GTM(\partial \Omega \cup \Sigma)} = 1.
\ea
From \eqref{ue}, \eqref{Q1G}, \eqref{Q1K} and noting that $q \leq \frac{N-\am}{N-\am-2}$ since $\am>0$, we deduce
\ba
e(s) \leq s^{-q}\| v+\xr\BBG_\mu[\tau]+ \xs\BBK_\mu[\nu]\|_{L_w^{q}(\Gw\setminus \Sigma;\ei)}^{q} \leq Cs^{-q}(\| v\|_{L_w^{q}(\Gw\setminus \Sigma;\ei)}^{q}+\xr^{q} +\xs^{q}).\label{40}
\ea

By \eqref{53} and \eqref{53-a}, taking into account \eqref{40} and the assumptions on $g$, we have
\ba\nonumber
&\int_\xO |g(v +\xr\BBG_\mu[\tau] + \xs\BBK_\mu[\nu])|\ei \,\dx \\ \nonumber
&= \int_{\Omega \setminus E_1}|g(v +\xr\BBG_\mu[\tau] + \xs\BBK_\mu[\nu])| \ei \, \dx + \int_{E_1}|g(v +\xr\BBG_\mu[\tau] + \xs\BBK_\mu[\nu])| \ei \, \dx \\ \nonumber
&\leq C(q,\Lambda_g)\bigg( \int_{\Omega \setminus E_1} |v +\xr\BBG_\mu[\tau] + \xs\BBK_\mu[\nu]|^{\tilde q} \ei \dx + \| v\|_{L_w^{q}(\Gw\setminus \Sigma;\ei)}^{q}+\xr^{q} +\xs^{q} \bigg) \\ \nonumber
&\leq C(Q_1(v)^{q}+ Q_2(v)^{\xk} +\xr^\xk Q_2(\BBG_\mu[\tau])^{\xk} + \sigma^\kappa Q_2(\BBK_\mu[\nu])^\xk+\xr^{q} +\xs^{q}) \\ \nonumber
&\leq C(Q_1(v)^{q}+ Q_2(v)^{\xk} + \xr^\xk + \xs^\xk  + \xr^{q} +\xs^{q}).
\ea
It follows that
\ba \label{Q1Q2-1} \BAL
Q_1(\BBS(v))+ Q_2(\BBS(v)) &\leq C\| g(v +\xr\BBG_\mu[\tau] + \xs\BBK_\mu[\nu])\|_{L^{1}(\Gw\setminus \Sigma;\ei)}\\
&\leq C(Q_1(v)^{q}+ Q_2(v)^{\xk} + \xr^\xk + \xs^\xk  + \xr^{q} +\xs^{q}),
\EAL \ea
where $C$ depends only on $\mu,\Gw,\Sigma,\Gl_g,a, \tilde q,\xg.$

Therefore if $Q(v) \leq t$ then
\ba \label{Qt0} Q(\BBS(v)) \leq C\left(t^q+t^{\xk} +\xr^\xk + \xs^\xk + \xr^{q} +\xs^{q}\right).
\ea
	Since $q>\kappa>1$, there exist positive number $t_0$, $\rho_0$ and $\xs_0$ depending on $\mu,\Gw,\Sigma,\Gl_g, \tilde q, \xg, \xk$ such that for any
$t\in(0,t_0)$, $\rho \in (0,\rho_0)$  and $\xs\in (0,\xs_0),$ the following inequality holds
\bal
C\left(t^{q}+t^{\xk}+ \xs^\xk +\xr^\xk + \xr^{q} +\xs^{q}\right) \leq t_0,
\eal
where $C$ is the constant in \eqref{Qt0}.
 Therefore,
	\bel{ul11} Q(v) \leq t_0  \Longrightarrow Q(\BBS(v)) \leq t_0.
	\ee

\noindent \textbf{Step 2:} We apply Schauder fixed point theorem to our setting.

	\textit{We claim that $\BBS$ is continuous}. Indeed, if $u_n\rightarrow u$ in $L^1(\Gw;\ei)$ then since $ g \in L^\infty(\BBR) \cap C(\R)$ and is nondecreasing, it follows that $g(v_n +\xr\BBG_\mu[\tau] + \xs\BBK_\mu[\nu]) \to g(v +\xr\BBG_\mu[\tau] + \xs\BBK_\mu[\nu])$ in $L^1(\Gw;\ei).$ Hence, $\BBS(u_n)\to\BBS(u)$ in $L^1(\Gw;\ei).$
	
	\textit{Next we claim that $\BBS$ is compact}. Indeed, set
	$M:=\sup_{t>0}|g(t)|<+\infty$. Then we can easily deduce that there exists $C=C(\xO,\xS,M,\mu)$ such that
	\be\label{sup1}
	|\BBS(w)|\leq C\ei \quad \text{a.e. in } \xO, \quad \forall w\in L^1(\Gw;\ei).
	\ee

Also, by Theorem \eqref{linear-problem}, $-L_\xm\BBS(w)=g(w +\xr\BBG_\mu[\tau] + \xs\BBK_\mu[\nu])$ in the sense of distributions in $\xO\setminus \xS$. By \cite[Corollary 1.2.3]{MVbook}, $\BBS(w)\in W^{1,r}_{loc}(\xO\setminus \Sigma),$ for any $1<r<\frac{N}{N-1}$ and for any open $D\Subset \xO\setminus \xS,$ there exists $C_1=C_1(\xO,\xS,M,\mu,D,p)$ such that
\ba \label{W1pest} \|\BBS(w)\|_{W^{1,r}(D)}\leq C_1(D).
\ea

	Let $\{v_n\}$ be a sequence in $ L^1(\Gw;\ei)$ then by \eqref{sup1} and \eqref{W1pest}, there exist $\psi$ and a subsequence still denoted by $\{\BBS(v_n)\}$ such that $\BBS(v_n)\rightarrow \psi$ a.e. in $\xO\setminus \Sigma$. In addition, by \eqref{sup1} and the dominated convergence theorem we obtain that $\BBS(v_n)\rightarrow \psi$ in $L^1(\Gw;\ei).$
	
	Now set
	\be
	\CO:=\{ v \in L^1(\Gw;\ei): Q(v) \leq t_0  \}.\label{O}
	\ee
	Then $\CO$ is a closed, convex subset of $L^1(\Gw;\ei)$ and by \eqref{ul11}, $\BBS(\CO) \sbs \CO$.
	Thus we can apply Schauder fixed point theorem to obtain the existence of a function $v \in \CO$ such that $\BBS(v)=v$. This means that $v$ is a nonnegative solution of \eqref{trans} and hence there holds
	\bal  -\int_{\Gw}v L_\gm\zeta \,\dx= \int_{\Gw}  g(v +\xr\BBG_\mu[\tau] + \xs\BBK_\mu[\nu]) \zeta \,\dx  \forevery \zeta \in {\bf X}_\mu(\Gw\setminus \Sigma). \eal
	\medskip

(ii) The case $\1_{\partial \Omega} \nu \not \equiv 0$ and \eqref{subcd-00} holds for with $q=\frac{N+1}{N-1}$ ($\leq \frac{N+\gamma}{N+\gamma-2}$) can be proceeded similarly as case (i) with minor modifications and hence we omit it.
\end{proof}

\begin{proof}[\textbf{Proof of Theorem \ref{th1}}.]
	
(i) We assume that $\1_{\partial \Omega} \nu \equiv 0$, namely $\nu$ has compact support in $\Sigma$, and \eqref{subcd-00} holds for $q= \frac{N+\gamma}{N+\gamma-2}$ (in the proof of statement (i) we always assume $q=\frac{N+\gamma}{N+\gamma-2}$).

Let $0\leq\eta_n(t)\leq 1$ be a smooth function in $\mathbb{R}$ such that $\eta_n(t)=1$ for any $|t|\leq n$ and $\eta_n(t)=0$ for any $|t| \geq n+1$. Set $g_n=\eta_n g$ then $g_n \in L^\infty(\BBR) \cap C(\R)$ is a nondecreasing function in $\BBR$. Moreover $g_n$ satisfies \eqref{subcd-0} for $q=\frac{N+\gamma}{N+\gamma-2}$ and $\Lambda_{g_n} \leq \Lambda_g$. Furthermore,  $|g_n(s)|\leq a|s|^{\tilde q}$ for any $|s|\leq 1$ with the same constants $a>0,\; \tilde q>1$. Therefore the constants $\rho_0,\xs_0,t_0$ in Lemma \ref{transf} can be chosen to depend on $\mu,\Gw,\Sigma,\Gl_g,p,a, \tilde q,\xg$, but independent of $n$. By Lemma \ref{transf}, for any $\rho \in (0,\rho_0)$ and $\xs\in (0,\xs_0)$ and $n \in \BBN$, there exists a solution $v_n \in \CO$ (where $\CO$ is defined in \eqref{O}) of
	\bal \left\{ \BAL -L_\mu v_n &= g_n (v_n +\xr\BBG_\mu[\tau] + \xs\BBK_\mu[\nu]) \quad \text{in } \Gw\setminus \xS, \\ \tr(v_n) &= 0.
	\EAL \right. \eal
	Set $u_n=v_n +\xr\BBG_\mu[\tau] + \xs\BBK_\mu[\nu]$ then $\tr(u_n)=\xs \nu$ and
	\bel{ul13} -\int_{\Gw} u_n L_\gm\zeta \dd x= \int_{\Gw} g_n(u_n) \zeta \dd x + \rho \int_{\Omega}\zeta \dd \tau -
	\xs \int_{\Gw} \BBK_\gm[\gn]L_\gm \zeta \dd x \forevery \zeta \in {\bf X}_\mu(\Gw\setminus \xS). \ee
	
Since $\{v_n\} \sbs \CO$, the sequence $\{ u_n \}$ is uniformly bounded in $L^{\xk}(\xO;\ei)\cap L^{q}_w(\xO\setminus \Sigma;\ei)$. More precisely,
\be \label{42}
\|u_n\|_{L_w^{q}(\Gw\setminus\xS;\ei)}+\|u_n\|_{L^{\xk}(\Gw;\ei)} \leq t_0 \quad\forall n\in \mathbb{N}.
\ee

In view of the proof of \eqref{53}, for any Borel set $E\subset \xO\setminus \Sigma$ and $s_0>1,$ we have
\ba\label{43}
\int_E|g_n(u_n)|\ei \dd x&\leq (g(s_0)-g(-s_0))\int_{E}\ei \dd x + Ct_0^{q}\int_{s_0}^\infty  s^{-q-1}( g(s)-g(s)) \dd s,
\ea
which implies that $\{g_n(u_n)\}$ is equi-integrable in $L^1(\Gw;\ei).$

Also, by Theorem \ref{linear-problem}, $-L_\xm u_n=g_n(u_n)+\xr\tau$ in the sense of distribution in $\xO\setminus \xS.$ By \cite[Corollary 1.2.3]{MVbook} and \eqref{42}, $u_n\in W^{1,r}_{loc}(\xO\setminus \xS),$ for any $1<r<\frac{N}{N-1}$ and for any open $D\Subset \xO\setminus \xS,$ there exists $C_2=C_2(\xO,\xS,M,\mu,D,p,t_0)$ such that $\|u_n\|_{W^{1,r}(D)}\leq C_2$.
Hence there exists a subsequence still denoted by $\{u_n\}$ such that $u_n\rightarrow u$ a.e. in $\xO\setminus \xS.$ In additions by \eqref{42} and \eqref{43}, we may invoke Vitali's convergence theorem to derive that $u_n \to u$  and $g_n(u_n)\to g(u)$ in $L^1(\Gw;\ei).$ Thus, by letting $n \to \infty$ in \eqref{ul13}, we derive that $u$ satisfies \eqref{lweakform}, namely $u$ is a positive solution of \eqref{NLP}. The proof is complete. \medskip

(ii) The case $\1_{\partial \Omega}\nu \not \equiv 0$ and \eqref{subcd-00} holds for with $q=\frac{N+1}{N-1}$ can be proceeded similarly as in case (i) with minor modification and hence we omit it.
\end{proof}

 \begin{proof}[\textbf{Proof of Theorem \ref{th2}}.]
The proof of statements (i), (ii) and (iv) is similar to that of Theorem \ref{th1} and we omit it. As for the proof of statement (iii), the point that needs to be paid attention is the use of Theorem \ref{lpweakgreen} (for $\mu>0$) and Theorem \ref{lpweakgreen2} (for $\mu \leq 0$) for $Q_1(\BBS(v))$ as in the first estimate in \eqref{Q1Q2-1}. In particular, for $\mu \leq H^2$, the estimate
\bal
\|\BBS(v)\|_{L_w^q(\Omega \setminus;\ei)} \leq C\| g(v + \xs\BBK_\mu[\nu])\|_{L^{1}(\Gw\setminus \Sigma;\ei)}
\eal
is valid for $q=\min\left\{\frac{N+1}{N-1},\frac{N-\am}{N-\am-2}\right\}$. The rest of the proof of statement (iii) can be proceeded as in the proof of Lemma \ref{transf} and of Theorem \ref{th1} and is left to the reader. 	
\end{proof}

\section{Power case } \label{sec:powercase}
In this section we study the following problem
\ba\label{power} \left\{ \BAL
- L_\gm u&= |u|^{p-1}u +\xr\tau\qquad \text{in }\;\Gw\setminus \Sigma,\\
\tr(u)&=\xs\xn.
\EAL \right. \ea
where $p>1$, $\rho \geq 0$, $\sigma \geq 0$, $\tau\in\mathfrak{M}^+(\xO\setminus \Sigma;\ei)$ and $\nu \in \mathfrak{M}^+(\partial\xO\cup \Sigma)$.

\subsection{Partial existence results}
We provide below necessary and sufficient conditions expressed in terms of Green kernel and Martin kernel for the existence of a solution to \eqref{power}.
\begin{proposition}\label{equivint}
Assume $\mu \leq H^2$, $p>1$ and $\tau\in\mathfrak{M}^+(\xO\setminus \Sigma;\ei)$. Then problem \eqref{power} with $\xn=0$ admits a nonnegative solution  and for some $\xs>0$ if and only if there is a constant $C>0$ such that
\ba\label{58}
\BBG_\xm[\BBG_\xm[\tau]^p]\leq C\,\BBG_\xm[\tau] \quad \text{a.e. in } \Omega\setminus\xS.
\ea
\end{proposition}
\begin{proof}
If \eqref{58} holds then the existence of a positive solution to problem \eqref{power} with $\xn=0$ follows by a rather similar argument as in the proof of \cite[Proposition 3.5]{GkNg}.

So we will only show that if $u$ is a positive solution of \eqref{power} with $\xn=0$ for some $\xs>0$ then \eqref{58} holds. We adapt the argument used in the proof of \cite[Proposition 3.5]{BY}. We may suppose that $\xs=1.$ By the assumption, we have $u=\BBG_\xm[u^p+\gt]$. By applying Proposition \ref{prop} with $\vgf$ replaced by $ u^p,$ $w$ by $u$ and with
\bal
\gf(s)= \left\{ \BAL &(1-s^{1-p})/(p-1) \quad &\text{if } s \geq 1, \\
&s-1 &\text{if } s<1,
\EAL \right.
\eal
we obtain
\ba\label{59}
-L_\xm(\psi \xf(u/\psi))\geq\xf'(u/\psi)u^p=\BBG_\xm[\tau]^p,
\ea
in the weak sense.
Since $u \geq \BBG_\mu[\tau]=\psi$, we see that
\ba \label{u/psi}
\psi \xf(u/\psi)\leq \frac{1}{p-1}\BBG_\xm[\tau],
\ea
 which, together with Proposition \ref{traceKG}, implies $\tr(\psi \xf(u/\psi))=0$. By \eqref{59} and Proposition \ref{super} there exist $\xn\in \mathfrak{M}^+(\partial\xO\cup \Sigma)$ and $\tilde \tau \in \mathfrak{M}^+(\xO\setminus \Sigma;\ei)$ such that $\dd \tilde \tau \geq \BBG_\xm^p[\tau] \dd x$ and
\be \label{tildetau}
\psi \xf(u/\psi)=\mathbb{G}_{\mu}[\tilde \tau]+\mathbb{K}_{\xm}[\xn].
\ee
Since $\tr(\psi \xf(u/\psi))=0,$ by Proposition \ref{traceKG}, we deduce that $\xn=0$. From \eqref{u/psi} and \eqref{tildetau}, we obtain \eqref{58} with $C=\frac{1}{p-1}$.
\end{proof}

\begin{proposition}\label{equivba}
Assume $\mu \leq H^2$, $p>1$ and $\nu \in \mathfrak{M}^+(\partial\xO\cup \Sigma).$ Then problem \eqref{power} admits a solution with $\tau=0$ if and only if there exists a positive constant $C>0$ such that
\bal
\BBG_\xm[\BBK_\xm[\xn]^p]\leq C\, \BBK_\xm[\xn] \quad \text{a.e. in } \Omega\setminus\xS.
\eal
\end{proposition}
\begin{proof}
Proceeding as in the proof of Proposition \ref{equivint} and using Proposition \ref{prop2} instead of Proposition \ref{prop}, we obtain the desired result (see also \cite[Lemma 4.1]{BVi}).
\end{proof}

\subsection{Abstract setting} In this subsection, we present an abstract setting which will be applied to our particular framework in the next subsection.

Let $\mathbf{Z}$ be a metric space and $\gw \in\GTM^+(\mathbf{Z}).$ Let $J : \mathbf{Z} \times \mathbf{Z} \to (0,\infty]$ be a Borel positive kernel such that $J$ is symmetric and $J^{-1}$ satisfies a quasi-metric
inequality, i.e. there is a constant $C>1$ such that for all $x, y, z \in \mathbf{Z}$,
\ba \label{Jest} \frac{1}{J(x,y)}\leq C\left(\frac{1}{J(x,z)}+\frac{1}{J(z,y)}\right).
\ea
Under these conditions, one can define the quasi-metric $d$ by
\bal
\mathbf{d}(x,y):=\frac{1}{J(x,y)}
\eal
and denote by $\GTB(x,r):=\{y\in\mathbf{Z}:\; \mathbf{d}(x,y)<r\}$ the open $\mathbf{d}$-ball of radius $r > 0$ and center $x$.
Note that this set can be empty.

For $\xo\in\GTM^+(\mathbf{Z})$ and a positive function $\phi$, we define the potentials $\BBJ[\gw]$ and $\BBJ[\gf,\gw]$ by
\bal
\BBJ[\gw](x):=\int_{\mathbf{Z}}J(x,y) \dd\xo(y)\quad\text{and}\quad \BBJ[\gf,\gw](x):=\int_{\mathbf{Z}}J(x,y)\gf(y) \dd\gw(y).
\eal
For $t>1$ the capacity $\text{Cap}_{\BBJ,t}^\gw$ in $\mathbf{Z}$ is defined for any Borel $E\subset\mathbf{Z}$ by
\bal
\text{Cap}_{\BBJ,t}^\gw(E):=\inf\left\{\int_{\mathbf{Z}}\gf(x)^{t} \dd\gw(x):\;\;\gf\geq0,\;\; \BBJ[\gf,\gw] \geq\1_E\right\}.
\eal

\begin{proposition}\label{t2.1}  \emph{(\cite{KV})} Let $p>1$ and $\gl,\gw \in\GTM^+(\mathbf{Z})$ such that
	\ba
	\int_0^{2r}\frac{\gw\left(\GTB(x,s)\right)}{s^2} \dd s &\leq C\int_0^{r}\frac{\gw\left(\GTB(x,s)\right)}{s^2} \dd s ,\label{2.3}\\
	\sup_{y\in \GTB(x,r)}\int_0^{r}\frac{\gw\left(\GTB(y,r)\right)}{s^2} \dd s&\leq C\int_0^{r}\frac{\gw\left(\GTB(x,s)\right)}{s^2} \dd s,\label{2.4}
	\ea
for any $r > 0,$ $x \in \mathbf{Z}$, where $C > 0$ is a constant. Then the following statements are equivalent.
	
	1. The equation $v=\BBJ[|v|^p,\gw]+\ell \BBJ[\gl]$ has a positive solution for $\ell>0$ small.
	
	2. For any Borel set $E \subset \mathbf{Z}$, there holds
	$
\int_E \BBJ[\1_E\gl]^p \dd \gw \leq C\, \gl(E).
	$
	
	3. The following inequality holds
	$
\BBJ[\BBJ[\gl]^p,\gw]\leq C\BBJ[\gl]<\infty\quad \gw-a.e.
$

	4. For any Borel set $E \subset \mathbf{Z}$ there holds
$
\gl(E)\leq C\, \emph{Cap}_{\BBJ,p'}^\gw(E).
$
\end{proposition}

\subsection{Necessary and sufficient conditions for the existence}
For $\xa\leq N-2$, set
\bal
\CN_{\ga}(x,y):=\frac{\max\{|x-y|,d_\xS(x),d_\xS(y)\}^\xa}{|x-y|^{N-2}\max\{|x-y|,d(x),d(y)\}^2},\quad (x,y)\in\overline{\xO}\times\overline{\xO}, x \neq y,
\eal
\bel{opN} \Nthb[\omega](x):=\int_{\overline{\Gw}} \CN_{\xa}(x,y) \dd\omega(y), \quad \omega \in \GTM^+(\overline \Gw).\ee

We will point out below that $\Nthb$ defined in \eqref{opN} with $\dd \gw=d(x)^b d_\xS(x)^\theta \1_{\Omega \setminus \Sigma}(x)\,\dx$ satisfies all assumptions of $\BBJ$  in Proposition \ref{t2.1}, for some appropriate $b,\theta\in \BBR$. Let us first prove the quasi-metric inequality.

\begin{lemma}\label{ineq}
Let $\xa\leq N-2.$ There exists a positive constant $C=C(\xO,\Sigma,\xa)$ such that
\ba \label{dist-ineq} \frac{1}{\CN_\xa(x,y)}\leq C\left(\frac{1}{\CN_\xa(x,z)}+\frac{1}{\CN_\xa(z,y)}\right),\quad \forall x,y,z\in \overline{\xO}.
\ea
\end{lemma}
\begin{proof}
Let $0\leq b\leq 2,$ we first claim that there exists a positive constant $C=C(N,b,\xa)$ such that the following inequality is valid
\ba \label{35}\begin{aligned}
&\frac{|x-y|^{N-b}}{\max\{|x-y|,d_\xS(x),d_\xS(y)\}^\xa}\\
&\qquad\leq C\left(\frac{|x-z|^{N-b}}{\max\{|x-z|,d_\xS(x),d_\xS(z)\}^\xa}+\frac{|z-y|^{N-b}}{\max\{|z-y|,d_\xS(z),d_\xS(y)\}^\xa}\right).
\end{aligned} \ea
In order to prove \eqref{35}, we consider two cases. \medskip

\noindent \textbf{Case 1: $0<\xa\leq N-2$.}  We first assume that $|x-y|<2|x-z|$. Then by triangle inequality, we have
$d_\xS(z)\leq |x-z|+d_\xS(x)\leq 2\max\{|x-z|,d_\xS(x)\}$ hence
\bal
\max\{|x-z|,d_\xS(x),d_\xS(z)\}\leq 2\max\{ |x-z|,d_\xS(x)\}.
\eal
If $|x-z|\geq d_\xS(x)$ then

\ba\nonumber
\frac{|x-z|^{N-b}}{\max\{|x-z|,d_\xS(x),d_\xS(z)\}^\xa}&\geq 2^{-\xa}|x-z|^{N-b-\xa}\geq 2^{-N+b}|x-y|^{N-b-\xa}\\
&\geq 2^{-N+b}\frac{|x-y|^{N-b}}{\max\{|x-y|,d_\xS(x),d_\xS(y)\}^\xa}.\label{36}
\ea
If $|x-z|\leq d_\xS(x)$ then
\ba\nonumber
\frac{|x-z|^{N-b}}{\max\{|x-z|,d_\xS(x),d_\xS(z)\}^\xa}&\geq 2^{-\xa}d_\xS(x)^{-\xa}|x-z|^{N-b} \\
&\geq 2^{-\xa-N+b}d_\xS(x)^{-\xa}|x-y|^{N-b}\\
&\geq 2^{-N+b+\xa}\frac{|x-y|^{N-b}}{\max\{|x-y|,d_\xS(x),d_\xS(y)\}^\xa}.\label{37}
\ea
Combining \eqref{36}--\eqref{37}, we obtain \eqref{35} with $C=2^{N-b}$.

Next we consider the case $2|x-z|\leq|x-y|$. Then $\frac{1}{2} |x-y|\leq |y-z|,$ thus by symmetry we obtain \eqref{35} with $C=2^{N-b}$. \medskip

\noindent \textbf{Case 2: $\xa\leq 0$.} Since $d_\xS(x)\leq |x-y|+d_\xS(y)$,  it follows that
\bal
\max\{|x-y|,d_\xS(x),d_\xS(y)\}\leq |x-y|+\min\{d_\xS(x),d_\xS(y)\}.
\eal
Using the above estimate, we obtain
\bal
&|x-y|^{N-b}\max\{|x-y|,d_\xS(x),d_\xS(y)\}^{-\xa} \\
&\leq |x-y|^{N-b-\xa}+\min\{d_\xS(x),d_\xS(y)\}^{-\xa}|x-y|^{N-b}\\ \nonumber
&\leq 2^{N-b-\xa}(|x-z|^{N-b-\xa}+|y-z|^{N-b-\xa})\\ \nonumber
&+2^{N-b}(|x-z|^{N-b}\min\{d_\xS(x),d_\xS(y)\}^{-\xa}+|y-z|^{N-b}\min\{d_\xS(x),d_\xS(y)\}^{-\xa})\\
&\leq ( 2^{N-b-\xa}+1)\left(\frac{|x-z|^{N-b}}{\max\{|x-z|,d_\xS(x),d_\xS(z)\}^\xa}+\frac{|z-y|^{N-b}}{\max\{|z-y|,d_\xS(z),d_\xS(y)\}^\xa}\right),
\eal
which yields \eqref{35}. \smallskip

Now we will use \eqref{35} with $b=2$ to prove \eqref{dist-ineq}. Since $d(x)\leq |x-y|+d(y),$ we can easily show that $\max\{|x-y|,d(x),d(y)\}\leq |x-y|+\min\{d(x),d(y)\}.$ Hence
\bal
\frac{1}{\CN_\xa(x,y)}&=\frac{\max\{|x-y|,d(x),d(y)\}^2|x-y|^{N-2}}{\max\{|x-y|,d_\xS(x),d_\xS(y)\}^\xa} \\
&\leq  \frac{2|x-y|^{N}}{\max\{|x-y|,d_\xS(x),d_\xS(y)\}^\xa}
+\frac{2\min\{d(x),d(y)\}^2|x-y|^{N-2}}{\max\{|x-y|,d_\xS(x),d_\xS(y)\}^\xa}\\
&\leq C(N,\xa)\left(\frac{1}{\CN_\xa(x,z)}+\frac{1}{\CN_\xa(z,y)}\right),
\eal
where in the last inequality we have used \eqref{35}. The proof is complete.
\end{proof}

Next we give sufficient conditions for \eqref{2.3} and \eqref{2.4} to hold.
\begin{lemma}\label{l2.3} Let $b>0$, $\theta>k-N$ and  $\dd \gw=d(x)^b d_\xS(x)^\theta \1_{\xO \setminus \Sigma}(x)\,\dx$. Then
\be \label{doubling}
\gw(B(x,s))\approx \max\{d(x),s\}^b\max\{d_\xS(x),s\}^{\theta} s^N, \; \text{for all } x\in\xO \text{ and } 0 < s\leq 4 \mathrm{diam}(\xO).
\ee
\end{lemma}
\begin{proof}
Let $\beta_0$ be as in Subsection \ref{assumptionK} and $s<\frac{\xb_0}{16}$. First we assume that $x\in \xS_{\frac{\xb_0}{4}}$ then $d(y)\approx 1$ for any $y\in B(x,s)$. Thus, it is enough to show that
\be \label{39}
\int_{B(x,s)}d_\xS(y)^{\theta} \dy\approx \max\{d_\xS(x),s\}^{\theta} s^N.
\ee
\noindent\textbf{Case 1:  $d_\xS(x)\geq 2s$.} Then $\frac{1}{2}d_{\xS}(x)\leq d_\xS(y)\leq\frac{3}{2} d_\xS(x)$ for any $y\in B(x,s)$, therefore \eqref{39} follows easily in this case. \medskip

\noindent \textbf{Case 2: $d_\xS(x)\leq 2s$.} Then there exists  $\xi \in\xS $ such that $B(x,s)\subset V(\xi,4\xb_0)$. If $y\in B(x,s),$ then $|y'-x'|<s$ and $d_\xS(y)\leq d_\xS(x)+|x-y|\leq 3s$.
Thus by \eqref{propdist},
$\xd_\Sigma^\xi(y)\leq C_1s$ for any $y\in B(x,s)$, where $C_1$ depends on $\| \Sigma \|_{C^2},N$ and $k$. Thus
\bal
\int_{B(x,s)}d_\xS(y)^\theta \dy \lesssim \int_{\{|x'-y'|<s\}}\int_{\{\xd_\Sigma^\xi(y)\leq C_1s\}}(\xd_\Sigma^\xi(y))^\theta \dy''\dy'
\approx s^{N+\theta} \approx \max\{d_\xS(x),s\}^\theta s^{N}.
\eal
Here the similar constants depend on $N,k,\| \Sigma \|_{C^2}$ and $\beta_0$. \medskip

\noindent \textbf{Case 3:} $d_\xS(x)\leq 2s$ and $\theta<0$. We have that $d_\xS(y)^\theta \geq3^\theta s^{\theta}$  for any $y\in B(x,s)$, which leads to \eqref{39}. \medskip

\noindent \textbf{Case 4:} $d_\xS(x)\leq 2s$ and $\theta \geq 0$. Let $C_2=C \| \Sigma \|_{C^2}$ be the constant in \eqref{propdist}.

If $d_\xS(x) \leq \frac{s}{6(N-k)C_2}$ then
by \eqref{propdist} we have $\xd_\Sigma^\xi(x)\leq\frac{s}{6(N-k)}$. In addition for any
\bal
y\in \left\{\psi = (\psi',\psi'') \in\xO \setminus \Sigma:|x'-\psi'|\leq \frac{s}{6(N-k)C_2},\; \xd_\Sigma^\xi(\psi)\leq \frac{s}{6(N-k)} \right\}=:\CA,
\eal
we have
\bal
|x''-y''|\leq \xd_\Sigma^\xi(x)+\xd_\Sigma^\xi(y)+\left(\sum_{i=k+1}^N|\xG_i^\xi(x')-\xG_i^\xi(y')|^2\right)^\frac{1}{2}\leq  \frac{s}{3}+(N-k)\| \Sigma \|_{C^2}|x'-y'|\leq \frac{s}{2}.
\eal
This implies that $\CA\subset B(x,s)$.  Consequently,
\bal
\int_{B(x,s)}d_\xS(y)^\theta\dy&\approx \int_{B(x,s)}(\xd_\xS^\xi(y))^\theta \dy\gtrsim \int_{\CA}(\xd_\xS^\xi(y))^\theta \dy \approx C\max\{d_\xS(x),s\}^\theta s^{N}.
\eal

If $d_\xS(x)\geq \frac{s}{6(N-k)C_2}$ then
\bal
\int_{B(x,s)}d_\xS(y)^\theta \dy\geq\int_{B(x,\frac{s}{12(N-k)C_2})}d_\xS(y)^\theta \dy
\eal
and hence \eqref{39} follows by case 1.

Next we consider $x\in \xO_{\frac{\xb_0}{4}}$. Then $d_\xS(y)\approx 1$ for any $y\in \xO_{\frac{\xb_0}{2}}$. By proceeding as before we may prove \eqref{doubling} for any $s<\frac{\xb_0}{16}$.

If $x\in \xO\setminus(\xO_{\frac{\xb_0}{4}}\cup \Sigma_{\frac{\xb_0}{4}})$ then
$d_\xS(y),d(x)\approx 1$ for any $y\in B(y,s),$ with $s<\frac{\xb_0}{16}.$ Thus, in this case, we can easily prove \eqref{39} for any $s<\frac{\xb_0}{16}.$

If $ \frac{\xb_0}{16}\leq s\leq 4 \text{diam}(\xO)$ then  $\gw(B(x,s))\approx 1$, hence estimate \eqref{doubling} follows straightforward. The proof is complete.
\end{proof}

\begin{lemma}\label{vol}
Let $\xa< N-2$, $b> 0$, $\theta>\max\{k-N,-2-\xa\}$ and $\dd \gw=d(x)^b d_\xS(x)^\theta$ $\1_{\Omega \setminus \Sigma}(x)\,\dx$.  Then \eqref{2.3}  holds.
\end{lemma}
\begin{proof}
We note that if $s\geq (4\diam(\xO))^{N-\xa}$ then $\xo(\GTB(x,s))=\xo(\overline{\xO})<\infty$, where $\GTB(x,s)$ is defined after \eqref{Jest}, namely $\GTB(x,s)=\{y \in \Omega \setminus \Sigma: \mathbf{d}(x,y)<s\}$ and $\mathbf{d}(x,y)= \frac{1}{\CN_\xa(x,y)}$.

We first assume that $0<\xa<N-2$.  Let $x\in \Sigma_\frac{\xb_0}{4}$ then
\ba \label{C0d}
0<C_0 \leq d(x)\leq 2 \diam(\xO),
\ea
where $C_0$ depends on $\xO,\Sigma,\beta_0$. Set
\bal
\GTC(x,s):=\left\{y \in \Omega \setminus \Sigma: \frac{|x-y|^{N-2}}{\max\{|x-y|,d_\xS(x),d_\xS(y)\}^\xa}< s\right\}.
\eal
Then
\ba \label{CB} \GTC\left(x,\frac{s}{4\diam(\xO)^2}\right)\subset\GTB(x,s)\subset \GTC\left(x,\frac{s}{C_0^2}\right).
\ea
We note that $B(x,S_1)\subset\GTC(x,s)\subset B(x,l_1S_1)$
where $S_1=\max\{s^{\frac{1}{N-2-\xa}},s^{\frac{1}{N-2}}d_\xS(x)^{\frac{\xa}{N-2}} \}$ and $l_1=2^{\frac{\xa}{N-2-a}}$. Therefore, by Lemma \ref{l2.3}, we obtain
\ba\nonumber
\xo(\GTB(x,s))&\approx \max\left\{d_\xS(x),\max\{s^{\frac{1}{N-2-\xa}},s^{\frac{1}{N-2}}d_\xS(x)^{\frac{\xa}{N-2}} \}\right\}^\theta \max\{s^{\frac{1}{N-2-\xa}},s^{\frac{1}{N-2}}d_\xS(x)^{\frac{\xa}{N-2}} \}^N\\
&\approx\left\{\BAL                     &d_\xS(x)^{\theta+\frac{\xa N}{N-2}}s^{\frac{N}{N-2}} \quad&&\text{if } s\in (0,d_\xS(x)^{N-2-\xa}), \\
                                        &s^{\frac{\theta+N}{N-2-\xa}} \quad &&\text{if } s\in [d_\xS(x)^{N-2-\xa},M),\\
                                        &1 \quad &&\text{if } s \in [M,\infty).
\EAL\right.\label{xoest1}
\ea
where
\bal
M:=(4\diam(\xO))^\frac{N(N-\xa)}{b+N}+(4\diam(\xO))^{\frac{(N-2)(N-\xa)}{N}}+(4\diam(\xO))^{\frac{(N-\xa-2)(N-\xa)}{\theta +N}}.
\eal

Next we assume that $\xa\leq0$, Let $x\in \Sigma_\frac{\xb_0}{4}$ then \eqref{C0d} and \eqref{CB} hold. We also have $B(x,{l_2S_2})\subset\GTC(x,s)\subset B(x,S_2)$, where $S_2=\min\{s^{\frac{1}{N-2-\xa}},s^{\frac{1}{N-2}}d_\xS(x)^{\frac{\xa}{N-2}}\}$ and $l_2=2^{\frac{\xa}{N-2}}$. Therefore by Lemma \ref{l2.3}, we obtain
\ba\nonumber
\xo(\GTB(x,s))&\approx \max\left\{d_\xS(x),\min\{s^{\frac{1}{N-2-\xa}},s^{\frac{1}{N-2}}d_\xS(x)^{\frac{\xa}{N-2}}\}\right\}^\theta \min\{s^{\frac{1}{N-2-\xa}},s^{\frac{1}{N-2}}d_\xS(x)^{\frac{\xa}{N-2}}\}^N\\
&\approx\left\{\BAL &d_\xS(x)^{\theta+\frac{\xa N}{N-2}}s^{\frac{N}{N-2}} \quad &&\text{if } s\in (0,d_\xS(x)^{N-2-\xa}),\\
&s^{\frac{\theta+N}{N-2-\xa}} \quad &&\text{if } s\in [d_\xS(x)^{N-2-\xa},M),\\
 &1 \quad &&\text{if } s \in [M,\infty).
\EAL\right.\label{xoest2}
\ea

Next consider $x\in \xO_{\frac{\xb_0}{4}},$ then there exists a positive constant $C_3=C_3(\xO,\Sigma,\xa,\xb_0)$ such that $C_3\leq d_\xS(x)< 2 \diam(\xO).$ Set
\bal
\mathcal{E}(x,s):=\{y \in \Omega \setminus \Sigma: |x-y|^{N-2}\max\{|x-y|,d(x),d(y)\}^{2}< s\}.
\eal
We obtain
\bal
\mathcal{E}(x,\min\{C_3^\xa,2^\xa\diam^\xa(\xO)\}s)\subset\GTB(x,s)\subset \mathcal{E}(x,\max\{C_3^\xa,2^\xa\diam^\xa(\xO)\}s).
\eal
We also have
\bal
B(x,l_3S_3)\subset\mathcal{E}(x,s)\subset B(x,S_3),
\eal
where $S_3=\min\{s^{\frac{1}{N}},s^{\frac{1}{N-2}}d(x)^{-\frac{2}{N-2}}\}$ and $l_3=2^{-\frac{2}{N-2}}$. Again, by Lemma \ref{l2.3}, we obtain
\ba\nonumber
\xo(\GTB(x,s))&\approx \max\left\{d(x),\min\{s^{\frac{1}{N}},s^{\frac{1}{N-2}}d(x)^{-\frac{2}{N-2}}\}\right\}^b \min\{s^{\frac{1}{N}},s^{\frac{1}{N-2}}d(x)^{-\frac{2}{N-2}}\}^N\\
&\approx\left\{\BAL &d(x)^{b-\frac{2 N}{N-2}}s^{\frac{N}{N-2}} \quad &&\text{if } s\in (0,d(x)^{N}),\\
&s^{\frac{b+N}{N}}\quad &&\text{if } s\in [d(x)^{N},M),\\
&1\quad &&\text{if } s \in [M,\infty).
\EAL\right.\label{xoest3}
\ea

Let $0<\bar \xb \leq \frac{\xb_0}{4}$ and $x\in \xO\setminus (\xO_{\bar \xb}\cup \Sigma_{\bar \xb}).$ Then there exists a positive constant $C_4=C_4(\xO,\Sigma,\bar \xb)$ such that $C_4\leq d_\xS(x),d(x)< 2 \diam(\xO).$ By Lemma \ref{l2.3}, we can show that
\ba
\xo(\GTB(x,s))&\approx \left\{\BAL &s^{\frac{N}{N-2}} \quad &&\text{if } s \in (0,M),\\
                                           &1 \quad &&\text{if } s \in [M,\infty).
\EAL\right.
\label{xoest4}
\ea

Combining \eqref{xoest1}--\eqref{xoest4} leads to \eqref{2.3}. The proof is complete.
\end{proof}

\begin{lemma}\label{vol-2}
	We assume that $\xa< N-2$, $b> 0$, $\theta>\max\{k-N,-2-\xa\}$ and $\dd \gw=d(x)^b d_\xS(x)^\theta\1_{\Omega \setminus \Sigma}(x)\,\dx$.  Then \eqref{2.4}  holds.
\end{lemma}
\begin{proof}
We  consider only the case $\xa>0$  and $x\in \Sigma_{\frac{\xb_0}{16}}$ since the other cases $x\in \xO_{\frac{\xb_0}{16}}$ and $x\in \xO\setminus(\xO_{\frac{\xb_0}{16}}\cup \Sigma_{\frac{\xb_0}{16}})$ can be treated similarly and we omit them. We take $r>0$.

\noindent \textbf{Case 1:} $0<r<\left(\frac{\xb_0}{16(2\diam(\xO))^\xa}\right)^{N}$. In this case, we note that $\GTB(x,r) \subset \Sigma_{\frac{\xb_0}{8}}$. This and  \eqref{xoest1} imply that, for any $y \in \GTB(x,r)$,
\ba\nonumber
\int_0^r\frac{\xo(\GTB(y,s))}{s^2} \dd s\approx\left\{\BAL                     &d_\xS(y)^{\theta+\frac{\xa N}{N-2}}r^{\frac{2}{N-2}} \quad &&\text{if } r\in (0,d_\xS(y)^{N-2-\xa}),\\
                                        &r^{\frac{\theta+2+\xa}{N-2-\xa}} \quad &&\text{if } r\in [d_\xS(y)^{N-2-\xa},M), \\
                                        &1 \quad &&\text{if } r \in [M,\infty).
\EAL\right.\label{intxoest1}
\ea

If $|x-y|\leq \frac{1}{2}d_\xS(x)$ then $\frac{1}{2}d_\xS(x)\leq d_\xS(y)\leq \frac{3}{2}d_\xS(x)$. Therefore, when $d_\xS(y),d_\xS(x)\geq r^{\frac{1}{N-2-\xa}}$, we obtain
\bal
\int_0^r\frac{\xo(\GTB(y,s))}{s^2} \dd s \approx d_\xS(y)^{\theta+\frac{\xa N}{N-2}}r^{\frac{2}{N-2}}\approx d_\xS(x)^{\theta+\frac{\xa N}{N-2}}r^{\frac{2}{N-2}}\approx \int_0^r\frac{\xo(\GTB(x,s))}{s^2} \dd s.
\eal
If $d_\xS(y)\geq r^{\frac{1}{N-2-\xa}}$ and $d_\xS(x)\leq r^{\frac{1}{N-2-\xa}}$ then $d_\xS(y)\leq \frac{3}{2}r^{\frac{1}{N-2-\xa}}$, which implies
\bal
\int_0^r\frac{\xo(\GTB(y,s))}{s^2} \dd s \approx d_\xS(y)^{\theta+\frac{\xa N}{N-2}}r^{\frac{2}{N-2}}\approx r^{\frac{\theta+2+\xa}{N-2-\xa}}\approx \int_0^r\frac{\xo(\GTB(x,s))}{s^2} \dd s.
\eal
If $d_\xS(y)\leq r^{\frac{1}{N-2-\xa}}$ and $d_\xS(x)\geq r^{\frac{1}{N-2-\xa}}$ then $d_\xS(x)\leq 2r^{\frac{1}{N-2-\xa}}$, which yields
\bal
\int_0^r\frac{\xo(\GTB(x,s))}{s^2} \dd s\approx d_\xS(x)^{\theta+\frac{\xa N}{N-2}}r^{\frac{2}{N-2}}\approx r^{\frac{\theta+2+\xa}{N-2-\xa}}\approx \int_0^r\frac{\xo(\GTB(y,s))}{s^2} \dd s.
\eal
If $d_\xS(y)\leq r^{\frac{1}{N-2-\xa}}$ and $d_\xS(x)\leq r^{\frac{1}{N-2-\xa}}$ then
\bal
\int_0^r\frac{\xo(\GTB(x,s))}{s^2} \dd s\approx r^{\frac{\theta+2+\xa}{N-2-\xa}}\approx \int_0^r\frac{\xo(\GTB(y,s))}{s^2} \dd s.
\eal

Now we assume that $y \in \GTB(x,r)$ and $|x-y|\geq \frac{1}{2}d_\xS(x).$ Then
\bal d_\xS(y)\leq \frac{3}{2}|x-y| \quad \text{and} \quad
|x-y|\leq C(\xb_0,\xO,N,\xS) r^{\frac{1}{N-\xa-2}}.
\eal
Hence $d_\xS(x),d_\xS(y)\lesssim r^{\frac{1}{N-\xa-2}}$. Proceeding as above we obtain the desired result. \medskip

\noindent \textbf{Case 2:} $r\geq \left(\frac{\xb_0}{16(2\diam(\xO))^\xa}\right)^{N}$. By \eqref{xoest1}--\eqref{xoest4}, we can easily prove that
\bal
\int_0^r\frac{\xo(\GTB(y,s))}{s^2} \dd s \approx 1,\quad \forall y\in \overline{\xO},
\eal
and the desired result follows easily in this case.
\end{proof}

For $b>0$, $\theta>-N+k$ and $s>1$, define the capacity $\text{Cap}_{\Nthb,s}^{b,\theta}$ by
\bal \text{Cap}_{\Nthb,s}^{b,\theta}(E) :=\inf\left\{\int_{\overline{\xO}}d^b d^\theta_\xS\gf^s\,\dx:\;\; \gf \geq 0, \;\;\Nthb[ d^b d_{\Sigma}^\theta\gf ]\geq\1_E\right\} \quad \text{for  Borel set } E\subset\overline{\xO}.
\eal
Here $\1_E$ denotes the indicator function of $E$. Furthermore, by \cite[Theorem 2.5.1]{Ad},
\be\label{dualcap}
(\text{Cap}_{\Nthb,s}^{b,\theta}(E))^\frac{1}{s}=\sup\{\tau(E):\tau\in\GTM^+(E), \|\Nthb[\tau]\|_{L^{s'}(\xO;d^b d^\theta_\xS)} \leq 1 \}.
\ee

%
%
%
%

Now we are ready to prove Theorem \ref{theoremint}.

\begin{proof}[\textbf{Proof of Theorem \ref{theoremint}}.] We will apply Proposition \ref{t2.1} with $J(x,y)=\CN_{2\am}(x,y)$, $\dd \xo= \left(d(x)d_{\xS}(x)^{-\am}\right)^{p+1}\dx$ and $\dd \lambda = \ei\1_{\xO\setminus \Sigma} \dd \gt$. Estimate \eqref{Jest} is satisfied thanks to Lemma \ref{ineq}, while assumptions \eqref{2.3}--\eqref{2.4} are fulfilled thanks to Lemmas \ref{vol}--\ref{vol-2} respectively with $\alpha=2\am$, $b=p+1$ and $\theta=-\am(p+1)$. We note that condition \eqref{p-cond} ensures that $b$ and $\theta$ satisfy the assumptions in Lemmas \eqref{vol}--\eqref{vol-2}.

Moreover, we have the following observations.

(i) There holds
\ba \label{GN2a}
G_\mu(x,y)\approx d(x)d(y)(d_{\xS}(x)d_{\xS}(y))^{-\am} \CN_{2\am}(x,y) \quad \forall x,y \in \Omega \setminus \Sigma, x \neq y.
\ea
Consequently, if  the equation
\ba \label{vN2a} v=\BBN_{2\am}[(dd_{\xS}^{-\am})^{p+1}v^p]+\ell \BBN_{2\am}[\gl]
\ea
has a solution $v$ for some $\ell>0$ then the function $\tilde v(x)=d(x)d_{\xS}(x)^{-\am}v(x)$ satisfies
$\tilde v \approx\BBG_\xm[\tilde v^p]+\ell \BBG_\xm[\gt]$. By \cite[Proposition 2.7]{BHV}, there exists $\rho>0$ small such that equation \eqref{u-rhotau} has a positive solution $u$. By the above argument, we can show that equations \eqref{vN2a} has a solution for $\ell>0$ small if and only if equation \eqref{u-rhotau} has a solution for $\rho>0$ small. In other words, statement 1 of Proposition \ref{t2.1} is equivalent to statement 1 of the present Theorem.

(ii) With $J,\omega$ and $\lambda$ as above, from \eqref{GN2a}, we deduce easily that statements 2--4 of Proposition \ref{t2.1} reduce to statements 2--4 of the present Theorem respectively.

From the above observations and  Proposition \ref{t2.1}, we obtain the desired results.
\end{proof}
\begin{remark} \label{capp-1} Assume $0<\mu \leq H^2$. By combining
\eqref{dualcap}, \eqref{GN2a} and \eqref{Gphi_mu}, we derive that	for any $1<p<\frac{N+1}{N-1}$,
\be \label{Nz}
\inf_{z \in \Omega \setminus \Sigma} \mathrm{Cap}_{\BBN_{2\am},p'}^{p+1,-\am(p+1)}(\{z\})> C.
\ee	
Hence, for $1<p<\frac{N+1}{N-1}$, statement 3 of Theorem \ref{theoremint} is valid, therefore, statements 1 and 2 of Theorem \ref{theoremint} hold true. This covers Theorem \ref{th2} (i) with $\gamma=1$ and Proposition \eqref{equivint}.
\end{remark}
\begin{proposition} \label{nonexist-1} Assume $0<\mu<\left( \frac{N-2}{2} \right)^2$ and $p \geq \frac{N+1}{N-1}$. Then there exists a measure $\tau \in \GTM^+(\Omega \setminus \Sigma;\ei)$ with $\| \tau \|_{\GTM(\Omega \setminus \Sigma;\ei)}=1$ such that problem \eqref{u-rhotau} does not admit positive solution for any $\xr>0$.
\end{proposition}
\begin{proof}
Suppose by contradiction that for every $\tau \in \GTM^+(\Omega \setminus \Sigma;\ei)$ with $\norm{\tau}_{\GTM(\Omega \setminus \Sigma;\ei)} = 1$, there exists a positive solution to problem \eqref{u-rhotau} for some $\xr>0.$ Let $y^* \in \partial \Omega$ and $\{y_n\} \subset \Omega \setminus \Sigma$ such that $y_n \to y^* \in \partial \Omega$ and $\dist(y_n,\xS)>\xe>0,$ for some $\xe>0$.

From \eqref{GN2a} and \eqref{eigenfunctionestimates}, we have
\ba\label{72}
G_\mu(x,y_n)\ei(y_n)^{-1} \gtrsim \dfrac{1}{|x-y_n|^{N-2}}\cdot \frac{\ei(x)}{\max\{d(x)^2,d(y_n)^2,|x-y_n|^2\}}=:F(x,y_n).
\ea
By using  Fatou lemma and \eqref{eigenfunctionestimates}, we deduce that
\ba \label{eq:bound1} \nonumber
		\liminf_{ n \to \infty} \int_{\Omega \setminus \Sigma} F(x,y_n)^p \ei(x) dx& \geq \int_{\Omega \setminus \Sigma} (\liminf_{n \to \infty} F(x,y_n)^p ) \ei(x) \dd x \\  \nonumber
		&\gtrsim  \int_{\Omega \setminus \Sigma} \left( \frac{\ei(x)}{|x-y^*|^{N}} \right)^p \ei(x) \dd x \\
		&\approx  \int_{\Omega \setminus \Sigma} \left( \frac{d(x)d_\Sigma(x)^{-\am}}{|x-y^*|^{N}} \right)^p d(x)d_\Sigma(x)^{-\am} \dd x.
\ea

Since $\Omega$ is a $C^2$ domain, it satisfies the interior cone condition, hence there exists $r_0 > 0$ small enough such that the circular cone at vertex $y^*$
\bal
{\mathcal C}_{r_0}(y^*):=\left\{ x \in B_{r_0}(y^*):  (x-y^*)\cdot {\bf n}_{y^*} > \frac{1}{2}|x-y^*| \right\}  \subset \Omega \setminus \Sigma,
\eal
where ${\bf n}_{y^*}$ denotes the inward unit normal vector to $\partial \Omega$ at $y^*$.

Without loss of generality, suppose that the coordinates are placed so that $y^* = 0 \in \partial \Omega$, the tangent hyperplane to $\partial \Omega$ at $0$ is $\{ x=(x_1,\ldots,x_{N-1},x_N) \in \R^N: x_N=0\}$ and ${\bf n}_0 = (0,\ldots, 0,1)$. We can choose $r_0$ small enough such that $d(x) \geq \alpha |x|$ for all $x \in \mathcal{C}_{r_0}(0)$ and for some $\alpha \in (0,1)$.
Then we have
\be\label{intr0}
		\int_{\Omega \setminus \Sigma} \left( \frac{d(x)d_\Sigma(x)^{-\am}}{|x|^{N}} \right)^p d(x)d_\Sigma(x)^{-\am} \dd x
		 \gtrsim \int_{{\mathcal C}_{r_0}(0)} |x|^{1-(N-1)p} \dd x  \sim \int_0^{r_0} t^{N-(N-1)p} \dd t.
\ee
Since $p \geq \frac{N+1}{N-1}$, the last integral in \eqref{intr0} is divergent. This and \eqref{eq:bound1}, \eqref{intr0} yield
$\liminf_{n \to \infty}\int_{\Omega \setminus \Sigma}F(x,y_n)^p \ei \dd x  =\infty$. Consequently, for any $j\in \BBN,$ there exists $n_j\in \BBN$ such that
\ba\label{73}
2^{jp}\leq \int_{\Omega \setminus \Sigma}F(x,y_{n_j})^p \ei \dd x.
\ea

Put $\tau_k: = \sum_{j=1}^k 2^{-j}\frac{\delta_{y_{n_j}}}{\phi_\mu}$ then $\| \tau_k \|_{\GTM^+(\Omega \setminus \Sigma;\ei)} \leq 1$ and  $\tau_k\leq \tau_{k+1}$ for any $k\in \BBN$. Put $\tau=\lim_{k\to\infty}\tau_k$ then 
\bal \int_{\xO\setminus\xS}\ei \dd \tau=\sum_{j=1}^\infty 2^{-j}=1.
\eal
By the supposition, there exists a positive solution $u\in L^p(\xO\setminus\xS;\ei)$ of problem \eqref{u-rhotau} with datum $\rho \tau$.
From the representation formula and \eqref{72}, we deduce
\bal
 u = \BBG_\mu[u^p] + \rho \BBG_\mu[\tau]\geq \xr\sum_{j=1}^\infty 2^{-j}\BBG_\mu[\frac{\delta_{y_{n_j}}}{\phi_\mu}]\gtrsim \xr\sum_{j=1}^\infty 2^{-j}F(x,y_{n_j}).
\eal
The above inequality and \eqref{73} yield
\bal
\int_{\xO\setminus \xS}u^p\ei dx\gtrsim \xr^p\sum_{j=1}^\infty 2^{-jp}\int_{\xO\setminus \xS}F(x,y_{n_j})^p\ei dx\geq \xr^p\sum_{j=1}^\infty1=\infty,
\eal
which is clearly a contradiction since $u\in L^p(\xO\setminus\xS;\ei).$ The proof is complete.
\end{proof}

\begin{proposition} \label{remove1} Assume $0<\mu<\left( \frac{N-2}{2} \right)^2$ and $p \geq \frac{\am+2}{\am}$. Then  for any $\rho>0$ and any $\tau \in \GTM^+(\Omega \setminus \Sigma;\ei)$ with $\| \tau \|_{\GTM(\Omega \setminus \Sigma;\ei)}=1$, there is no solution of problem \eqref{u-rhotau}.
\end{proposition}
\begin{proof}
Suppose by contradiction that there exist $\tau \in \GTM^+(\Omega \setminus \Sigma;\ei)$ with $\| \tau \|_{\GTM(\Omega \setminus \Sigma;\ei)}=1$ and $\xr>0$ such that problem \eqref{u-rhotau} admits a positive solution $u \in L^p(\Omega;\ei)$.

Since $\tau\not\equiv0,$ there exist $x_0\in\xO\setminus\xS$, $r,\xe>0$ such that $B(x_0,r)\subset\xO\setminus\xS$,  $\dist(B(x_0,r),\xS)>\xe$, and $\tau(B(x_0,r))>0$. Set $\tau_B=\1_{B(x_0,r)}\tau,$ then $\tau_B\leq \tau$. Let $v_1=\BBG_\xm[\xr \tau_B],$ we consider the sequence $\{v_{k}\}_{k=1}^\infty\subset L^p(\Omega;\ei)$ which satisfies the following problem
\bal
- L_\gm v_{k+1}= |v_{k}|^{p-1}v_{k} +\xr \tau_B\qquad \text{in }\;\Gw\setminus \Sigma, \quad 
\tr(v_{k+1})=0,
\eal
for any $k\in\BBN.$ Using \eqref{poi5}, we can easily show that $0 \leq v_k\leq v_{k+1}$ and $v_k\leq u$ for any $k\in\BBN.$ Since $u\in  L^p(\Omega;\ei),$ by monotone convergence theorem, we deduce that $v=\lim_{k\to\infty}v_k$ belongs to $ L^p(\Omega;\ei)$, $v \geq 0$, and
$
v=\lim_{k\to\infty} v_{k+1}=\lim_{k\to\infty}\BBG_\xm[v_k^p+\xr \tau_B]=\BBG_\xm[v^p+\xr \tau_B],
$
which means that $v\in L^p(\Omega;\ei)$ is a weak solution of
\bal 
- L_\gm v= |v|^{p-1}v +\xr \tau_B\qquad \text{in }\;\Gw\setminus \Sigma, \quad 
\tr(v)=0.
\eal
By Proposition \ref{equivint}, there exists a positive constant $C$ depending on $\xr$ and $p$ such that
\ba\label{58c}
\BBG_\xm[\BBG_\xm[\tau_B]^p]\leq C\, \BBG_\xm[\tau_B] \quad \text{a.e. in } \Omega\setminus\xS.
\ea

 Assume $0 \in \Sigma$ and set $\xb=\frac{1}{4}\min\{\xb_0,r\}$. Let $x\in\xO\setminus \xS$ such that $|x|\leq \frac{\xb}{2}$. Since
 \ba\label{74}
 \BBG_\xm[\tau_B] \approx \ei\approx   d_{\xS}^{-\am}\quad \text{in}\;\; \xS_\xb,
 \ea
 and $d_\Sigma(y) \leq |y|$ for any $y \in \Sigma_\beta$, we have, for any $x \in B(0,\frac{\beta}{2}) \setminus \Sigma$,
\bal
\BBG_\xm[\BBG_\xm[\gt_B]^p](x)
&\gtrsim d_\xS(x)^{-\xa}\int_{\Sigma_{\xb}} d_{\xS}(y)^{-\am(p+1)}|x-y|^{2+2\am-N}\dy\\
&\geq d_\xS(x)^{-\am}\int_{\Sigma_{\xb}} |y|^{-\am(p+1)}|x-y|^{2+2\am-N}\dy\\
&\approx\left\{
\BAL
&d_\xS(x)^{-\am}|\ln|x|| \quad &&\text{if}\;p=\frac{2+\am}{\am},\\
&d_\xS(x)^{-\am}|x|^{2+\am-p\am} \quad &&\text{if}\;p>\frac{2+\am}{\am}.
\EAL\right.
\eal
This and \eqref{74} yield that \eqref{58c} is not valid as $|x| \to 0$, which is clearly a contradiction.
\end{proof}

In order to study the boundary value problem with measure data concentrated on $\partial \Omega \cup \Sigma$, we make use of  specific capacities which are defined below.

For $\ga\in\BBR$ we define the Bessel kernel of order $\ga$ in $\R^d$ by  $\CB_{d,\ga}(\xi):=\CF^{-1}\left((1+|.|^2)^{-\frac{\ga}{2}}\right)(\xi),
$
where $\CF$ is the Fourier transform in the space $\mathcal{S}'(\R^d)$ of moderate distributions in $\BBR^d$. For $\lambda \in \GTM(\R^d)$, set
\bal
\BBB_{d,\alpha}[\lambda](x):= \int_{\R^d}\CB_{d,\alpha}(x-y) \dd\lambda(y), \quad x \in \R^d.
\eal
Let
$L_{\ga,\kappa}(\BBR^d):=\{f=\CB_{d,\alpha} \ast g:g\in L^{\kappa}(\BBR^d)\}
$ be the Bessel space with the norm
\bal
\|f\|_{L_{\ga,\kappa}}:=\|g\|_{L^\kappa}=\|\CB_{d,-\ga}\ast f\|_{L^\kappa}.
\eal
It is known that if $1<\kappa<\infty$ and $\ga>0$, $L_{\ga,\kappa}(\BBR^d)=W^{\ga,\kappa}(\BBR^d)$ if $\ga\in\BBN.$ If $\ga\notin\BBN$ then the positive cone of their dual coincide, i.e. $(L_{-\ga,\kappa'}(\BBR^d))_+=(B^{-\ga,\kappa'}(\BBR^d))_+$, always with equivalent norms. The Bessel capacity is defined for compact subsets
$K \subset\BBR^d$ by
\bal
\mathrm{Cap}_{{\CB_{d,\alpha},\kappa}}^{\R^d}(K):=\inf\{\|f\|^\kappa_{L_{\ga,\kappa}}, f\in\CS'(\BBR^d),\,f\geq \1_K \}.
\eal

If $\Gamma \subset \overline{\Omega}$ is a $C^2$ submanifold without boundary, of dimension $d$ with $1 \leq d \leq N-1$ then there exist open sets $O_1,...,O_m$ in $\BBR^N$, diffeomorphism $T_i: O_i \to B^{d}(0,1)\times B^{N-d}(0,1) $ and compact sets $K_1,...,K_m$ in $\Gamma$ such that

(i) $K_i \sbs O_i$, $1 \leq i \leq m$ and $ \Gamma= \cup_{i=1}^m K_i$,

(ii) $T_i(O_i \cap \Gamma)=B_1^{d}(0) \times \{ x'' = 0_{\mathbb{R}^{N-d}} \}$, $T_i(O_i \cap \Gw)=B_1^{d}(0)\times B_1^{N-d}(0)$,

(iii) For any $x \in O_i \cap (\xO\setminus \Gamma)$, there exists $y \in O_i \cap  \xS$ such that $d_\Gamma(x)=|x-y|$ (here $d_\Gamma(x)$ denotes the distance from $x$ to $\Gamma$). \smallskip

We then define the $\mathrm{Cap}_{\gth,s}^{\Gamma}-$capacity of a compact set $E \sbs \Gamma$ by
\bel{Capsub} \mathrm{Cap}_{\gth,s}^{\Gamma}(E):=\sum_{i=1}^m \mathrm{Cap}_{\CB_{d,\gth},s}^{\mathbb{R}^d}(\tl T_i(E \cap K_i)), \ee
where $T_i(E \cap K_i)=\tl T_i(E \cap K_i) \times  \{ x'' = 0_{\mathbb{R}^{N-d}} \}$. We remark that the definition of the capacities does not depends on $O_i$.

Note that if $\theta s > d$ then
\be \label{CapGamma}
\inf_{z \in \Gamma}\mathrm{Cap}_{\gth,s}^{\Gamma}(\{z\})>C>0.
\ee

By using the above capacities and Proposition \ref{t2.1}, we are able to prove Theorem \ref{subm}.

\begin{proof}[\textbf{Proof of Theorem \ref{subm}}.]
First we note that \eqref{GN2a} holds and
\ba \label{KN2a1}
K_\mu(x,z)\approx d(x)d_{\xS}(x)^{-\am} N_{2\am}(x,z) \quad \forall x \in \Omega \setminus \Sigma, z \in \Sigma.
\ea
By using a similar argument as in the proof of Theorem \ref{theoremint}, together with \eqref{GN2a} and \eqref{KN2a1}, we deduce that equation
$
v=\BBN_{2\am}[v^p (dd_{\xS}^{-\am})^{p+1}]+\ell\BBN_{2\am}[\xn]
$
has a positive solution for $\ell>0$ small if and only if equation \eqref{u-sigmanu} has a positive solution $u$ for $\sigma$ small enough.

Therefore, as in the proof of Theorem \ref{theoremint}, in light of Lemmas \ref{ineq}, \ref{vol} and \ref{vol-2}, we may apply Proposition \ref{t2.1} with $J(x,y)=\CN_{2\am}(x,y)$, $\dd \xo= \left(d(x)d_{\xS}(x)^{-\am}\right)^{p+1}\dx$ and $\lambda=\nu$. Estimate \eqref{Jest} is satisfied thanks to Lemma \ref{ineq}, while assumptions \eqref{2.3}--\eqref{2.4} are fulfilled thanks to Lemmas \ref{vol}--\ref{vol-2} respectively with $b=p+1$ and $\theta=-\am(p+1)$. We note that condition $p<\frac{2+\ap}{\ap}$ ensures that $b$ and $\theta$ satisfy the assumptions in Lemmas \eqref{vol}--\eqref{vol-2}. Therefore, by employing Proposition \ref{t2.1}, we can show that statements 1--3 of Proposition \ref{t2.1} are equivalent to statements 1--3 of the present theorem respectively.

Next we will show that, under assumption \eqref{p-cond-2},  statement 4 of Proposition \ref{t2.1} is equivalent to statement 4 of the present theorem.
More precisely, we show that for any compact subset $E \subset \Sigma$, there hold
\ba \label{Cap-equi-1} \mathrm{Cap}_{\vartheta,p'}^{\xS}(E) \approx \mathrm{Cap}_{\BBN_{2\am},p'}^{p+1,-\am(p+1)}(E),
\ea
where $\vartheta$ is defined in \eqref{gamma}. From \eqref{Capsub}, we see that
\bal
\mathrm{Cap}_{\vartheta,p'}^{\xS}(E):=\sum_{i=1}^m \mathrm{Cap}_{\CB_{k,\vartheta},p'}^{\mathbb{R}^k}(\tl T_i(E \cap K_i)),
\eal
where $T_i(E \cap K_i)=\tl T_i(E \cap K_i) \times  \{ x'' = 0_{\mathbb{R}^{N-k}} \}$. Also,
\bal
\mathrm{Cap}_{\BBN_{2\am},p'}^{p+1,-\am(p+1)}(E)\approx\sum_{i=1}^m\mathrm{Cap}_{\BBN_{2\am},p'}^{p+1,-\am(p+1)}(E\cap K_i).
\eal
Therefore, in order to prove \eqref{Cap-equi-1}, it's enough to show that
\ba \label{Cap-split} \mathrm{Cap}_{\CB_{k,\vartheta},p'}^{\mathbb{R}^k}(\tl T_i(E \cap K_i))\approx \mathrm{Cap}_{\BBN_{2\am},p'}^{p+1,-\am(p+1)}(E \cap K_i), \quad i=1,2,\ldots,m.
\ea

Let $\xl \in \GTM^+(\partial\xO\cup \xS)$ with compact support in $\xS$ be such that $\BBK_\xm[\xl]\in L^p(\xO;\ei)$. Put $\lambda_{K_i} = \1_{K_i}\lambda$. On one hand, from \eqref{eigenfunctionestimates}, \eqref{Martinest1} and since $p< \frac{2+\ap}{\ap} \leq  \frac{N-k-\am}{\am}$, we have
\bal
\int_{O_i}\BBK_\xm[\xl_{K_i}]^p\ei \,\dx\gtrsim \lambda(K_i)^p\int_{O_i}d(x)^{p+1}d_\Sigma(x)^{-(p+1)\am} \,\dx\gtrsim\lambda(K_i)^p.
\eal
On the other hand,
\bal
\int_{\xO\setminus O_i}\BBK_\xm[\lambda_{K_i}]^p\ei \dx\lesssim \lambda(K_i)^p\int_{O_i}d(x)^{p+1}d_\Sigma(x)^{-(p+1)\am}\dx\lesssim \lambda(K_i)^p.
\eal
Combining the above estimate, we obtain
\ba\label{45}
\int_\xO\BBK_\xm[\lambda_{K_i}]^p \ei \,\dx\approx  \int_{O_i}\BBK_\xm[\lambda_{K_i}]^p\ei \,\dx, \quad \forall i=1,2,..,m.
\ea

In view of the proof of \cite[Lemma 5.2.2]{Ad}, there exists a measure $\overline \lambda_i\in \GTM^+(\mathbb{R}^k)$ with compact support in $B^k(0,1)$ such that for any Borel $E\subset B^k(0,1)$, there holds
	\bal
\overline \lambda_i(E)=\lambda(T_i^{-1}(E\times \{0_{\R^{N-k}}\})).
\eal
Set $\psi=(\psi',\psi'')=T_i(x)$ then. By \eqref{propdist}, \eqref{eigenfunctionestimates} and \eqref{Martinest1}, we have
	\bal
	&\ei(x)\approx |\psi''|^{-\am},\\
	 &K_{\mu}(x,y)\approx |\psi''|^{-\am}(|\psi''|+|\psi'-y'|)^{-(N-2\am-2)}, \quad
	 \forall x\in  O_i\setminus \Sigma,\;\forall y\in O_i\cap \Sigma.
	\eal
The above estimates, together with \eqref{45}, imply
\ba \label{46} \BAL
	&\int_{\Omega} \BBK_{\mu}[\lambda_{K_i}]^p\ei \,\dx \approx \int_{ O_i } \BBK_{\mu}[\lambda_{K_i}]^p\ei \,\dx\\
	&\approx \int_{B^k(0,1)}\int_{B^{N-k}(0,1)}|\psi''|^{-(p+1)\am}
	\left(\int_{B^k(0,1)}(|\psi''|+|\psi'-y'|)^{-(N-2\am-2)}\dd\overline{\lambda_i}(y')\right)^p \dd \psi'' \dd \psi'\\
	&=C(N,k)\int_{B^k(0,1)}\int_{0}^{\xb_0}r^{N-k-1-(p+1)\am}
	\left(\int_{B^k(0,1)}(r+|\psi'-y'|)^{-(N-2\am-2)}\dd\overline{\lambda_i}(y')\right)^p \dd r \dd\psi'.\\
    &\approx \int_{\mathbb{R}^k} \BB_{k,\vartheta}[\overline \lambda_i](x')^p\dx'.
	\EAL
	\ea
Here the last estimate is due to \cite[Lemma 8.1]{GkiNg_absorption} (note that \cite[Lemma 8.1]{GkiNg_absorption} holds under assumptions \eqref{p-cond-2}). Combining \eqref{KN2a} and \eqref{46} yields
\bal
\| \BBN_{2\am}[\xl_{K_i}]  \|_{L^p(\Omega;d^{p+1} d_\Sigma^{(p+1)\am})}  \approx \| \BBK_\xm[\xl_{K_i}] \|_{L^p(\Omega;\phi_\mu)}  \approx \| \BBB_{k,\vartheta}[\overline \lambda_i] \|_{L^p(\R^k)}.
\eal
This and \eqref{dualcap} lead to \eqref{Cap-split}, which in turn implies \eqref{Cap-equi-1}. The proof is complete.
\end{proof}

\begin{remark} \label{existSigma}
By \eqref{CapGamma}, if $p<\frac{N-\am}{N-2-\am}$ (equivalently $\vartheta p' > k$) then $\inf_{z \in \Sigma}\mathrm{Cap}_{\vartheta,p'}^{\xS}(\{z\}) >0$. Hence, under the assumption of Theorem \ref{subm}, statement 3 of Theorem \ref{subm} holds and therefore statement 1 also holds true.	
\end{remark}

\begin{remark} \label{nonexistSigma}
Assume $\mu < \left( \frac{N-2}{2} \right)^2$ and $p \geq \frac{N-\am}{N-\am-2}$. Then for any $z \in \Sigma$ and any $\sigma>0$, problem \eqref{u-sigmanu} with $\nu=\delta_z$ does not admit any positive weak solution. Indeed, suppose by contradiction that for some $z \in \Sigma$ and $\sigma>0$, there exists a positive solution $u \in L^p(\Omega;\ei)$ of equation \eqref{u-sigmanu}. Without loss of generality, we can assume that $z =0 \in \Sigma$ and $\sigma=1$. From \eqref{u-sigmanu}, $u(x) \geq \BBK_{\mu}[\delta_0] (x)= K_\mu(x,0)$ for a.e. $x \in \Omega \setminus \Sigma$. Let $\CC$ be a cone of vertex $0$ such that $\CC \subset \Omega \setminus \Sigma$ and there exist $r>0$, $0<\ell<1$ satisfying for any $x \in \CC$, $|x|<r$ and $d_\Sigma(x)> \ell |x|$. Then, by \eqref{Martinest1} and \eqref{eigenfunctionestimates},
\bal
\int_{\Omega \setminus \Sigma} u(x)^p \ei(x) \dd x \geq \int_{\CC} K_\mu(x,0)^p \ei(x) \dd x
 \geq \int_{\CC} |x|^{-\am-(N-\am-2)p} \dd x
 \approx \int_0^r t^{N-1-\am - (N-\am-2)}\dd t.
\eal
Since $p \geq \frac{N-\am}{N-\am-2}$, the last integral is divergent, hence $u \not \in L^p(\Omega \setminus \Sigma;\ei)$, which leads to a contradiction.
\end{remark}

\begin{remark}\label{rem2}
	Assume $\xS=\{0\}$ and $\mu= \left( \frac{N-2}{2} \right)^2$. If $p< \frac{2+\ap}{\ap}$ then there is a solution of \eqref{u-sigmanu} with $\xn=\sigma\xd_0$ for $\sigma>0$ small. Indeed, for any $1<p<\frac{2+\ap}{\ap},$ we have $0<\int_\xO \BBK_\xm[\xd_0]^p \ei \, \dx<\infty$.
	Therefore, by \eqref{dualcap}, we find
	$
\mathrm{Cap}_{\BBN_{2\am},p'}^{p+1,-\am(p+1)}(\{0\})>0.
$
	In view of the proof of Theorem \ref{subm}, we may apply Proposition \ref{t2.1} for $J(x,y)=\CN_{2\am}(x,y),$ for $\dd \xo= \left(d(x)d_{\xS}(x)^{-\am}\right)^{p+1}\dx$ and $\lambda=\xd_0$ to obtain the desired result.
\end{remark}

When $p \geq \frac{2+\ap}{\ap}$, the nonexistence occurs, as shown in the following remark.

\begin{remark} \label{rem3}
If $p\geq \frac{2+\ap}{\ap}$ then, for any measure $\xn \in \GTM^+(\partial \Omega \cup \Sigma)$ with compact support in $\xS$ and any $\sigma>0$, there is no solution of problem \eqref{u-sigmanu}. Indeed, it can be proved by contradiction. Suppose that we can find $\sigma>0$ and a measure $\xn \in \GTM^+(\partial \Omega \cup \Sigma)$ with compact support in $\xS$ such that there exists a solution $0\leq u \in L^p(\xO;\ei)$ of \eqref{u-sigmanu}.
It follows that $\BBK_\xm[\xn]\in L^p(\xO;\ei)$. On one hand, by \cite[Theorem 1.4]{GkiNg_absorption} there is a unique nontrivial nonnegative solution $v$ of
\bal
- L_\gm v + |v|^{p-1}v =0\qquad \text{in }\;\Gw\setminus \Sigma,\quad \tr(v)=\xn.
\eal

Moreover, $v \leq \BBK_{\mu}[\nu]$ in $\Omega \setminus \Sigma$. This, together with Proposition \ref{Martin} and the fact that $\nu$ has compact support in $\Sigma$, implies, for $x$ near $\partial \Omega$,
$ v(x) \leq \BBK_{\mu}[\nu](x)  \lesssim d(x)\nu(\Sigma)$. Therefore, by \cite[Theorem 1.8]{GkiNg_absorption}, we have that $v \equiv 0$, which leads to a contradiction.

\end{remark}

When $\nu$ concentrates on $\partial \Omega$, we also obtain criteria for the existence of problem \eqref{power}. We will treat the case $\mu < \left(\frac{N-2}{2}\right)^2$ and the case $\mu = \left(\frac{N-2}{2}\right)^2$ separably.

\begin{proof}[\textbf{Proof of Theorem \ref{th:existnu-prtO} when $\xm<\left(\frac{N-2}{2}\right)^2$}.]
As in the proof of Theorem \ref{theoremint}, in light of Lemmas \ref{ineq}, \ref{vol} and \ref{vol-2}, we may apply Proposition \ref{t2.1} with $J(x,y)=\CN_{2\am}(x,y)$, $\dd\xo= \left(d(x)d_{\xS}(x)^{-\am}\right)^{p+1}\dx$ and $\lambda=\nu$ in order to show that statements 1--3 of Proposition \ref{t2.1} are equivalent to statements 1--3 of the present theorem respectively.

Next we will show that statement 4 of Proposition \ref{t2.1} is equivalent to statement 4 of the present theorem. More precisely, we will show that for any subset $E \subset \partial \Omega$, there holds
\ba \label{Cap-equi-2} \mathrm{Cap}_{\frac{2}{p},p'}^{\partial \Omega}(E) \approx \mathrm{Cap}_{\BBN_{2\am},p'}^{p+1,-\am(p+1)}(E).
\ea

Indeed, by a similar argument as in the proof of \eqref{45}, under the stated assumptions on $p$, we can show that
for any $\lambda \in \GTM^+(\partial \Omega \cup \Sigma)$ with compact support in $\partial \Omega$, there holds
\bal
\int_\xO\BBK_\xm^p[\lambda]\ei \dx\approx \sum_{i=1}^m \int_{O_i}\BBK_\xm^p[ \1_{K_i} \lambda]\ei \dx.
\eal
This and the estimate
\ba \label{KN2a}
K_\mu(x,z)\approx d(x)d_{\xS}(x)^{-\am} N_{2\am}(x,z) \quad \forall x \in \Omega \setminus \Sigma, z \in \partial \Omega,
\ea
imply
\bal \begin{aligned}
\int_{\Omega} \BBN_{2\am}[\lambda]^p d^{p+1} d_\Sigma^{-\am(p+1)}\dx & \approx \sum_{i=1}^m \int_{O_i} \BBN_{2\am}[\1_{K_i} \lambda]^p d^{p+1} d_\Sigma^{-\am(p+1)} \dx \\
& \approx \sum_{i=1}^m \int_{O_i} \BBN_{2\am}[\1_{K_i} \lambda]^p d^{p+1} \dx.
\end{aligned} \eal
Therefore, in view of the proof of \cite[Proposition 2.9]{BHV} (with $\xa=\xb=2$, $s=p'$ and $\xa_0=p+1$) and \eqref{dualcap}, we obtain \eqref{Cap-equi-2}. The proof is complete.
\end{proof}

%
%
%
%
%
%
%
%
%
%
%
%
%
%
%

\begin{remark} \label{rem4}
If $\am>0$ and $p\geq \frac{2+\am}{\am}$ then for any measure $\xn \in \GTM^+(\partial \Omega \cup \Sigma)$ with compact support in $\partial\xO$ and any $\sigma>0$, there is no solution of \eqref{u-sigmanu}. Indeed, it can be proved by contradiction. Suppose that we can find a measure $\xn \in \GTM^+(\partial \Omega \cup \Sigma)$ with compact support in $\partial \Omega$ and $\sigma>0$ such that there exists a solution $0\leq u \in L^p(\xO;\ei)$ of \eqref{u-sigmanu}. Then by Theorem \ref{th:existnu-prtO}, estimate \eqref{GKp<K} holds for some constant $C>0$.

For simplicity, we assume that $0\in \Sigma$.
Then, for $x$ near $0$, we have
\ba \label{GKp>} \begin{aligned}
\int_{\xO} G_\xm(x,y)\BBK_\xm[\xn](y)^p \dy&\gtrsim d_{\xS}(x)^{-\am}\xn(\partial \Omega)^p \int_{\Sigma_{\xb_0}} |y|^{-(p+1)\am}|x-y|^{-(N-2\am-2)}\dy\\
&\gtrsim \left\{\BAL
&d_\xS(x)^{-\am}|\ln|x|| \quad &&\text{if}\;p=\frac{2+\am}{\am}\\
&d_\xS(x)^{-\am}|x|^{2+\am-p\am} \quad &&\text{if}\;p>\frac{2+\am}{\am}.
\EAL\right.
\end{aligned} \ea
From  \eqref{GKp<K} and \eqref{GKp>}, we can reach at a contradiction by letting  $|x|\to0$.
\end{remark}

\subsection{The case $\xS=\{0\}$ and $\xm= H^2$} In this subsection we treat the case $\xS=\{0\}$ and $\xm= H^2$. Let us introduce some notations.
Let $0<\xe<N-2$, put
\bal
\CN_{1,\xe}(x,y):=\frac{\max\{|x-y|,|x|,|y|\}^{N-2}+|x-y|^{N-2-\xe}}{|x-y|^{N-2}\max\{|x-y|,d(x),d(y)\}^2},\quad\forall(x,y)\in\overline{\xO}\times\overline{\xO}, x \neq y,	
\eal
\bal
\CN_{N-2-\xe}(x,y):=\frac{\max\{|x-y|,|x|,|y|\}^{N-2-\xe}}{|x-y|^{N-2}\max\{|x-y|,d(x),d(y)\}^2},\quad\forall(x,y)\in\overline{\xO}\times\overline{\xO}, x \neq y,
\eal
\bal \begin{aligned}
G_{H^2,\xe}(x,y) &:= |x-y|^{2-N} \left(1 \wedge \frac{d(x)d(y)}{|x-y|^2}\right) \left(1 \wedge \frac{|x||y|}{|x-y|^2} \right)^{-\frac{N-2}{2}} \\
&+(|x||y|)^{-\frac{N-2}{2}}|x-y|^{-\xe}  \left(1 \wedge \frac{d(x)d(y)}{|x-y|^2}\right), \quad x,y \in \Omega \setminus \{0\}, \, x \neq y,
\end{aligned} \eal
\ba \label{tilG}
\tilde G_{H^2,\xe}(x,y):=d(x)d(y)(|x||y|)^{-\frac{N-2}{2}} \CN_{N-2-\xe}(x,y), \quad \forall x,y \in \Omega \setminus \{0\}, \, x \neq y.
\ea

Note that
\bal \begin{aligned}
(|x||y|)^{-\frac{N-2}{2}}\left|\ln\left(1 \wedge \frac{|x-y|^2}{d(x)d(y)}\right)\right|
&\leq(|x||y|)^{-\frac{N-2}{2}}\left|\ln\frac{|x-y|}{\mathcal{D}_\Omega}\right|\left(1 \wedge \frac{d(x)d(y)}{|x-y|^2}\right) \\
&\leq C(\Omega,\xe) (|x||y|)^{-\frac{N-2}{2}}|x-y|^{-\xe}  \left(1 \wedge \frac{d(x)d(y)}{|x-y|^2}\right),
\end{aligned}
\eal
which together with \eqref{Greenestb}, implies
\ba \label{GGe}
G_{H^2}(x,y)\lesssim G_{H^2,\xe}(x,y), \quad \forall x,y \in \Omega \setminus \{0\}, \, x \neq y.
\ea

Next, from the estimates
\bal
G_{H^2,\varepsilon}(x,y) \approx d(x)d(y)(|x||y|)^{-\frac{N-2}{2}} \CN_{1,\varepsilon}(x,y), \quad x,y \in \Omega \setminus \{0\}, \, x \neq y,
\eal
\bal
\CN_{1,\xe}(x,y)\leq C(\xe,\xO) \CN_{N-2-\xe}(x,y),\quad x,y \in \Omega \setminus \{0\}, \, x \neq y,
\eal
we obtain
\ba \label{GGe1}
G_{H^2,\xe}(x,y)\lesssim \tilde G_{H^2,\xe}(x,y), \quad \forall x,y \in \Omega \setminus \{0\}, \, x \neq y.
\ea

Set
\bal
\tilde\BBG_{H^2,\xe}[\tau](x):=\int_{\xO\setminus \xS} \tilde G_{H^2,\xe}(x,y) \dd\tau(y), \\ \BBN_{N-2-\varepsilon}[\tau](x):=\int_{\xO\setminus \xS} \CN_{N-2-\varepsilon}(x,y) \dd\tau(y).
\eal
Proceeding as in the proof of Theorem \ref{theoremint}, we obtain the following result

\begin{theorem}\label{theoremint2}
 Let $0<\xe<\min\{N-2,2\},$ $1<p<\frac{N+2-2\xe}{N-2}$ and $\tau \in \GTM^+(\Omega \setminus \{0\}; \phi_{H^2})$. Then the following statements are equivalent.
	
	1. The equation
	\be \label{eq:uH2e} u=\tilde\BBG_{H^2,\xe}[u^p]+\rho\tilde\BBG_{H^2,\xe}[\gt]
	\ee
	has a positive solution for $\rho>0$ small.
	
	2. For any Borel set $E \subset \xO\setminus \{0\}$, there holds
	\bal
\int_E \tilde\BBG_{H^2,\xe}[\1_E\gt]^p \phi_{H^2}\, \dx \leq C\int_E\phi_{H^2}\dd\gt.
\eal
	
	3. The following inequality holds
	\bal
\tilde\BBG_{H^2,\xe}[\tilde\BBG_{H^2,\xe}[\gt]^p]\leq C\,\tilde\BBG_{H^2,\xe}[\gt]<\infty\quad a.e.
\eal

	4. For any Borel set $E \subset \xO\setminus \{0\}$ there holds
\bal
\int_E\phi_{H^2}\dd\gt\leq C\, \mathrm{Cap}_{\BBN_{N-2-\xe},p'}^{1,-\frac{N-2}{2}(p+1)}(E).
\eal
\end{theorem}

\begin{theorem}\label{singl}
We assume that at least one of the statements 1--4 of Theorem \ref{theoremint2} is valid. Then the equation
\ba \label{eq:uH2}
u=\BBG_{H^2}[u^p]+\rho\BBG_{H^2}[\gt]
\ea
has a positive solution for $\rho>0$ small.
\end{theorem}
\begin{proof}
 From the assumption, by Theorem \ref{theoremint2}, there exists a solution $u$ equation \eqref{eq:uH2e} for $\rho>0$ small.  By \eqref{GGe} and \eqref{GGe1}, we have
 $u\gtrsim \BBG_{H^2}[u^p]+\rho\BBG_{H^2}[\gt]$.
By \cite[Proposition 2.7]{BHV}, we deduce that equation \eqref{eq:uH2} has a solution for $\rho>0$ small.
\end{proof}

\begin{theorem}\label{theeqint}
Assume $\xS=\{0\},$ $\mu=\left( \frac{N-2}{2}\right)^2$ and $\tau \in \GTM^+(\Omega \setminus \{0\}; \phi_{\mu})$ has compact support in $\xO\setminus \{0\}$.
Then Theorem \ref{theoremint} is valid.	
\end{theorem}
\begin{proof}
Let $\xe>0$ be small enough such that $1<p<\frac{N+2-2\xe}{N-2}.$ Let $K=\supp(\tau)\Subset \xO\setminus \{0\}$ and $\tilde \xb=\frac{1}{2}\dist(K,\partial\xO\cup \{0\})>0$. By \eqref{Greenestb}, \eqref{GGe} and \eqref{GGe1}, we can show that
\ba\label{62}
\BBG_{H^2}[\1_E \tau] \approx \BBG_{H^2,\xe}[\1_E \tau] \approx \tilde \BBG_{H^2,\xe}[\1_E \tau] \quad\text{and}\quad \BBN_{N-2}[\tilde\tau]\approx \BBN_{N-2-\xe}[\tilde\tau] \quad \text{in } \Omega \setminus \{0\},
\ea
for all Borel $E\subset \xO\setminus \{0\}$ and $\tilde\tau \in \GTM^+(\Omega \setminus \{0\}; \phi_{H^2})$ with $\supp(\tilde\tau)\subset K$. The implicit constants in the above estimates depend only on $N,\Omega,\tilde \beta, \varepsilon$. Hence, statements 2,4 of Theorem \ref{theoremint2} are equivalent with respective statements 2,4 (with $\am=\frac{N-2}{2}$) of Theorem \ref{theoremint}.

By Proposition \ref{equivint}, it is enough to show that statement 3 of Theorem \ref{theoremint2} is equivalent with statement 3 of Theorem \ref{theoremint}. By \eqref{62}, it is enough to prove that
\ba\label{63}
\tilde\BBG_{H^2,\xe}[\tilde\BBG_{H^2,\xe}[\gt]^p]\approx \BBG_{H^2}[\BBG_{H^2}[\gt]^p] \quad \text{in } \Omega \setminus \{0\}.
\ea
By \eqref{GGe} and \eqref{GGe1}, it is sufficient to show that
 \ba\label{64}
\tilde\BBG_{H^2,\xe}[\tilde\BBG_{H^2,\xe}[\gt]^p]\lesssim \BBG_{H^2}[\BBG_{H^2}[\gt]^p] \quad  \text{in } \Omega \setminus \{0\}.
\ea

Indeed, on one hand, since $1<p<\frac{N+2-2\xe}{N-2}$, we have, for any $x\in \xO\setminus \{0\}$,
\ba\label{65}\BAL
&\int_{B(0,\frac{\tilde \xb}{4})}\tilde G_{H^2,\xe}(x,y)\tilde\BBG_{H^2,\xe}[\gt](y)^p \dd y\approx \tau(K)^p \int_{B(0,\frac{\tilde \xb}{4})} \tilde G_{H^2,\xe}(x,y)|y|^{-\frac{p(N-2)}{2}} \dd y\\
&\lesssim  \tau(K)^p d(x)|x|^{-\frac{N-2}{2}}\int_{B(0,\frac{\tilde \xb}{4})} |x-y|^{-\xe}|y|^{-\frac{(p+1)(N-2)}{2}} \dd y\\
& \quad +\tau(K)^p d(x)|x|^{-\frac{N-2}{2}}\int_{B(0,\frac{\tilde \xb}{4})} |x-y|^{-N+2}|y|^{N-2-\xe-\frac{(p+1)(N-2)}{2}} \dd y \\
& \lesssim\tau(K)^p d(x)|x|^{-\frac{N-2}{2}}.
\EAL
\ea
The implicit constants in the above inequalities depend only on $\xO,K,\tilde \beta,p,\xe$.

On the other hand, we have
\ba\label{66}\BAL
\int_{B(0,\frac{\tilde \xb}{4})} G_{H^2}(x,y)\BBG_{H^2}[\gt](y)^p \dd y&\gtrsim\tau(K)^p d(x)|x|^{-\frac{N-2}{2}}\int_{B(0,\frac{\tilde \xb}{4})} |y|^{-\frac{(p+1)(N-2)}{2}} \dd y \\
&\gtrsim\tau(K)^p d(x)|x|^{-\frac{N-2}{2}},
\EAL
\ea
where the implicit constants in the above inequalities depend only on $\xO,K,\tilde \beta,p$. Hence by \eqref{65} and \eqref{66}, we have that
\ba\label{67}
\int_{B(0,\frac{\tilde \xb}{4})} \tilde G_{H^2,\xe}(x,y)\tilde\BBG_{H^2,\xe}[\gt](y)^p \dd y\lesssim\int_{B(0,\frac{\tilde \xb}{4})} G_{H^2}(x,y)\BBG_{H^2}[\gt](y)^p \dd y \quad\forall x\in \xO\setminus \{0\}.
\ea

Next, by \eqref{Greenestb} and \eqref{tilG}, we have, for $x \in \Omega \setminus \{0\}$ and $y \in \Omega \setminus B(0,\frac{\tilde \beta}{4})$,
\bal
\tilde G_{H^2,\xe}(x,y) \approx d(x)d(y)(|x||y|)^{-\frac{N-2}{2}} \CN_{N-2}(x,y) \lesssim G_{H^2}(x,y).
\eal
This and \eqref{62} yield
\ba\label{68}
\int_{\xO\setminus B(0,\frac{\tilde \xb}{4})} \tilde G_{H^2,\xe}(x,y)\tilde\BBG_{H^2,\xe}[\gt](y)^p \, \dd y\lesssim\int_{\xO\setminus B(0,\frac{\tilde \xb}{4})} G_{H^2}(x,y)\BBG_{H^2}[\gt](y)^p \, \dd y \quad\forall x\in \xO\setminus \{0\}.
\ea

Combining \eqref{67} and \eqref{68}, we deduce \eqref{64}. The proof is complete.
\end{proof}

\begin{proof}[\textbf{Proof of Theorem \ref{th:existnu-prtO} when $\xS=\{0\}$ and $\xm=\frac{(N-2)^2}{4}$}.] Proceeding as in the proof of Theorem \ref{theeqint}, we obtain the desired result.
\end{proof}

\begin{remark} \label{partO-N}
	If $p<\frac{N+1}{N-1}$, by using a \eqref{CapGamma}, we obtain that $\inf_{z \in \partial \Omega}\mathrm{Cap}_{\frac{2}{p},p'}^{\partial \Omega}(\{z\})>C>0$, hence statement 3 of Theorem \ref{th:existnu-prtO} holds true. Consequently, under the assumptions of Theorem \ref{th:existnu-prtO}, equation \eqref{u-sigmanu} has a positive solution for $\gs>0$ small. When $p \geq \frac{N+1}{N-1}$, by using a similar argument as in Remark \ref{nonexistSigma}, we can show that for any $\sigma>0$ and $z \in \partial \Omega$,  equation \eqref{u-sigmanu} does not admit any positive weak solution.
\end{remark}

\appendix\section{Some estimates} \label{app:A}
\setcounter{equation}{0}
In this appendix, we give an estimate which is used several times in the paper.
\begin{lemma} \label{lemapp:1}
Assume $\ell_1>0$, $\ell_2>0$, $\alpha_1$ and $\alpha_2$ such that $N-k+\alpha_1 + k\alpha_2 >0$. For $y \in \Omega \setminus \Sigma$, put
$
 \CA(y):= \{ x \in (\Omega \setminus \Sigma): d_{\Sigma}(x) \leq \ell_1 \quad \text{and} \quad |x-y| \leq \ell_2 d_{\Sigma}(x)^{\alpha_2} \}.
$
Then
\bal
\int_{\CA(y) \cap \Sigma_{\beta_1}} d_{\Sigma}(x)^{\alpha_1}\dx \lesssim \ell_1^{N-k+\alpha_1 + k\alpha_2}\ell_2^k.
\eal
\end{lemma}
\begin{proof}
By \eqref{cover}, we have
\bal
\int_{\CA(y) \cap \Sigma_{\beta_1}} d_{\Sigma}(x)^{\alpha_1}\dx \leq \sum_{j=1}^{m_0}\int_{\CA(y) \cap V(\xi^j,\beta_0)} d_{\Sigma}(x)^{\alpha_1}\dx.
\eal	
For any $j \in \{1,...,m_0\}$, in view of \eqref{propdist}, we have
\ba \label{app:3}
d_\Sigma(x) \leq \delta_\Sigma^{\xi^j}(x) \leq C \| \Sigma\|_{C^2} d_\Sigma(x) \quad \forall x \in V(\xi^j,\beta_0),
\ea
where
\bal
\delta_\Sigma^{\xi^j}(x):=\sqrt{\sum_{i=k+1}^N|x_i-\Gamma_i^{\xi^j}(x')|^2}, \qquad x=(x',x'')\in V(\xi^j,\beta_0).
\eal
Therefore, by the change of variables $z'=x'-(\xi^j)'$ and $z''=(z_{k+1},\ldots,z_N)$ with $z_i=x_i-\Gamma_i^{\xi^j}(x')$, $i=k+1,..,N$, and \eqref{app:3}, we have
\bal \begin{aligned}
\int_{ \CA(y) \cap V(\xi^j,\beta_0) }d_\Sigma(x)^{\alpha_1}\dx &\lesssim \int_{ \{\delta_\Sigma^{\xi_j}(x) \leq c\ell_1, |x-y| \leq c\ell_2 \delta_\Sigma^{\xi_j}(x)^{\alpha_2}  \} \cap V(\xi^j,\beta_1) }\delta_\Sigma^{\xi_j}(x)^{\alpha_1}\dx \\
&\lesssim   \int_{ \{ |z''| \leq c\ell_1 \} } \int_{ \{ |z'|< c\ell_2 |z''|^{\alpha_2} \} } |z''|^{\alpha_1} \dz' \dz'' \lesssim \ell_1^{N-k+\alpha_1 + k\alpha_2} \ell_2^k.
\end{aligned}
 \eal
The last estimate holds because $N-k+\alpha_1 + k\alpha_2>0$. The proof is complete.
\end{proof}


\end{document}